\documentclass[12pt]{amsart}
\usepackage{mathrsfs}
\usepackage{cases}
\usepackage{epic}
\usepackage{amsfonts}
\usepackage{graphicx}
\usepackage{amsmath}
\usepackage{amssymb, upgreek}
\usepackage{bm}
\usepackage{latexsym,todonotes}
\usepackage{pdflscape}
\usepackage[all]{xypic}
\usepackage[all]{xy}
\usepackage{color}
\usepackage{colordvi}
\usepackage{multicol}
\usepackage[normalem]{ulem}
\input diagxy
\usepackage[linktocpage=true]{hyperref}
\hypersetup{colorlinks,linkcolor=blue,urlcolor=cyan,citecolor=blue}

\begin{document}
\input xy
\xyoption{all}

\newtheorem{innercustomthm}{{\bf Theorem}}
\newenvironment{customthm}[1]
  {\renewcommand\theinnercustomthm{#1}\innercustomthm}
  {\endinnercustomthm}

  \newtheorem{innercustomcor}{{\bf Corollary}}
\newenvironment{customcor}[1]
  {\renewcommand\theinnercustomcor{#1}\innercustomcor}
  {\endinnercustomthm}

  \newtheorem{innercustomprop}{{\bf Proposition}}
\newenvironment{customprop}[1]
  {\renewcommand\theinnercustomprop{#1}\innercustomprop}
  {\endinnercustomthm}

\newcommand{\iadd}{\operatorname{iadd}\nolimits}

\renewcommand{\mod}{\operatorname{mod}\nolimits}
\newcommand{\proj}{\operatorname{proj}\nolimits}
\newcommand{\inj}{\operatorname{inj}\nolimits}
\newcommand{\rad}{\operatorname{rad}\nolimits}
\newcommand{\Span}{\operatorname{Span}\nolimits}
\newcommand{\soc}{\operatorname{soc}\nolimits}
\newcommand{\ind}{\operatorname{inj.dim}\nolimits}
\newcommand{\Ginj}{\operatorname{Ginj}\nolimits}
\newcommand{\res}{\operatorname{res}\nolimits}
\newcommand{\np}{\operatorname{np}\nolimits}
\newcommand{\Fac}{\operatorname{Fac}\nolimits}
\newcommand{\Aut}{\operatorname{Aut}\nolimits}
\def\DTr{D{\rm Tr}}
\def\TrD{{\rm Tr}D}
\newcommand{\Gr}{\operatorname{Gr}\nolimits}

\newcommand{\Mod}{\operatorname{Mod}\nolimits}
\newcommand{\R}{\operatorname{R}\nolimits}
\newcommand{\End}{\operatorname{End}\nolimits}
\newcommand{\lf}{\operatorname{l.f.}\nolimits}
\newcommand{\Iso}{\operatorname{Iso}\nolimits}
\newcommand{\aut}{\operatorname{Aut}\nolimits}
\newcommand{\Ui}{{\mathbf U}^\imath}
\newcommand{\UU}{{\mathbf U}\otimes {\mathbf U}}
\newcommand{\UUi}{(\UU)^\imath}
\newcommand{\tUU}{{\tU}\otimes {\tU}}
\newcommand{\tUUi}{(\tUU)^\imath}
\newcommand{\tUi}{\widetilde{{\mathbf U}}^\imath}
\newcommand{\sqq}{{\bf v}}
\newcommand{\sqvs}{\sqrt{\vs}}
\newcommand{\dbl}{\operatorname{dbl}\nolimits}
\newcommand{\swa}{\operatorname{swap}\nolimits}
\newcommand{\Gp}{\operatorname{Gp}\nolimits}

\newcommand{\iLa}{\Lambda^{\imath}}
\newcommand{\U}{{\mathbf U}}
\newcommand{\tU}{\widetilde{\mathbf U}}
\newcommand{\bvs}{\mathbf{\varsigma}}

\newcommand{\vs}{\varsigma}
\newcommand{\ov}{\overline}
\newcommand{\tk}{\widetilde{k}}
\newcommand{\tK}{\widetilde{K}}

\newcommand{\tH}{\widetilde{\ch}}

\newcommand{\utM}{\operatorname{\cs\cd\ch}\nolimits}
\newcommand{\tM}{\operatorname{\cs\cd\widetilde{\ch}}\nolimits}
\newcommand{\rM}{\operatorname{\cs\cd\ch_{\rm{red}}}\nolimits}
\newcommand{\utMH}{\operatorname{\cs\cd\ch(\iLa)}\nolimits}
\newcommand{\tMH}{\operatorname{\cs\cd\widetilde{\ch}(\iLa)}\nolimits}
\newcommand{\rMH}{\operatorname{\cs\cd\ch_{\rm{red}}(\iLa)}\nolimits}
\newcommand{\utMHg}{\operatorname{\ch(Q,\btau)}\nolimits}
\newcommand{\tMHg}{\operatorname{\widetilde{\ch}(Q,\btau)}\nolimits}
\newcommand{\rMHg}{\operatorname{\ch_{\rm{red}}(Q,\btau)}\nolimits}

\newcommand{\rMHd}{\operatorname{\cs\cd\ch_{\rm{red}}(\iLa)_{\bvsd}}\nolimits}
\newcommand{\tMHd}{\operatorname{\cs\cd\widetilde{\ch}(\iLa)_{\bvsd}}\nolimits}

\newcommand{\tMHl}{\cs\cd\widetilde{\ch}({\bs}_\ell\iLa)}
\newcommand{\rMHl}{\cs\cd\ch_{\rm{red}}({\bs}_\ell\iLa)_{\bvsd}}
\newcommand{\tMHi}{\cs\cd\widetilde{\ch}({\bs}_i\iLa)}
\newcommand{\rMHi}{\cs\cd\ch_{\rm{red}}({\bs}_i\iLa)_{\bvsd}}
\newcommand{\tMHgl}{\widetilde{\ch}({\bs}_\ell Q,\btau)}
\newcommand{\tMHgi}{\widetilde{\ch}({\bs}_i Q,\btau)}

\newcommand{\utGpg}{\operatorname{\ch^{\rm Gp}(Q,\btau)}\nolimits}
\newcommand{\tGpg}{\operatorname{\widetilde{\ch}^{\rm Gp}(Q,\btau)}\nolimits}
\newcommand{\rGpg}{\operatorname{\ch_{red}^{\rm Gp}(Q,\btau)}\nolimits}

\newcommand{\colim}{\operatorname{colim}\nolimits}
\newcommand{\gldim}{\operatorname{gl.dim}\nolimits}
\newcommand{\cone}{\operatorname{cone}\nolimits}
\newcommand{\rep}{\operatorname{rep}\nolimits}
\newcommand{\Ext}{\operatorname{Ext}\nolimits}
\newcommand{\Tor}{\operatorname{Tor}\nolimits}
\newcommand{\Hom}{\operatorname{Hom}\nolimits}
\newcommand{\Top}{\operatorname{top}\nolimits}
\newcommand{\Coker}{\operatorname{Coker}\nolimits}
\newcommand{\thick}{\operatorname{thick}\nolimits}
\newcommand{\rank}{\operatorname{rank}\nolimits}
\newcommand{\Gproj}{\operatorname{Gproj}\nolimits}
\newcommand{\Len}{\operatorname{Length}\nolimits}
\newcommand{\RHom}{\operatorname{RHom}\nolimits}
\renewcommand{\deg}{\operatorname{deg}\nolimits}
\renewcommand{\Im}{\operatorname{Im}\nolimits}
\newcommand{\Ker}{\operatorname{Ker}\nolimits}
\newcommand{\Coh}{\operatorname{Coh}\nolimits}
\newcommand{\Id}{\operatorname{Id}\nolimits}
\newcommand{\Qcoh}{\operatorname{Qch}\nolimits}
\newcommand{\CM}{\operatorname{CM}\nolimits}
\newcommand{\sgn}{\operatorname{sgn}\nolimits}
\newcommand{\Gdim}{\operatorname{G.dim}\nolimits}
\newcommand{\fpr}{\operatorname{\mathcal{P}^{\leq1}}\nolimits}

\newcommand{\For}{\operatorname{{\bf F}or}\nolimits}
\newcommand{\coker}{\operatorname{Coker}\nolimits}
\renewcommand{\dim}{\operatorname{dim}\nolimits}
\newcommand{\rankv}{\operatorname{\underline{rank}}\nolimits}
\newcommand{\dimv}{{\operatorname{\underline{dim}}\nolimits}}
\newcommand{\diag}{{\operatorname{diag}\nolimits}}
\newcommand{\qbinom}[2]{\begin{bmatrix} #1\\#2 \end{bmatrix} }

\renewcommand{\Vec}{{\operatorname{Vec}\nolimits}}
\newcommand{\pd}{\operatorname{proj.dim}\nolimits}
\newcommand{\gr}{\operatorname{gr}\nolimits}
\newcommand{\id}{\operatorname{Id}\nolimits}
\newcommand{\Res}{\operatorname{Res}\nolimits}

\newcommand{\pdim}{\operatorname{proj.dim}\nolimits}
\newcommand{\idim}{\operatorname{inj.dim}\nolimits}
\newcommand{\Gd}{\operatorname{G.dim}\nolimits}
\newcommand{\Ind}{\operatorname{Ind}\nolimits}
\newcommand{\add}{\operatorname{add}\nolimits}
\newcommand{\pr}{\operatorname{pr}\nolimits}
\newcommand{\oR}{\operatorname{R}\nolimits}
\newcommand{\oL}{\operatorname{L}\nolimits}
\newcommand{\ext}{{ \mathfrak{Ext}}}
\newcommand{\Perf}{{\mathfrak Perf}}
\def\scrP{\mathscr{P}}
\newcommand{\bk}{{\mathbb K}}
\newcommand{\cc}{{\mathcal C}}
\newcommand{\gc}{{\mathcal GC}}
\newcommand{\dg}{{\rm dg}}
\newcommand{\ce}{{\mathcal E}}
\newcommand{\cs}{{\mathcal S}}
\newcommand{\cl}{{\mathcal L}}
\newcommand{\cf}{{\mathcal F}}
\newcommand{\cx}{{\mathcal X}}
\newcommand{\cy}{{\mathcal Y}}
\newcommand{\ct}{{\mathcal T}}
\newcommand{\cu}{{\mathcal U}}
\newcommand{\cv}{{\mathcal V}}
\newcommand{\cn}{{\mathcal N}}
\newcommand{\mcr}{{\mathcal R}}
\newcommand{\ch}{{\mathcal H}}
\newcommand{\ca}{{\mathcal A}}
\newcommand{\cb}{{\mathcal B}}
\newcommand{\ci}{{\I}_{\btau}}
\newcommand{\cj}{{\mathcal J}}
\newcommand{\cm}{{\mathcal M}}
\newcommand{\cp}{{\mathcal P}}
\newcommand{\cg}{{\mathcal G}}
\newcommand{\cw}{{\mathcal W}}
\newcommand{\co}{{\mathcal O}}
\newcommand{\cq}{{Q^{\rm dbl}}}
\newcommand{\cd}{{\mathcal D}}
\newcommand{\cz}{{\mathcal Z}}
\newcommand{\ck}{\widetilde{\mathcal K}}
\newcommand{\calr}{{\mathcal R}}
\newcommand{\La}{\Lambda}
\newcommand{\LL}{\texttt{L}}
\newcommand{\RR}{\texttt{R}}
\newcommand{\ol}{\overline}
\newcommand{\st}{[1]}
\newcommand{\ow}{\widetilde}
\renewcommand{\P}{\mathbf{P}}
\newcommand{\pic}{\operatorname{Pic}\nolimits}
\newcommand{\Spec}{\operatorname{Spec}\nolimits}

\newtheorem{theorem}{Theorem}[section]
\newtheorem{acknowledgement}[theorem]{Acknowledgement}
\newtheorem{algorithm}[theorem]{Algorithm}
\newtheorem{axiom}[theorem]{Axiom}
\newtheorem{case}[theorem]{Case}
\newtheorem{claim}[theorem]{Claim}
\newtheorem{conclusion}[theorem]{Conclusion}
\newtheorem{condition}[theorem]{Condition}
\newtheorem{conjecture}[theorem]{Conjecture}
\newtheorem{construction}[theorem]{Construction}
\newtheorem{corollary}[theorem]{Corollary}
\newtheorem{criterion}[theorem]{Criterion}
\newtheorem{definition}[theorem]{Definition}
\newtheorem{example}[theorem]{Example}
\newtheorem{exercise}[theorem]{Exercise}
\newtheorem{lemma}[theorem]{Lemma}
\newtheorem{notation}[theorem]{Notation}
\newtheorem{problem}[theorem]{Problem}
\newtheorem{proposition}[theorem]{Proposition}
\newtheorem{solution}[theorem]{Solution}
\newtheorem{summary}[theorem]{Summary}
\numberwithin{equation}{section}
\allowdisplaybreaks

\theoremstyle{remark}
\newtheorem{remark}[theorem]{Remark}
\newcommand{\Pd}{\pi_*}
\def \bvs{{\boldsymbol{\varsigma}}}
\def \bvsd{{\boldsymbol{\varsigma}_{\diamond}}}
\def \btau{{{\uptau}}}

\def \bp{{\mathbf p}}
\def \bq{{\bm q}}
\def \bv{{v}}
\def \bs{{\bm s}}
\newcommand{\bsigma}{\bm{\sigma}}

\newcommand{\bfv}{\mathbf{v}}
\def \brW{{\rm Br}(W_{\btau})}

\def \bA{{\mathbf A}}
\def \ba{{\mathbf a}}
\def \bb{{\mathbf b}}
\def \bL{{\mathbf L}}
\def \bF{{\mathbf F}}
\def \bS{{\mathbf S}}
\def \bC{{\mathbf C}}
\def \bU{{\mathbf U}}
\def \bc{{\mathbf c}}
\def \fpi{\mathfrak{P}^\imath}
\def \Ni{N_\imath}
\def \fp{\mathfrak{P}}
\def \fg{\mathfrak{g}}
\def \fb{\mathfrak{b}}
\def \fk{\fg^\theta}  
\def \p{p}
\def \fn{\mathfrak{n}}
\def \fh{\mathfrak{h}}
\def \fu{\mathfrak{u}}
\def \fv{\mathfrak{v}}
\def \fa{\mathfrak{a}}
\def \Z{{\Bbb Z}}
\def \F{{\Bbb F}}
\def \D{{\Bbb D}}
\def \C{{\Bbb C}}
\def \N{{\Bbb N}}
\def \Q{{\Bbb Q}}
\def \G{{\Bbb G}}
\def \P{{\Bbb P}}
\def \K{{k}}
\def \E{{\Bbb E}}
\def \A{{\Bbb A}}
\def \L{{\Bbb L}}
\def \I{{\Bbb I}}
\def \BH{{\Bbb H}}
\def \T{{\mathcal T}}
\def \tT{\widetilde{\mathcal T}}
\def \tTL{\tT(\iLa)}
\newcommand{\TT}{\operatorname{\texttt{\rm T}}\nolimits}
\newcommand{\tTT}{\operatorname{\texttt{\rm T}}\nolimits}

\newcommand {\lu}[1]{\textcolor{red}{$\clubsuit$: #1}}

\newcommand{\browntext}[1]{\textcolor{brown}{#1}}
\newcommand{\greentext}[1]{\textcolor{green}{#1}}
\newcommand{\redtext}[1]{\textcolor{red}{#1}}
\newcommand{\bluetext}[1]{\textcolor{blue}{#1}}
\newcommand{\brown}[1]{\browntext{ #1}}
\newcommand{\green}[1]{\greentext{ #1}}
\newcommand{\red}[1]{\redtext{ #1}}
\newcommand{\blue}[1]{\bluetext{ #1}}

\newcommand{\wtodo}{\todo[inline,color=orange!20, caption={}]}
\newcommand{\lutodo}{\todo[inline,color=green!20, caption={}]}

\title[Hall algebras and quantum symmetric pairs II]{Hall algebras and quantum symmetric pairs II: reflection functors}

\author[Ming Lu]{Ming Lu}
\address{Department of Mathematics, Sichuan University, Chengdu 610064, P.R.China}
\email{luming@scu.edu.cn}

\author[Weiqiang Wang]{Weiqiang Wang}
\address{Department of Mathematics, University of Virginia, Charlottesville, VA 22904, USA}
\email{ww9c@virginia.edu}


\begin{abstract}
Recently the authors initiated an $\imath$Hall algebra approach to (universal) $\imath$quantum groups arising from quantum symmetric pairs. In this paper we construct and study BGP type reflection functors which lead to isomorphisms of the $\imath$Hall algebras associated to acyclic $\imath$quivers.
For Dynkin quivers, these symmetries on $\imath$Hall algebras induce automorphisms of universal $\imath$quantum groups, which are shown to satisfy the braid group relations associated to the restricted Weyl group of a symmetric pair. This leads to a conceptual construction of $q$-root vectors and PBW bases for (universal) quasi-split $\imath$quantum groups of ADE type.
\end{abstract}

\maketitle
 \tableofcontents

\section{Introduction}

\subsection{Background}
\subsubsection{}

Lusztig \cite{Lus90a} \cite[Part VI]{Lus93} constructed various automorphisms on quantum groups $\U$ which satisfy the braid group relations, and used them to define PBW and canonical bases for quantum groups of finite type. The braid group actions have played fundamental roles in diverse areas such as geometric representation theory, categorification, and quantum topology.

Ringel's Hall algebra realization of halves of quantum groups $\U^+$ \cite{Rin90} (cf. Green \cite{Gr95} for extension to acyclic quivers) has led to fruitful ways of understanding the quantum groups. The automorphisms of $\U$ by Lusztig can also be understood from a categorical viewpoint via the BGP reflection functors (which were originally introduced in \cite{BGP73} to provide a conceptual explanation of Gabriel's theorem). The reflection functors were adapted for the Hall algebra setting by Ringel  \cite{Rin96} to provide isomorphisms of subalgebras of $\U^+$. Later the reflection functors were extended to automorphisms of the reduced Drinfeld doubles of Hall algebras and quantum groups $\U$, and they were identified with Lusztig's braid group actions on $\U$; see \cite{SV99, XY01}.

A conceptual approach of reinterpreting the BGP reflection functors in terms of tilting modules was developed by \cite{APR79}.
%
%
%
\subsubsection{}

Generalizing the Ringel-Green construction, Bridgeland \cite{Br13} realized  quantum groups $\U$ via Hall algebras of $\Z/2\Z$-graded complexes. His construction has led to several ramifications, such as Gorsky's semi-derived Hall algebras \cite{Gor13} and a more general version of semi-derived (Ringel-)Hall algebras by the first author and Peng \cite{LP}; the construction of semi-derived Hall algebras was further extended to 1-Gorenstein algebras by the first author in \cite{Lu19}.

Recently the authors \cite{LW19} studied in depth the $\imath$Hall algebras associated to $\imath$quivers (aka quivers with involutions) in the framework of semi-derived Hall algebras of 1-Gorenstein algebras. Starting from an acyclic $\imath$quiver $(Q,\btau)$, i.e., a quiver $Q$ with an involution $\btau$ (which is allowed to be trivial), we construct a family of finite-dimensional 1-Gorenstein algebras called {\em $\imath$quiver algebras}. In case of split $\imath$quivers (i.e., $\btau =\Id$), the $\imath$quiver algebras can be identified with $kQ \otimes k [\varepsilon]/ (\varepsilon^2)$, which were studied in \cite{RZ17, GLS17} from very different perspectives. The $\imath$Hall algebras are by definition the twisted semi-derived Hall algebras for the $\imath$quiver algebras.

The notion of $\imath$quantum groups $\Ui$ plays a key role in the definition of quantum symmetric pairs $(\U, \Ui)$ associated to Satake diagrams \cite{Let99, Let02} (also see \cite{Ko14} for extension to the Kac-Moody setting). We view $\imath$quantum groups as a natural generalization of quantum groups. The $\imath$-program (as proposed in \cite{BW18}) aims at providing natural generalizations of various fundamental (algebraic, geometric, and categorical) constructions for quantum groups to $\imath$quantum groups.

The $\imath$Hall algebras associated to  Dynkin $\imath$quivers $(Q,\btau)$ were shown by the authors in \cite{LW19} to provide a realization of the quasi-split $\imath$quantum groups, where the $\imath$quiver with orientation forgotten can be viewed as a Satake diagram with no black nodes; this is a generalization of Bridgeland's construction in \cite{Br13}. Actually the $\imath$Hall algebras are naturally isomorphic to the universal $\imath$quantum groups $\tUi$ (introduced in \cite{LW19}), which contain various central elements. By definition, the $\imath$quantum groups $\Ui=\Ui_{\bvs}$ depend on parameters $\bvs =(\vs_i)_{i\in \I}$, and they can be obtained from $\tUi$ by central reductions.

\subsubsection{}

Kolb-Pellegrini \cite{KP11} constructed (via computer computations) automorphisms of $\Ui$ (over the extension field $\Q(\sqrt{v})$ of $\Q(v)$ for some distinguished parameter $\bvsd$) which satisfy braid group relations. For split type A $\imath$quantum groups, the braid group actions were found earlier in \cite{MR08} and \cite{Ch07}. It was already known by then that Lusztig's symmetries on $\U$ do not preserve $\Ui$ in general. We refer to \S\ref{subsec:KP} below for a more detailed discussion.

Like quantum groups of affine type, the $\imath$quantum groups of affine type are  of considerable interest in quantum integrable models. The $\imath$quantum group of split affine $\mathfrak{sl}_2$ type (i.e., the simplest $\imath$quantum group of affine type) is known as the $q$-Onsager algebra and has been extensively studied in physics, cf. \cite{BB12} and references therein. Lusztig type automorphisms for the $q$-Onsager algebra have recently been constructed and verified via computer computations in \cite{BK17} (also see \cite{Ter17}).

\subsection{The goal}

In this sequel to \cite{LW19}, we construct and study in depth BGP type reflection functors associated to the acyclic $\imath$quivers $(Q, \btau)$, which lead to isomorphisms $\Gamma_i$ (for a sink/source $i$ of $Q$) between $\imath$Hall algebras.

In the case of Dynkin $\imath$quivers, we obtain automorphisms $\tTT_i$ on the universal $\imath$quantum group $\tUi$ via the main isomorphism theorem in \cite{LW19} between $\imath$Hall algebras and $\imath$quantum groups. We show that the automorphisms $\tTT_i$ satisfy the braid group relations associated to the restricted Weyl group $W_\btau$ of the symmetric pair. For the distinguished parameter $\bvsd$ in \eqref{eq:bvsd}, the automorphisms $\tTT_i$ descend to automorphisms of $\Ui_{\bvsd}$, which are then identified with a variant of braid group symmetries given in \cite{KP11}.

For Dynkin $\imath$quivers, we then construct the braid group symmetries on $\Ui$ with arbitrary parameters $\bvs$.  The reflection functors allow us to define the $q$-root vectors in $\tUi$ (and $\Ui$). This leads to natural PBW bases for $\tUi$ (and for $\Ui$).

\subsection{An overview}

Geiss-Leclerc-Schr\"oer \cite{GLS17} generalized Gabriel's classic theorem to non-simply-laced types BCFG using certain quivers with relations.  The construction of BGP-type reflection functors in this paper bears similarities with the corresponding construction in \cite{GLS17}, as $\imath$quiver algebras can also be realized as modulated graphs; see \S\ref{subsec:modulated}. We also realize reflection functors in terms of tilting modules.

For a sink $\ell$ of $Q_0$, let $\bs_\ell Q$ be the quiver obtained from $Q$ by reversing the arrows ending at $\ell$ and $\btau \ell$. Denote by $\iLa$ (respectively, $\bs_\ell \iLa$)  the $\imath$quiver algebra for the $\imath$quiver $(Q,\btau)$ (respectively, $(\bs_\ell Q,\btau)$). A reflection functor $F_\ell ^+: \mod(\Lambda^{\imath})\rightarrow \mod(\bs_\ell \iLa )$ is defined in \eqref{def:RefF}.

\begin{customthm}{\bf A}  [Theorem~\ref{thm:tilting module}, \eqref{cor:Tiso}]
Let $\ell$ be a sink of $Q_0$.

(1) The module $T$ in \eqref{eqn:def of T} is a tilting module of $\Lambda^{\imath}$.

(2) The functors
$F_\ell ^+: \mod(\Lambda^{\imath})\rightarrow \mod(\bs_\ell \iLa )$ and $\Hom_{\Lambda^{\imath}}(T, -): \mod(\Lambda^{\imath})\rightarrow \mod(B)$
are equivalent, where $B=\End_{\Lambda^{\imath}}(T)^{op}$.

(3) We have an algebra isomorphism $\bs_\ell \iLa \cong \End_{\Lambda^{\imath}}(T)^{op}$.
\end{customthm}

The construction of reflection functors provides algebra isomorphisms between $\imath$Hall algebras $\tMH$ thanks to the tilting invariance property  (cf. Theorem ~{\bf A} and \cite[Theorem~A.22]{LW19}).

\begin{customprop}{\bf B}  [Theorem~\ref{thm:Gamma} and Proposition~\ref{prop:reflection}]
Let $\ell$ be a sink of $Q_0$. The functor
$F_\ell ^+$ induces an isomorphism of $\imath$Hall algebras,
$\Gamma_{\ell}:\tMH  \stackrel{\sim}{\longrightarrow}  \tMHl$, with explicit formulas for the action on the generators given by \eqref{eqn:reflection 1}--\eqref{eqn:reflection 4}.
\end{customprop}
The isomorphisms lift to generic $\imath$Hall algebras $\tMHg$ associated to acyclic $\imath$quivers. When composed with Fourier transforms, these isomorphisms become automorphisms of $\tMHg$.

The isomorphism $\Gamma_i$ is fairly explicit and computable thanks to Theorem~{\bf A}. For Dynkin $\imath$quivers, the actions of $\Gamma_i$ on the generators of $\tMH$, which consist of the classes of (generalized) simple modules, are computed explicitly; see Lemmas~\ref{lemma:reflection for split involution}--\ref{lem:RefFunctorE}. The formulas for $\Gamma_i$ can be lifted to generic $\imath$Hall algebras.

While our framework is valid for acyclic $\imath$quivers, our further results are most complete for {\em Dynkin} $\imath$quivers $(Q, \btau)$; we shall assume that $(Q, \btau)$ are Dynkin $\imath$quivers in the remainder of the introduction.

Thanks to the isomorphisms between $\imath$Hall algebras and $\tUi$ (proved in \cite{LW19}), we convert in \eqref{eq:defT} the isomorphisms $\Gamma_i$ to automorphisms $\tTT_i$ on  (quasi-split) universal $\imath$quantum groups $\tUi$ of finite type. In particular, we convert the formulas for $\Gamma_i$ on the $\imath$Hall algebras to formulas for $\tTT_i$ on the generators of $\tUi$; see Lemmas \ref{lemma:braid group of split involution}--\ref{lem:braid group of E}. Denote by $\brW =\langle t_i \mid i\in \ci \rangle$ the braid group associated to the restricted Weyl group $W_\btau$ of the symmetric pair underlying $(\U, \Ui)$; see \eqref{eq:Wtau} and \eqref{eq:braidCox}--\eqref{eq:BrBr}. Let $\F$ denote a field extension of $\Q(v)$ given in \eqref{def:ai}.

\begin{customthm}{\bf C} [Theorem~\ref{thm:bgtui}, Theorem~\ref{thm:braidUigeneral}]
Let $(Q,\btau)$ be a Dynkin $\imath$quiver. Then there exists a homomorphism
$\brW \rightarrow \aut( \tUi)$ (respectively, $\brW \rightarrow \aut({}_\F\Ui)$), $t_i\mapsto \tTT_i$, for $i\in \ci $.
\end{customthm}

A direct verification using a presentation of $\tUi$ would be extremely tedious and long, as demonstrated in the case of quantum groups \cite[Part VI]{Lus93}. We verify the braid group relations for $\imath$quantum groups in Theorem~{\bf C} via computations in $\imath$Hall algebras. Note that a verification of the braid group relations for a quantum group $\U$ in the setting of {\em reduced} Drinfeld doubles of Ringel-Hall algebras was given in \cite{SV99, XY01}. Based on these works, the braid group actions on $\U$ were also recovered from reflection functors in the setting of semi-derived Hall algebras  \cite{Gor13}.  The braid group actions in the modified quantum group are realized via BGP reflection functors in the setting of Drinfeld doubles of Ringel-Hall algebras in \cite{XZ13}.

The reduced $\imath$Hall algebra $\rMH =\rMH_\bvs$ is by definition the quotient of $\tMH$ by the ideal generated by \eqref{eqn: reduce1} associated to a parameter $\bvs$. Let $\bvsd$ denote the distinguished parameter given by formula \eqref{eq:bvsd}. For any sink $i\in Q_0$, $\Gamma_i$ in Theorem~{\bf B} descends to an algebra isomorphism {\em naturally} (for this distinguished parameter $\bvsd$ {\em only}) $\bar{\Gamma}_i: \rMHd \xrightarrow{\sim} \rMHi$ (see Proposition~\ref{prop:derived invariant of QSP}). This leads to an automorphism $\tTT_i$ on $\Ui_{\bvsd}$; cf. Corollary~\ref{cor:braidUi1}.

Up to a base change \eqref{def:ai}--\eqref{def:basechange}, the algebras ${}_\F\Ui ={}_\F\Ui_\bvs$ for different $\bvs$ are all isomorphic via conjugations by some rescaling isomorphisms. Then we can convert the automorphism $\TT_i$ of $\Ui_{\bvsd}$ to an automorphism, again denoted by $\TT_i$, of ${}_\F\Ui$ (for a general parameter $\bvs$).

We emphasize that Theorem~{\bf C} on braid group actions is expected to hold for $\tUi$ and $\Ui$ of Kac-Moody type associated to acyclic $\imath$quivers (see Conjecture~\ref{conj:braidKM}). Indeed, the same arguments in this paper already work for various simply-laced acyclic $\imath$quivers (including most affine ADE types).

We define (see Definition~\ref{def:admissible}) a notion of $\imath$-admissible sequences associated to Dynkin $\imath$quiver $(Q, \btau)$, generalizing the $(+)$-admissible sequences associated to reduced expressions of elements in a Weyl group $W$; cf., e.g., \cite{DDPW}. In case when $\btau=\Id$, $\imath$-admissible sequences are just $(+)$-admissible sequences. Note that the Weyl groups $W$ and $W_\btau$ share the same longest element $w_0$, and we denote by $N_{\imath}$ the length of $w_0$ in $W_\btau$.

\begin{customthm}{\bf D}  [Theorem~\ref{thm:i-seq}]
Let $(Q,\btau)$ be a Dynkin $\imath$quiver. Then there is an $\imath$-admissible sequence $i_1,\dots,i_{\Ni}$ such that
\[
\beta_{1}, \btau(\beta_1), \beta_{2}, \btau(\beta_2), \dots,\beta_{\Ni },\btau(\beta_{\Ni})
\]
is a $Q$-admissible sequence of roots in $\Phi^+$, where $\beta_j$ is defined in \eqref{def:betaj}. (By convention the redundant $\btau(\beta_j)$ is omitted here whenever $\btau(\beta_j)=\beta_j$.)
\end{customthm}

The braid group action allows us to study $q$-root vectors $B_\beta$, for $\beta\in \Phi^+$, in \eqref{def: B beta} for $\tUi$ (or $\Ui$) systematically.

\begin{customthm}{\bf E} [PBW basis Theorem~\ref{thm:PBWUi}]
For any ordering $\gamma_1,\ldots,\gamma_N$ of the roots in $\Phi^+$, the set
\[
\{ B_{\gamma_1}^{\lambda_1}\cdots B_{\gamma_N}^{\lambda_N}\mid \lambda\in \fp \}
\]
provides a basis of $\Ui$ as a right $\U^{\imath 0}$-module, where $\lambda_i=\lambda(\gamma_i)$ for $1\le i\le N$.
\end{customthm}

In analogy with quantum groups, one expects that the PBW bases associated to the $\imath$-admissible sequences may have some additional favorable property.

\subsection{Comparison with earlier work}
  \label{subsec:KP}

There has been a construction by Kolb-Pellegrini \cite{KP11} of braid group actions $\brW$ on $\Ui_{\bvsd}$ over the extension field $\Q(\sqrt{v})$ of $\Q(v)$; it is remarkable that they chose to work with the same parameter $\bvsd$ for perhaps computational reason while $\bvsd$ is distinguished from categorical viewpoint in this paper. Their construction covers all quasi-split finite type and also $\imath$quantum groups of type AII. 
Let us denote by $\tTT_i'$ the automorphism on $\Q(\sqrt{v}) \otimes_{\Q(v)} \Ui_{\bvsd}$ defined in \cite{KP11}.

We show that a conjugation by a nontrivial rescaling automorphism  of $\Q(\sqrt{v}) \otimes_{\Q(v)} \Ui_{\bvsd}$ changes $\tTT_i'$ to $\tTT_i$, see Remark~\ref{rem:KP11}; in this way we recover the braid group actions in \cite{KP11} for $\Ui$ arising from Dynkin $\imath$quivers.  The formulas for $\TT_i'$ and the claim that $\TT_i'$ are automorphisms and they satisfy the braid relations were discovered and verified in {\em loc. cit.} by computer computations. By contrast, in our approach there is no need to verify that $\TT_i$ is an automorphism using a presentation of $\tUi$ or $\Ui$, as it follows conceptually from its categorical origin. The question of developing a computer-free approach was raised in \cite{KP11}. We emphasize that our approach makes sense in Kac-Moody setting.

While the braid group relations at the level of $\tUi$ imply the braid relations at the level of $\Ui_{\bvsd}$, the converse is not true. The action of $\tTT_i$ on $\tUi$ moves its central elements around, while fixing the scalars (and hence the chosen parameters $\vs_i$).

\subsection{Further directions}

The approach developed in \cite{LW19} and this paper opens up new venues for research. Besides those alluded {\em loc. cit.}, let us list additional topics which are more in line with the present work. We plan to take up some of the problems in the sequels.

The construction of reflection functors in this paper is valid for all acyclic $\imath$quivers,  leading to automorphisms of $\imath$Hall algebras.
One natural question is to show that they give rise to braid group actions for $\imath$quantum groups of Kac-Moody type (cf. Conjecture~\ref{conj:braidKM}); the conjecture can be verified already in various cases beyond finite type. A further development of our previous work \cite{CLW18} with Chen will also provide another complementary conceptual approach toward Lusztig type automorphisms for $\Ui$ of Kac-Moody type.
In particular, the construction of automorphisms of the $q$-Onsager algebra (cf. \cite{BK17, Ter17}) will be recovered as the simplest case in such generalizations.

It should be possible (though highly nontrivial) to apply braid group operators to construct PBW bases for $\imath$quantum groups of affine type; cf. \cite{BK17} for the split affine $\mathfrak{sl}_2$ case.

It will be very useful to construct braid group actions of $\brW$ on the module level. This was done for the usual quantum groups by Lusztig \cite{Lus93}.

There is an inner product on the modified $\imath$quantum group $\dot \U^\imath$ defined in \cite{BW18b}. We may ask for the compatibility between the braid group action and the inner product. It is well known that similar compatibility holds for quantum groups \cite{Lus93}.

There are integral forms of modified $\imath$quantum groups $\dot \U^\imath$ \cite{BW18b}. It will be interesting to understand the braid group actions on the integral forms and more explicitly on $\imath$divided powers; compare  with \cite[Part~VI]{Lus93}. Furthermore, it will be important if a PBW approach toward the $\imath$canonical basis for $\dot\U^\imath$ (established first in \cite{BW18b}) can be developed.

Our $\imath$quiver approach excludes $\Ui$ of type $A_{2r}$ with nontrivial $\btau$ (since in this case $\btau$ does not preserve any orientation on $A_{2r}$). It will be of great interest to develop a more flexible and general framework for $\imath$Hall algebras to cover this missing case.

In \cite{LW19} and this paper, we have set up a framework to study the $\imath$quantum groups $\Ui$ from a Hall algebra approach. It will be of great interest to develop this approach fully to cover the $\imath$quantum groups associated to Satake diagrams (with black nodes). Then one should be able to extend the construction of reflection functors and braid group actions in this paper to the general setting as well. Recently Dobson \cite{D19} developed an algebraic approach to the braid group action for $\imath$quantum groups of type AIII/AIV (whose Satake diagram contains black nodes).


%
%
\subsection{Organization}

This paper is organized as follows.
In Section~\ref{sec:prelim}, we recall the basics on quantum groups and (universal) quantum symmetric pairs. We review the $\imath$Hall algebra approach to the $\imath$quantum groups from \cite{LW19}.

In Section~\ref{sec:RF}, we introduce the BGP type reflection functors and APR-tilting for $\imath$quivers. The Auslander-Reiten quivers for the rank 2 $\imath$quivers are described explicitly. These reflection functors are shown in Section~\ref{sec:symmetryHall} to give rise to isomorphisms $\Gamma_i$ between $\imath$Hall algebras, using the tilting invariance property. The formulas for the isomorphisms $\Gamma_i$ on (generalized) simple modules are worked out explicitly for Dynkin $\imath$quivers.

In Section~\ref{sec:symmetryUi}, for Dynkin $\imath$quivers, we convert the isomorphism $\Gamma_i$ on $\imath$Hall algebras (with formulas of $\Gamma_i$ on simple modules) to an automorphism $\TT_i$ of $\tUi$ (with formulas on generators).
Then in Section~\ref{sec:braid}, we verify the braid group relations between $\TT_i$ via $\imath$Hall algebra computations.

In Section~\ref{sec:distinguished}, at the distinguished parameter $\bvsd$,  the isomorphisms $\Gamma_i$ are shown to factor through reduced $\imath$Hall algebras. Accordingly, the automorphisms $\TT_i$ factor through $\imath$quantum group $\Ui_{\bvsd}$ and they satisfy the braid group relations. We explain the precise relation between our construction and the earlier construction of braid group actions in \cite{KP11}.

In Section~\ref{sec:PBW}, we construct the braid group actions on $\Ui$ with general parameter. Then we define the $q$-root vectors and use them to construct PBW bases for $\tUi$ and $\Ui$.

\subsection*{Acknowledgments}
We thank Institute of Mathematics at Academia Sinica (Taipei) and East China Normal University (Shanghai) for hospitality and support where part of this work was carried out. WW is partially supported by NSF grant DMS-1702254 and DMS-2001351. We thank 2 anonymous referees for their careful readings and suggestions which help make this paper more readable.

%

%

\section{The preliminaries}
\label{sec:prelim}

We review and set up notations for quantum groups, $\imath$quantum groups, and $\imath$Hall algebras.

\subsection{Notations}
\label{subsec: notat}
For an additive category $\ca$ and $M\in \ca$, we denote

$\triangleright$ $\add M$ -- the full subcategory of $\ca$ whose objects are the direct summands of finite direct sums of copies of $M$,

$\triangleright$ $\Fac M$-- the full subcategory of $\ca$ of epimorphic images of objects in $\add M$,

$\triangleright$ $\Ind (\ca)$ --  set of the isoclasses of indecomposable objects in $\ca$,

$\triangleright$ $\Iso(\ca)$ -- set of the isoclasses of objects in $\ca$,

$\triangleright$ $[M]$ -- the isoclass of $M$.
\medskip

For an exact category $\ca$ and $M\in\ca$, we denote

$\triangleright$ $K_0(\ca)$ -- the Grothendieck group of $\ca$,

$\triangleright$ $\widehat{M}$ -- the class of $M$ in $K_0(\ca)$.

\medskip

Let $\K$ be a field.
For a finite-dimensional $k$-algebra $A$, we denote

$\triangleright$ $\mod(A)$ -- category of finite-dimensional left $A$-modules,

$\triangleright$ $\proj(A)$ -- category of finite-dimensional left projective $A$-modules,

$\triangleright$ $\inj(A)$ -- category of finite-dimensional left injective $A$-modules,


$\triangleright$ $D=\Hom_\K(-,\K)$ -- the standard duality,

$\triangleright$ $D^b(A)$ -- bounded derived category of finite-dimensional $A$-modules.

\medskip

Let $Q=(Q_0, Q_1)$ be a quiver. We denote (cf. \cite{ASS04})

$\triangleright$ $e_i$ -- the trivial path at $i\in Q_0$,

$\triangleright$ $(Q,I)$  -- a bound quiver, where $Q$ is a quiver and $I$ is an admissible ideal of $kQ$,

$\triangleright$ $kQ/I$ -- the path algebra of the bound quiver $(Q,I)$,

$\triangleright$ $\rep(Q,I)$ -- category of finite-dimensional representations of $(Q,I)$.

It is well known that $\rep(Q,I) \simeq \mod(kQ/I)$, and we shall always identify them. 
\subsection{The $\imath$quivers and doubles}
Let $Q=(Q_0,Q_1)$ be an acyclic quiver (i.e., a quiver without oriented cycles), and we sometimes write $\I =Q_0$. An {\em involution} of $Q$ is defined to be an automorphism $\btau$ of the quiver $Q$ such that $\btau^2=\Id$. In particular, we allow the {\em trivial} involution $\Id:Q\rightarrow Q$. An involution $\btau$ of $Q$ induces an involution of the path algebra $\K Q$, again denoted by $\btau$.
A quiver together with an involution $\btau$, $(Q, \btau)$, will be called an {\em $\imath$quiver}.

Let $R_1$ denote the truncated polynomial algebra $\K[\varepsilon]/(\varepsilon^2)$.
Let $R_2$ denote the radical square zero of the path algebra of $\xymatrix{1 \ar@<0.5ex>[r]^{\varepsilon} & 1' \ar@<0.5ex>[l]^{\varepsilon'}}$, i.e., $\varepsilon' \varepsilon =0 =\varepsilon\varepsilon '$. Define a $\K$-algebra
\begin{equation}
  \label{eq:La}
\Lambda=\K Q\otimes_\K R_2.
\end{equation}

Associated to the quiver $Q$, the {\em double framed quiver} $Q^\sharp$
is the quiver such that
\begin{itemize}
\item the vertex set of $Q^{\sharp}$ consists of 2 copies of the vertex set $Q_0$, $\{i,i'\mid i\in Q_0\}$;
\item the arrow set of $Q^{\sharp}$ is
\[
\{\alpha: i\rightarrow j,\alpha': i'\rightarrow j'\mid(\alpha:i\rightarrow j)\in Q_1\}\cup\{ \varepsilon_i: i\rightarrow i' ,\varepsilon'_i: i'\rightarrow i\mid i\in Q_0 \}.
\]
\end{itemize}
Let $I^{\sharp}$ be the admissible ideal of $\K Q^{\sharp}$ generated by
\begin{itemize}
\item
(Nilpotent relations) $\varepsilon_i \varepsilon_i'$, $\varepsilon_i'\varepsilon_i$ for any $i\in Q_0$;
\item
(Commutative relations) $\varepsilon_j' \alpha' -\alpha\varepsilon_i'$, $\varepsilon_j \alpha -\alpha'\varepsilon_i$ for any $(\alpha:i\rightarrow j)\in Q_1$.
\end{itemize}
Then the algebra $\La$ can be realized as $\Lambda\cong \K Q^{\sharp} \big/ I^{\sharp}$.

Note $Q^\sharp$ admits a natural involution, $\swa$.
The involution $\btau$ of a quiver $Q$ induces an involution ${\btau}^{\sharp}$ of $Q^{\sharp}$ which is the composition of $\swa$ and $\btau$ (on the two copies of subquivers $Q$ and $Q'$ of $Q^\sharp$).

%

By \cite[Lemma~2.4]{LW19}, ${\btau}^{\sharp}$ on $Q^\sharp$ preserves $I^\sharp$ and hence induces an involution ${\btau}^{\sharp}$ on the algebra $\Lambda$. By \cite[Definition 2.5]{LW19}, the {\rm $\imath$quiver algebra} of $(Q, \btau)$ is the fixed point subalgebra of $\Lambda$ under ${\btau}^{\sharp}$,
\begin{equation}
   \label{eq:iLa}
\iLa
= \{x\in \Lambda\mid {\btau}^{\sharp}(x) =x\}.
\end{equation}
The algebra $\iLa$ can be described in terms of a certain quiver $\ov Q$ and its admissible ideal $\ov{I}$ so that $\iLa \cong \K \ov{Q} / \ov{I}$; see \cite[Proposition 2.6]{LW19}.
We recall $\ov{Q}$ and $\ov{I}$ as follows:
\begin{itemize}
\item[(i)] $\ov{Q}$ is constructed from $Q$ by adding a loop $\varepsilon_i$ at the vertex $i\in Q_0$ if $\btau i=i$, and adding an arrow $\varepsilon_i: i\rightarrow \btau i$ for each $i\in Q_0$ if $\btau i\neq i$;
\item[(ii)] $\ov{I}$ is generated by
\begin{itemize}
\item[(1)] (Nilpotent relations) $\varepsilon_{i}\varepsilon_{\btau i}$ for any $i\in\I$;
\item[(2)] (Commutative relations) $\varepsilon_i\alpha-\btau(\alpha)\varepsilon_j$ for any arrow $\alpha:j\rightarrow i$ in $Q_1$.
\end{itemize}
\end{itemize}
Moreover, it follows by \cite[Proposition 3.5]{LW19} that $\Lambda^{\imath}$ is a $1$-Gorenstein algebra.


By \cite[Corollary 2.12]{LW19}, $\K Q$ is naturally a subalgebra and also a quotient algebra of $\Lambda^\imath$.
Viewing $\K Q$ as a subalgebra of $\Lambda^{\imath}$, we have a restriction functor
\[
\res: \mod (\Lambda^{\imath})\longrightarrow \mod (\K Q).
\]
Viewing $\K Q$ as a quotient algebra of $\Lambda^{\imath}$, we obtain a pullback functor
\begin{equation}\label{eqn:rigt adjoint}
\iota:\mod(\K Q)\longrightarrow\mod(\Lambda^{\imath}).
\end{equation}
Hence a simple module $S_i (i\in Q_0)$ of $k Q$ is naturally a simple $\iLa$-module.

For each $i\in Q_0$, define a $k$-algebra (which can be viewed as a subalgebra of $\iLa$)
\begin{align}\label{dfn:Hi}
\BH _i:=\left\{ \begin{array}{cc}  \K[\varepsilon_i]/(\varepsilon_i^2) & \text{ if }\btau i=i,
 \\
\K(\xymatrix{i \ar@<0.5ex>[r]^{\varepsilon_i} & \btau i \ar@<0.5ex>[l]^{\varepsilon_{\btau i}}})/( \varepsilon_i\varepsilon_{\btau i},\varepsilon_{\btau i}\varepsilon_i)  &\text{ if } \btau i \neq i .\end{array}\right.
\end{align}
Note that $\BH _i=\BH _{\btau i}$ for any $i\in Q_0$. 
Choose one representative for each $\btau$-orbit on $\I=Q_0$, and let
\begin{align}   \label{eq:ci}
\ci = \{ \text{the chosen representatives of $\btau$-orbits in $\I$} \}.
\end{align}

Define the following subalgebra of $\Lambda^{\imath}$:
\begin{equation}  \label{eq:H}
\BH =\bigoplus_{i\in \ci }\BH _i.
\end{equation}
Note that $\BH $ is a radical square zero selfinjective algebra. Denote by
\begin{align}
\res_\BH :\mod(\iLa)\longrightarrow \mod(\BH )
\end{align}
the natural restriction functor.
On the other hand, as $\BH $ is a quotient algebra of $\iLa$ (cf. \cite[proof of Proposition~ 2.15]{LW19}), every $\BH $-module can be viewed as a $\iLa$-module.

For $i\in \I$, define the indecomposable module over $\BH _i$ (if $i\in \ci$) or over $\BH_{\btau i}$ (if $i\not \in \ci$)
\begin{align}
  \label{eq:E}
\E_i =\begin{cases}
k[\varepsilon_i]/(\varepsilon_i^2), & \text{ if }\btau i=i;
\\
\xymatrix{\K\ar@<0.5ex>[r]^1 & \K\ar@<0.5ex>[l]^0} \text{ on the quiver } \xymatrix{i\ar@<0.5ex>[r]^{\varepsilon_i} & \btau i\ar@<0.5ex>[l]^{\varepsilon_{\btau i}} }, & \text{ if } \btau i\neq i.
\end{cases}
\end{align}
Then $\E_i$, for $i\in Q_0$, can be viewed as a $\iLa$-module and will be called a {\em generalized simple} $\iLa$-module.

We have the following projective resolution of $\E_i$:
\begin{equation}\label{eqn:projective resolution of E}
0\longrightarrow \bigoplus_{(\alpha: i\rightarrow j)\in Q_1} \Lambda^{\imath}\, e_j\longrightarrow \Lambda^{\imath}\,e_i\longrightarrow \E_i\longrightarrow0.
\end{equation}
We also have the following injective resolution of $\E_i$:
\begin{equation}\label{eqn:injective resolution of E}
0\longrightarrow \E_i \longrightarrow D(e_{\btau i} \Lambda^{\imath})\longrightarrow \bigoplus_{(\alpha:j\rightarrow \btau i)\in Q_1}D(e_{j}  \Lambda^{\imath})\longrightarrow0.
\end{equation}

\subsection{$\imath$Hall algebras}

In this subsection we consider $k=\F_q$, and
\[
\sqq=\sqrt{q}.
\]

Generalizing \cite{LP}, the first author defined a (twisted) semi-derived  Hall algebra of a 1-Gorenstein algebra; see \cite{Lu19}. The $\imath$Hall algebra $\tMH$ for $\imath$quiver $(Q,\btau)$ is by definition the twisted semi-derived Hall algebra for the module category of the $\imath$quiver algebra $\iLa$ \cite{LW19}, and its generic version is denoted by $\tMHg$.
For convenience, we recall it here briefly.

Let $\ch(\Lambda^\imath)$ be the Ringel-Hall algebra of $\Lambda^\imath$, i.e.,
$$\ch(\Lambda^\imath)=\bigoplus_{[M]\in\Iso(\mod(\Lambda^\imath))}\Q(\sqq)[M],$$
with the multiplication defined by (see \cite{Br13})
\[
[M]\diamond [N]=\sum_{[M]\in\Iso(\mod(\Lambda^\imath))}\frac{|\Ext^1(M,N)_L|}{|\Hom(M,N)|}[L].
\]
Note that this Hall product here is different from the one defined by Ringel \cite{Rin90}.
Let $I$ be the two-sided ideal of $\ch(\Lambda^\imath)$ generated by all differences $[L]-[K\oplus M]$ if there is a short exact sequence
\begin{equation}
  \label{eq:ideal}
 0 \longrightarrow K \longrightarrow L \longrightarrow M \longrightarrow 0
\end{equation}
in $\mod(\Lambda^\imath)$ with $K\in\cp^{<\infty}(\Lambda^\imath)$. Here $\cp^{<\infty}(\Lambda^\imath)$ is the subcategory of $\mod(\Lambda^\imath)$ formed by modules of finite projective dimensions.

Consider the following multiplicatively closed subset $\cs$ of $\ch(\Lambda^\imath)/I$:
\begin{equation}
  \label{eq:Sca}
\cs = \{ a[K] \in \ch(\Lambda^\imath)/I \mid a\in \Q(\sqq)^\times, K\in \cp^{<\infty}(\Lambda^\imath)\}.
\end{equation}

The semi-derived (Ringel-)Hall algebra of $\Lambda^\imath$ \cite{Lu19} is defined to be the localization
$$\utMH:= (\ch(\Lambda^\imath)/I)[\cs^{-1}].$$

Let $\langle\cdot,\cdot \rangle_Q$ be the Euler form of $Q$,  i.e.,
\[
\langle M,N\rangle_Q=\dim_k\Hom_{kQ}(M,N)-\dim_k\Ext^1_{kQ}(M,N)
\]
 for $M,N\in\mod(kQ)$.
We define the {\em $\imath$Hall algebra} $\tMH$ as a twisted semi-derived Hall algebra \cite[\S4.4]{LW19}, that is, the $\Q(\sqq)$-algebra on the same vector space as $\utMH$ but with twisted multiplication given by
\begin{align}
   \label{eqn:twsited multiplication}
[M]* [N] =\sqq^{\langle \res(M),\res(N)\rangle_Q} [M]\diamond[N].
\end{align}

Let $\tTL$ be the subalgebra of $\tMH$ generated by $\E_\alpha$, $\alpha\in K_0(\mod(\K Q))$, which is a Laurent polynomial algebra in $[\E_i]$, for $i\in \I$.


Let $\bvs=(\vs_i)\in   (\Q(\sqq)^\times)^{\I}$ be such that $\vs_i=\vs_{\btau i}$ for each $i\in \I$. The \emph{reduced $\imath$Hall algebra associated to $(Q,\btau)$} \cite[Definition 4.11]{LW19}, denoted by $\rMH$, is defined to be the quotient $\Q(\sqq)$-algebra of $\tMH$ by the ideal generated by the central elements
\begin{align}
\label{eqn: reduce1}
[\E_i] +q \vs_i \; (\forall i\in \I \text{ with } \btau i=i), \text{ and }\; [\E_i]*[\E_{\btau i}] -\vs_i^2\; (\forall i\in \I \text{ with }\btau i\neq i).
\end{align}

\subsection{Generic $\imath$Hall algebras}
\label{sub:generic}

For a Dynkin $\imath$quiver $(Q,\btau)$,
we recall the generic $\imath$Hall algebras defined in \cite[\S9.3]{LW19}.
Let $\Phi^+$ be the set of positive roots of $Q$ with simple roots $\alpha_i$, $i\in Q_0$.
For any $\alpha\in\Phi^+$, denote by $M_q(\alpha)$ its corresponding indecomposable $\K Q$-module, i.e., $\dimv M_q(\alpha)=\alpha$.
Let $\mathfrak{P}:=\mathfrak{P}(Q)$ be the set of functions $\lambda: \Phi^+\rightarrow \N$.
Then the modules
\begin{align}
\label{def:Mlambda}
M_q(\lambda):= \bigoplus_{\alpha\in\Phi^+}\lambda(\alpha) M_q(\alpha),\quad \text{ for } \lambda\in\mathfrak{P},
\end{align}
provide a complete set of isoclasses of $\K Q$-modules.

Let $\Phi^0$ denote a set of symbols $\upgamma_i$, i.e., $\Phi^0=\{\upgamma_{i}\mid  i\in \I\}$, and let
\begin{align}
  \label{eq:Phi2}
  \Phi^\imath=&\Phi^+\cup \Phi^{0},
  \\
  \label{eq:tPi}
\tilde{\fp}^\imath=&\{\lambda: \Phi^\imath \rightarrow \Z \mid \lambda(\alpha)\in \N\,\, \forall \alpha\in \Phi^+,  \lambda(\upgamma_i)\in\Z\,\, \forall \upgamma_i\in\Phi^0\}.
\end{align}
Then for $\lambda\in \tilde{\fp}^\imath$, define
\[
[M_q(\lambda)]:=[ \oplus_{\alpha\in\Phi^+} {\lambda(\alpha)}M_q(\alpha)]*[\E_{\sum_{i\in \I}\lambda(\upgamma_i)\alpha_i}].
\]

It is shown in \cite[Lemma 9.6]{LW19} that for every $\lambda,\mu,\nu \in \tilde{\fp}^\imath$ there exists a polynomial $\boldsymbol{\varphi}^\lambda_{\mu,\nu}(\bv)\in\Z[\bv,\bv^{-1}]$ such that
\[
[M_q(\mu)]*[M_q(\nu)]=\sum_{\lambda\in\tilde{\fp}^\imath}\boldsymbol{\varphi}^\lambda_{\mu,\nu}({\sqq})[M_q(\lambda)].
\]

The algebra $\tMHg$ is defined to be the free $\Q(\bv)$-module with a basis $\{\fu_\lambda\mid \lambda\in\tilde{\fp}^\imath\}$ with multiplication
\begin{align}
\fu_\mu*\fu_\nu=\sum_{\lambda\in\tilde{\fp}^\imath}\boldsymbol{\varphi}_{\mu,\nu}^\lambda(\bv)\fu_\lambda.
\end{align}
{Let $\bvs=(\vs_i)\in   (\Q(v)^\times)^{\I}$ be such that $\vs_i=\vs_{\btau i}$ for each $i\in \I$. Its reduced generic version, denoted by $\rMHg$, is defined to be the quotient of $\tMHg$ by the ideal generated by
\begin{align}
   \label{eqn:specializing3}
\fu_{\upgamma_i} +v^2\vs_i \; (\text{for }\btau i=i),\qquad
\fu_{\upgamma_i}*\fu_{\btau \upgamma_i} -\vs_i^2 \; (\text{for }\btau i\neq i).
\end{align}


%

%
%
\subsection{Quantum groups}
  \label{subsec:QG}

For a  (generalized) Cartan matrix $C=(c_{ij})$, let $\Aut(C)$ be the group of all permutations $\btau$ of the set $\I$ such that $c_{ij}=c_{\btau i,\btau j}$. An element $\btau\in\Aut(C)$ is called an \emph{involution} if $\btau^2=\Id$.

Let $Q$ be an acyclic quiver with vertex set $Q_0= \I$.
Let $n_{ij}$ be the number of edges connecting vertex $i$ and $j$. Let $C=(c_{ij})_{i,j \in \I}$ be the symmetric generalized Cartan matrix of the underlying graph of $Q$, defined by $c_{ij}=2\delta_{ij}-n_{ij}.$ Let $\fg$ be the corresponding Kac-Moody Lie algebra. Let $\alpha_i$ ($i\in\I $) be the simple roots of $\fg$, and denote the root lattice by $\Z^{\I}:=\Z\alpha_1\oplus\cdots\oplus\Z\alpha_n$. The {\em simple reflection} $s_i:\Z^{\I}\rightarrow\Z^{\I}$ is defined to be $s_i(\alpha_j)=\alpha_j-c_{ij}\alpha_i$, for $i,j\in \I$.
Denote the Weyl group by $W =\langle s_i\mid i\in \I\rangle$.

Let $\btau$ be an involution of $Q$. We assume that $c_{i,\btau i}=0$ for all $i\neq \btau i$, which always holds for the {\em Dynkin} $\imath$quivers. We denote by $\bs_{i}$ the following element of order 2 in the Weyl group $W$, i.e.,
\begin{align}
\label{def:simple reflection}
\bs_i= \left\{
\begin{array}{ll}
s_{i}, & \text{ if } \btau i=i;
\\
s_is_{\btau i}, & \text{ if } \btau i\neq i.
\end{array}
\right.
\end{align}
It is well known (cf., e.g., \cite{KP11}) that the {\rm restricted Weyl group} associated to the quasi-split symmetric pair $(\fg, \fg^\theta)$ can be identified with the following subgroup $W_\btau$ of $W$:
\begin{align}
  \label{eq:Wtau}
W_{\btau} =\{w\in W\mid \btau w =w \btau\}
\end{align}
where $\btau$ is regarded as an automorphism of $\Aut(C)$. Moreover, it admits the following property  (cf., e.g., \cite{KP11}), which can also be checked directly using \eqref{def:simple reflection}.

\begin{lemma}
  \label{lem:iWeyl}
The restricted Weyl group $W_{\btau}$ can be identified with a Weyl group with $\bs_i$ ($i\in \I_\btau$) as its simple reflections. More precisely, associated to the Dynkin $\imath$quivers $(Q,\btau)$, we have
\begin{align*}
 W_{\btau} =
 \begin{cases}
 W, & \text{ if } \btau =\id,
 \\
  W(B_{r+1}), & \text{ if } \btau \neq \id, \text{ for }Q \text{ of type } A_{2r+1} \text{ or }D_{2r},
  \\
  W(F_4), & \text{ if } \btau \neq \id, \text{ for } Q \text{ of type } E_6.
 \end{cases}
\end{align*}
\end{lemma}

Let $\bv$ be an indeterminant. Write $[A, B]=AB-BA$. Denote, for $r,m \in \N$,
\[
 [r]=\frac{\bv^r-\bv^{-r}}{\bv-\bv^{-1}},
 \quad
 [r]!=\prod_{i=1}^r [i], \quad \qbinom{m}{r} =\frac{[m][m-1]\ldots [m-r+1]}{[r]!}.
\]
Then $\tU := \tU_\bv(\fg)$ is defined to be the $\Q(\bv)$-algebra generated by $E_i,F_i, \tK_i,\tK_i'$, $i\in \I$, where $\tK_i, \tK_i'$ are invertible, subject to the following relations:
\begin{align}
[E_i,F_j]= \delta_{ij} \frac{\tK_i-\tK_i'}{\bv-\bv^{-1}},  &\qquad [\tK_i,\tK_j]=[\tK_i,\tK_j']  =[\tK_i',\tK_j']=0,
\label{eq:KK}
\\
\tK_i E_j=\bv^{c_{ij}} E_j \tK_i, & \qquad \tK_i F_j=\bv^{-c_{ij}} F_j \tK_i,
\label{eq:EK}
\\
\tK_i' E_j=\bv^{-c_{ij}} E_j \tK_i', & \qquad \tK_i' F_j=\bv^{c_{ij}} F_j \tK_i',
 \label{eq:K2}
\end{align}
 and the quantum Serre relations, for $i\neq j \in \I$,
\begin{align}
& \sum_{r=0}^{1-c_{ij}} (-1)^r \left[ \begin{array}{c} 1-c_{ij} \\r \end{array} \right]  E_i^{1-c_{ij}-r} E_j  E_i^r =0,
  \\
& \sum_{r=0}^{1-c_{ij}} (-1)^r \left[ \begin{array}{c} 1-c_{ij} \\r \end{array} \right]  F_i^{1-c_{ij}-r} F_j  F_i^r =0.
  \label{eq:serre2}
\end{align}
Note that $\tK_i \tK_i'$ are central in $\tU$ for all $i$.
The comultiplication $\Delta: \widetilde{\U} \rightarrow \widetilde{\U} \otimes \widetilde{\U}$ is given by
\begin{align}  \label{eq:Delta}
\begin{split}
\Delta(E_i)  = E_i \otimes 1 + \tK_i \otimes E_i, & \quad \Delta(F_i) = 1 \otimes F_i + F_i \otimes \tK_{i}', \\
 \Delta(\tK_{i}) = \tK_{i} \otimes \tK_{i}, & \quad \Delta(\tK_{i}') = \tK_{i}' \otimes \tK_{i}'.
 \end{split}
\end{align}

Similarly to $\tU$, the quantum group $\bU$ is defined to be the $\Q(v)$-algebra generated by $E_i,F_i, K_i, K_i^{-1}$, $i\in \I$, subject to the  relations modified from \eqref{eq:KK}--\eqref{eq:serre2} with $\tK_i$ and $\tK_i'$ replaced by $K_i$ and $K_i^{-1}$, respectively. The comultiplication $\Delta$ is obtained by modifying \eqref{eq:Delta} with $\tK_i$ and $\tK_i'$ replaced by $K_i$ and $K_i^{-1}$, respectively (cf. \cite{Lus93}; beware that our $K_i$ has a different meaning from $K_i \in \U$ therein.)


Let $\widetilde{\bU}^+$ be the subalgebra of $\widetilde{\bU}$ generated by $E_i$ $(i\in \I)$, $\widetilde{\bU}^0$ be the subalgebra of $\widetilde{\bU}$ generated by $\tK_i, \tK_i'$ $(i\in \I)$, and $\widetilde{\bU}^-$ be the subalgebra of $\widetilde{\bU}$ generated by $F_i$ $(i\in \I)$, respectively.
The subalgebras $\bU^+$, $\bU^0$ and $\bU^-$ of $\bU$ are defined similarly. Then both $\widetilde{\bU}$ and $\bU$ have triangular decompositions:
\begin{align*}
\widetilde{\bU} =\widetilde{\bU}^+\otimes \widetilde{\bU}^0\otimes\widetilde{\bU}^-,
\qquad
\bU &=\bU^+\otimes \bU^0\otimes\bU^-.
\end{align*}
Let $\bvs=(\vs_i)\in(\Q(v)^\times)^\I$.
Clearly, ${\bU}^+\cong\widetilde{\bU}^+$, ${\bU}^-\cong \widetilde{\bU}^-$, and ${\bU}^0 \cong \widetilde{\bU}^0/(\tK_i \tK_i' -\vs_i \mid   i\in \I)$.

\subsection{The $\imath$quantum groups}
  \label{subsec:iQG}

For a  (generalized) Cartan matrix $C=(c_{ij})$, 
let $\btau$ be an involution in $\Aut(C)$. We define $\widetilde{\bU}^\imath:=\widetilde{\bU}'_\bv(\fk)$ to be the $\Q(v)$-subalgebra of $\tU$ generated by
\begin{equation}
  \label{eq:Bi}
B_i= F_i +  E_{\btau i} \tK_i',
\qquad \tk_i = \tK_i \tK_{\btau i}', \quad \forall i \in \I.
\end{equation}
Let $\tU^{\imath 0}$ be the $\Q(v)$-subalgebra of $\tUi$ generated by $\tk_i$, for $i\in \I$.
By \cite[Lemma 6.1]{LW19}, the elements $\tk_i$ (for $i= \btau i$) and $\tk_i \tk_{\btau i}$  (for $i\neq \btau i$) are central in $\tUi$.

When $(Q, \btau)$ is a Dynkin $\imath$quiver, a presentation for $\tUi$ can be found in \cite[Proposition~ 6.4]{LW19}.

Let $\bvs=(\vs_i)\in  (\Q(\bv)^\times)^{\I}$ be such that $\vs_i=\vs_{\btau i}$ for each $i\in \I$ which satisfies $c_{i, \btau i}=0$.
Let $\Ui:=\Ui_{\bvs}$ be the $\Q(v)$-subalgebra of $\bU$ generated by
\[
B_i= F_i+\vs_i E_{\btau i}K_i^{-1},
\quad
k_j= K_jK_{\btau j}^{-1},
\qquad  \forall i \in \I, j \in \ci.
\]
It is known \cite{Let99, Ko14} that $\bU^\imath$ is a right coideal subalgebra of $\bU$ in the sense that $\Delta: \Ui \rightarrow \Ui\otimes \U$; and $(\bU,\Ui)$ is called a \emph{quantum symmetric pair} ({\em QSP} for short), as they specialize at $v=1$ to $(U(\fg), U(\fg^\theta))$, where $\theta=\omega \circ \btau$, $\omega$ is the  Chevalley involution, and $\btau$ is understood here as an automorphism of $\fg$.

The algebras $\Ui_{\bvs}$, for $\bvs \in  (\Q(\bv)^\times)^{\I}$, are obtained from $\tUi$ by central reductions.

\begin{proposition}
[\text{\cite[Proposition 6.2]{LW19}}]
  \label{prop:QSP12}
(1) The algebra $\Ui$ is isomorphic to the quotient of $\tUi$ by the ideal generated by
\begin{align}   \label{eq:parameters}
\tk_i - \vs_i \; (\text{for } i =\btau i),
\qquad  \tk_i \tk_{\btau i} - \vs_i \vs_{\btau i}  \;(\text{for } i \neq \btau i).
\end{align}
The isomorphism is given by sending $B_i \mapsto B_i, k_j \mapsto \vs_j^{-1} \tk_j, k_j^{-1} \mapsto \vs_{\btau j}^{-1} \tk_{\btau j}, \forall i\in \I, j\in \ci$.

(2) The algebra $\widetilde{\bU}^\imath$ is a right coideal subalgebra of $\widetilde{\bU}$; that is, $(\widetilde{\bU}, \widetilde{\bU}^\imath)$ forms a QSP.
\end{proposition}

We shall refer to $\tUi$ and $\Ui$ as {\em (quasi-split) $\imath${}quantum groups}; they are called {\em split} if $\btau =\Id$.

Let us choose the subset $\ci$ of representatives of $\btau$-orbits on $\I$ as follows:
\begin{align}
\label{eqn:representative}
&\ci :=\left\{ \begin{array}{cl} \I, & \text{ if }\btau=\Id,\\
\left.\begin{array}{cl}
\{0,1,\dots,r\}, & \text{ if } \Delta \text{ is of type }A_{2r+1},\\
\{1,\dots,n-1\}, & \text{ if } \Delta \text{ is of type }D_n,\\
\{1,2,3,4\},  & \text{ if } \Delta \text{ is of type }E_6,\end{array} \right\}
& \text{ if }\btau\neq\Id.\end{array}\right.
\end{align}

The Dynkin diagrams for the non-split Dynkin $\imath$quivers can be found from diagrams below by ignoring the purple arrows (i.e., those labeled by the $\epsilon$'s).
\begin{center}\setlength{\unitlength}{0.7mm}
\vspace{-1.5cm}
\begin{equation}
\begin{picture}(100,40)(0,20)
\put(0,10){$\circ$}
\put(0,30){$\circ$}
\put(50,10){$\circ$}
\put(50,30){$\circ$}
\put(72,10){$\circ$}
\put(72,30){$\circ$}
\put(92,20){$\circ$}
\put(0,6){$r$}
\put(-2,34){${-r}$}
\put(50,6){\small $2$}
\put(48,34){\small ${-2}$}
\put(72,6){\small $1$}
\put(70,34){\small ${-1}$}
\put(92,16){\small $0$}

\put(3,11.5){\vector(1,0){16}}
\put(3,31.5){\vector(1,0){16}}
\put(23,10){$\cdots$}
\put(23,30){$\cdots$}
\put(33.5,11.5){\vector(1,0){16}}
\put(33.5,31.5){\vector(1,0){16}}
\put(53,11.5){\vector(1,0){18.5}}
\put(53,31.5){\vector(1,0){18.5}}
\put(75,12){\vector(2,1){17}}
\put(75,31){\vector(2,-1){17}}
\color{purple}
\put(0,13){\vector(0,1){17}}
\put(2,29.5){\vector(0,-1){17}}
\put(50,13){\vector(0,1){17}}
\put(52,29.5){\vector(0,-1){17}}
\put(72,13){\vector(0,1){17}}
\put(74,29.5){\vector(0,-1){17}}

\put(-5,20){$\varepsilon_r$}
\put(3,20){$\varepsilon_{-r}$}
\put(45,20){\small $\varepsilon_2$}
\put(53,20){\small $\varepsilon_{-2}$}
\put(67,20){\small $\varepsilon_1$}
\put(75,20){\small $\varepsilon_{-1}$}
\put(92,30){\small $\varepsilon_0$}

\qbezier(93,23)(90.5,25)(92,27)
\qbezier(92,27)(94,30)(97,27)
\qbezier(97,27)(98,24)(95.5,22.6)
\put(95.6,23){\vector(-1,-1){0.3}}
\end{picture}
\end{equation}
\vspace{-0.5cm}
\end{center}

\begin{center}\setlength{\unitlength}{0.8mm}
\begin{equation}
 \begin{picture}(100,25)(-5,0)
\put(0,-1){$\circ$}
\put(0,-5){\small$1$}
\put(20,-1){$\circ$}
\put(20,-5){\small$2$}
\put(64,-1){$\circ$}
\put(84,-10){$\circ$}
\put(80,-13){\small${n-1}$}
\put(84,9.5){$\circ$}
\put(84,12.5){\small${n}$}

\put(19.5,0){\vector(-1,0){16.8}}
\put(38,0){\vector(-1,0){15.5}}
\put(64,0){\vector(-1,0){15}}

\put(40,-1){$\cdots$}
\put(83.5,9.5){\vector(-2,-1){16}}
\put(83.5,-8.5){\vector(-2,1){16}}
\color{purple}
\put(86,-7){\vector(0,1){16.5}}
\put(84,9){\vector(0,-1){16.5}}

\qbezier(63,1)(60.5,3)(62,5.5)
\qbezier(62,5.5)(64.5,9)(67.5,5.5)
\qbezier(67.5,5.5)(68.5,3)(66.4,1)
\put(66.5,1.4){\vector(-1,-1){0.3}}
\qbezier(-1,1)(-3,3)(-2,5.5)
\qbezier(-2,5.5)(1,9)(4,5.5)
\qbezier(4,5.5)(5,3)(3,1)
\put(3.1,1.4){\vector(-1,-1){0.3}}
\qbezier(19,1)(17,3)(18,5.5)
\qbezier(18,5.5)(21,9)(24,5.5)
\qbezier(24,5.5)(25,3)(23,1)
\put(23.1,1.4){\vector(-1,-1){0.3}}

\put(-1,9.5){$\varepsilon_1$}
\put(19,9.5){$\varepsilon_2$}
\put(59,9.5){$\varepsilon_{n-2}$}
\put(79,-1){$\varepsilon_{n}$}
\put(87,-1){$\varepsilon_{n-1}$}
\end{picture}
\end{equation}
\vspace{.8cm}
\end{center}

\begin{center}\setlength{\unitlength}{0.8mm}
\vspace{-2.5cm}
\begin{equation}
 \begin{picture}(100,40)(0,20)
\put(10,6){\small${6}$}
\put(10,10){$\circ$}
\put(32,6){\small${5}$}
\put(32,10){$\circ$}

\put(10,30){$\circ$}
\put(10,33){\small${1}$}
\put(32,30){$\circ$}
\put(32,33){\small${2}$}

\put(31.5,11){\vector(-1,0){19}}
\put(31.5,31){\vector(-1,0){19}}

\put(52,22){\vector(-2,1){17.5}}
\put(52,20){\vector(-2,-1){17.5}}

\put(54.7,21.2){\vector(1,0){19}}

\put(52,20){$\circ$}
\put(52,16.5){\small$3$}
\put(74,20){$\circ$}
\put(74,16.5){\small$4$}
\color{purple}
\put(10,12.5){\vector(0,1){17}}
\put(12,29.5){\vector(0,-1){17}}
\put(32,12.5){\vector(0,1){17}}
\put(34,29.5){\vector(0,-1){17}}

\qbezier(52,22.5)(50,24)(51,26.5)
\qbezier(51,26.5)(53,29)(56,26.5)
\qbezier(56,26.5)(57.5,24)(55,22)
\put(55.1,22.4){\vector(-1,-1){0.3}}
\qbezier(74,22.5)(72,24)(73,26.5)
\qbezier(73,26.5)(75,29)(78,26.5)
\qbezier(78,26.5)(79,24)(77,22)
\put(77.1,22.4){\vector(-1,-1){0.3}}

\put(35,20){$\varepsilon_2$}
\put(27,20){$\varepsilon_5$}
\put(13,20){$\varepsilon_1$}
\put(5,20){$\varepsilon_6$}
\put(52,30){$\varepsilon_3$}
\put(73,30){$\varepsilon_4$}
\end{picture}
\end{equation}
\vspace{1cm}
\end{center}

\subsection{Isomorphisms $\widetilde{\psi}$ and $\psi$}

Recall $\I_\btau$ from \eqref{eqn:representative}.

\begin{theorem}
[\text{\cite[Proposition 7.5, Theorem 7.7]{LW19}}]
 \label{thm:Ui=iHall}
Let $(Q, \btau)$ be a Dynkin $\imath$quiver. Then there exists a $\Q({\sqq})$-algebra isomorphism
\begin{align}
   \label{eqn:psi morphism}
\widetilde{\psi}= \widetilde{\psi}_Q: \tUi_{|v={\sqq}} &\stackrel{\simeq}{\longrightarrow} \tMH,
\end{align}
which sends
\begin{align}
B_j \mapsto \frac{-1}{q-1}[S_{j}],\text{ if } j\in\ci,
&\qquad\qquad B_{j} \mapsto \frac{{\sqq}}{q-1}[S_{j}],\text{ if }j\notin \ci,
  \label{eq:split}
\\
\tk_i \mapsto - q^{-1}[\E_i], \text{ if }\btau i=i,
&\qquad\qquad
\tk_i \mapsto [\E_i],\quad \text{ if }\btau i\neq i.
  \label{eq:extra}
\end{align}
Moreover, it induces an isomorphism $\psi: \Ui_{|v={\sqq}}\stackrel{\simeq}{\longrightarrow} \rMH$, which acts on $B_i$ as in \eqref{eq:split}--\eqref{eq:extra} and sends $k_i \mapsto  \vs_i^{-1}[\E_i], \text{ for } i \in \ci$.
\end{theorem}

Theorem~\ref{thm:Ui=iHall} admits the following generic version.

\begin{theorem}
[\text{\cite[Theorem 9.8]{LW19}}]
   \label{generic U MRH}
Let $(Q, \btau)$ be a Dynkin $\imath$quiver. Then we have a $\Q(\bv)$-algebra isomorphism
\begin{align}
\widetilde{\psi}: \tUi\stackrel{\sim}{\longrightarrow}& \tMHg,
 \label{eq:psi2} \\
B_j\mapsto \frac{-1}{v^2-1}\fu_{\alpha_j},\text{ if } j\in\ci, &\qquad\qquad
B_{j}\mapsto\frac{\bv}{v^2-1}\fu_{\alpha_j},\text{ if }j\notin \ci,
 \label{eq:psiB} \\
\tk_i\mapsto -v^{-2}\fu_{\upgamma_i}, \text{ if }\btau i=i,
&\qquad\qquad
\tk_i\mapsto \fu_{\upgamma_i}, \text{ if }\btau i\neq i,
 \notag
\end{align}
This further induces a $\Q(\bv)$-algebra isomorphism
\begin{align}
\psi: \Ui\stackrel{\simeq}{\longrightarrow}& \rMHg,
 \label{eq:psi}
 \end{align}
 which acts on $B_j$ as in \eqref{eq:psiB} and sends $k_i\mapsto \frac{\fu_{\upgamma_i}}{\vs_i}, \text{ for }i\in\ci.$
\end{theorem}

\section{Reflection functors and APR-tilting  for $\imath$quivers}
\label{sec:RF}

In this section, we study representations of modulated graphs associated to $\imath$quivers. We then define the BGP type reflection functors as well as APR-tilting modules for $\imath$quivers.

\subsection{Modulated graphs for $\imath$quivers}
  \label{subsec:modulated}

We generalize the notion of representation theory of \emph{modulated graphs} to the setting of $\imath$quivers. See  \cite{DR74, Li12, GLS17} for details about representations of modulated graphs.

\subsubsection{} 

Define
\[
\Omega:=\Omega(Q) =\{(i,j) \in Q_0\times Q_0\mid  \exists (\alpha:i\rightarrow j)\in Q_1\}.
\]
Then $\Omega$ represents the orientation of $Q$. Since $Q$ is acyclic, if $(i,j)\in\Omega$, then $(j,i)\notin \Omega$.
We also use $\Omega(i,-)$ to denote the subset $\{j\in Q_0 \mid \exists (\alpha:i\rightarrow j)\in Q_1\}$, and $\Omega(-,i)$ is defined similarly.

Recall $\BH_i$ from \eqref{dfn:Hi}.  For any $(i,j)\in \Omega$, we define
\begin{equation}  \label{eq:jHi}
{}{}_j\BH_i := \BH _j\Span_k\{\alpha,\btau\alpha\mid(\alpha:i\rightarrow j)\in Q_1\text{ or } (\alpha:i\rightarrow \btau j)\in Q_1\}\BH_i.
\end{equation}
Note that ${}{}_j\BH_i ={}_{\btau j} \BH _{\btau i}={}_{j} \BH _{\btau i}={}_{\btau j} \BH _{i}$ for any $(i,j)\in \Omega$.

We describe a $k$-basis of ${}_j\BH_i $ for each $(i,j)\in\Omega$ by separating into 2 cases (i)-(ii):
\begin{itemize}
\item[(i)] $\btau i=i$ and $\btau j=j$. Then $\{\alpha,\alpha\varepsilon_i\mid (\alpha:i\rightarrow j)\in Q_1\}$ forms a basis of ${}_j\BH_i $.

\item[(ii)]  $\btau i\neq i$ or $\btau j\neq j$. Then
\begin{align*}
&\{\alpha,\btau\alpha,\varepsilon_j\alpha =\btau\alpha \varepsilon_i, \varepsilon_{\btau j} \btau\alpha =\alpha \varepsilon_{\btau i}\mid (\alpha:i\rightarrow j)\in Q_1\}
\\ \cup &
\{\alpha,\btau\alpha,\varepsilon_{\btau j}\alpha =\btau\alpha \varepsilon_i, \varepsilon_{j} \btau\alpha =\alpha \varepsilon_{\btau i}\mid (\alpha:i\rightarrow \btau j)\in Q_1\}
\end{align*}
forms a basis of ${}_j\BH_i $.
\end{itemize}
Hence ${}_j\BH_i $ is an $\BH _j\mbox{-}\BH _i$-bimodule, which is free as a left $\BH _j$-module (and respectively, right $\BH _i$-module), with a basis $_j\LL_i$ (and respectively, $_j\RR_i$) defined in \cite[(2.12), (2.13)]{LW19}:
\begin{eqnarray}\label{basis of Hij left}
_j\LL_i&=&\left\{ \begin{array}{cc}
\{\alpha | (\alpha:i\rightarrow j)\in Q_1\} & \text{ if }\btau i=i, \btau j=j,\\
\{\alpha+\btau\alpha | (\alpha:i\rightarrow j)\in Q_1\} & \text{ if }\btau i=i, \btau j\neq j,\\
\{\alpha,\btau \alpha | (\alpha:i\rightarrow j)\in Q_1\} & \text{ if }\btau i\neq i,\btau j=j,\\
\{\alpha+\btau \alpha | (\alpha:i\rightarrow j) \text{ or }(\alpha:i\rightarrow \btau j)\in Q_1\} & \text{ if }\btau i\neq i,\btau j\neq j;\label{eqn:basis of L}
\end{array}\right.
\end{eqnarray}
\begin{eqnarray}
_j\RR_i&=& \left\{ \begin{array}{cc}\label{basis of Hij right}
\{\alpha | (\alpha:i\rightarrow j)\in Q_1\} & \text{ if }\btau i=i,  \btau j=j,\\
\{\alpha,\btau \alpha | (\alpha:i\rightarrow j)\in Q_1\} & \text{ if }\btau i=i, \btau j\neq j,\\
\{\alpha+\btau \alpha | (\alpha:i\rightarrow j)\in Q_1\} & \text{ if }\btau i\neq i,\btau j=j,\\
\{\alpha+\btau \alpha | (\alpha:i\rightarrow j) \text{ or }(\alpha:i\rightarrow \btau j)\in Q_1\} & \text{ if }\btau i\neq i,\btau j\neq j.\end{array}\right. \label{eqn:basis of R}
\end{eqnarray}
Note that $\btau(_j\LL_i)= {}_j\LL_i$ and $\btau({}_j\RR_i)= {}_j\RR_i$.

\subsubsection{}

Denote
\begin{equation}  \label{eq:ovOmega}
\overline{\Omega}:=\{(i,j)\in \ci \times \ci \mid (i,j)\in\Omega\text{ or }(i,\btau j)\in\Omega\}.
\end{equation}
Recall that $\ci $ is a (fixed) subset of $Q_0$ formed by the representatives of all $\btau$-orbits.
The tuple $(\BH_i,\,_j\BH_i):=(\BH_i,\,_j\BH_i)_{i\in\ci ,(i,j)\in\ov{\Omega}}$  is called a \emph{modulation} of $(Q, \btau)$ and is denoted by $\cm(Q, \btau)$.

A representation $(N_i,N_{ji}):=(N_i,N_{ji})_{i\in\ci ,(i,j)\in\ov{\Omega}}$ of $\cm(Q, \btau)$ is defined to be a finite-dimensional $\BH_i$-module $N_i$ for each $i\in \ci $ and an $\BH_j$-morphism
$N_{ji}:\,_j\BH_i\otimes_{\BH_i} N_i\rightarrow N_j$ for each $(i,j)\in \overline{\Omega}$. A morphism $f:L\rightarrow N$ of representations $L=(L_i,L_{ji})$ and $N=(N_i,N_{ji})$ of $\cm(Q, \btau)$ is a tuple $f=(f_i)_i$ of $\BH_i$-morphisms $f_i:L_i\rightarrow N_i$ for $i\in \ci $ such that the following diagram commutes for each $(i,j)\in\overline{\Omega}$:
\[
\xymatrix{_j\BH_i\otimes_{\BH_i} L_i \ar[r]^{1\otimes f_i} \ar[d]^{L_{ji}}&  _j\BH_i\otimes_{\BH_i} N_i\ar[d]^{N_{ji}}\\
 L_j\ar[r]^{f_j} & N_j}
 \]

The representations of $\cm(Q, \btau)$ form an abelian category, which is denoted by $\rep(\cm(Q, \btau))$.
Similarly to \cite[Proposition 5.1]{GLS17} (see also \cite[Theorem 3.2]{Li12}), we have the following.

\begin{proposition}
[\text{\cite[Proposition 2.16]{LW19}}]
\label{prop:modulated representation}
The categories $\rep(\cm(Q, \btau))$ and $\rep(\ov{Q},\ov{I})$ are isomorphic.
\end{proposition}

\subsubsection{ }

The materials in this subsection are inspired by \cite{GLS17}  and will be used in \S\ref{subsec:APR}--\ref{subsec:APR2} to define reflection functors and APR-tilting modules for $\imath$quivers.

Let $Q^*$ be the quiver constructed from $Q$ by reversing all the arrows. We have $(i,j)\in\Omega$ if and only if $(j,i)\in \Omega^*:=\Omega(Q^*)$. For any $\alpha:i\rightarrow j$ in $Q$, denote by $\widetilde{\alpha}:j\rightarrow i$ the corresponding arrow in $Q^*$. In general, for any $b\in kQ$, we denote by $\widetilde{b}$ the corresponding element in $kQ^*$. Then $\btau$ induces an involution $\btau'$ of $Q^*$. Clearly, $\btau' i=\btau i$ for any vertex $i\in Q_0$. Then similarly we can define $\Lambda^*=\K Q^*\otimes_\K R_2$, and an involution $\btau^{'\sharp}$ for $\Lambda^*$, and its $\btau^{'\sharp}$-fixed point subalgebra $(\Lambda^*)^\imath$. Note that $\BH$ is also a subalgebra of $(\Lambda^*)^\imath$.

It is worth noting that $\ci$ is also a subset of representatives of $\btau'$-orbits. In this way, one can define $\ov{\Omega}^*$ (cf. \eqref{eq:ovOmega} for $\ov{\Omega}$).

For any $(i,j)\in\Omega$, we have $(j,i)\in\Omega^*$, and then we can define ${}_i\BH_j$ as in \eqref{eq:jHi} by considering $Q^*$. Recall from \eqref{basis of Hij left}--\eqref{basis of Hij right} the basis $_i\LL_j$ (and respectively, $_i\RR_j$) for ${}_i\BH_j $ as a left $\BH _i$-module (and respectively, right $\BH _j$-module). Note that $b\in {}_j\LL_i$ (respectively, $b\in {}_j\RR_i$) if and only if $\widetilde{b}\in {}_i\RR_j$ (respectively, $\widetilde{b}\in {}_i\LL_j$).
Let $_j\LL_i^*$ and $_j\RR_i^*$ be the dual bases of $\Hom_{\BH_j}(_j \BH_i, \BH_j)$ and $\Hom_{\BH_i}(_j\BH_i,\BH_i)$, respectively. Denote by $b^*$ the corresponding dual basis vector for any $b\in{_j\LL_i}$ or $b\in{_j\RR_i}$.

Since ${}_i\BH_j$ and $\Hom_{\BH_j}(_j\BH_i,\BH_j)$ are right free $\BH_j$-modules with bases given by ${}_i\RR_j$ and ${}_j\LL_i^*$ respectively, there is a right $\BH_j$-module isomorphism
\[
\rho:\, _i\BH_j \longrightarrow \Hom_{\BH_j}(_j\BH_i,\BH_j)
\]
such that $\rho(\widetilde{b})=b^*$ for any $b\in{}_j\LL_i$.
It is then routine to check that $\rho$ is actually an $\BH_i$-$\BH_j$-bimodule isomorphism.
Similarly, there is an $\BH_i$-$\BH_j$-bimodule isomorphism
$$\lambda: \, _i\BH_j\longrightarrow \Hom_{\BH_i}(_j\BH_i,\BH_i).$$
These two isomorphisms satisfy that $\rho(_i\RR_j)={}_j\LL_i^*$ and $\lambda(_i\LL_j)={}_j\RR_i^*$. We sometimes identify the spaces
$\Hom_{\BH_j}(_j\BH_i,\BH_j)$, $_i\BH_j$ and $\Hom_{\BH_i}(_j\BH_i,\BH_i)$ via $\rho$ and $\lambda$.

If $N_j$ is an $\BH_j$-module, then we have a natural isomorphism of $\BH_i$-modules
$$\Hom_{\BH_j}(_j\BH_i,N_j)\longrightarrow\, _i\BH_j\otimes_{\BH_j}N_j$$
defined by
$$f\mapsto \sum_{b\in_j\LL_i} b^*\otimes f(b).$$
Furthermore, for any $\BH_i$-module $L_i$, there is a natural isomorphism of $\K$-vector spaces:
$$\Hom_{\BH_j}(_j\BH_i\otimes_{\BH_i}L_i, N_j)\longrightarrow \Hom_{\BH_i}(L_i,\Hom_{\BH_j}(_j\BH_i,N_j)).$$
Composing the two maps above, we obtain the following.

\begin{lemma}
  \label{lem:ad}
There exists a canonical $k$-linear isomorphism
\begin{align*}
{\rm ad}_{ji}  ={\rm ad}_{ji}(L_i,N_j)  : & \Hom_{\BH_j}(_j\BH_i\otimes_{\BH_i} L_i,N_j)\longrightarrow \Hom_{\BH_i}(L_i,\,_i\BH_j\otimes_{\BH_j}N_j)
\\
{\rm ad}_{ji}: & f\mapsto \big(f^\vee:l\mapsto \sum_{b\in_j\LL_i} b^*\otimes f(b\otimes l) \big).
\end{align*}
The inverse ${\rm ad}_{ji}^{-1}$ is given by
$
{\rm ad}_{ji}^{-1} (g) = \big(g^\vee:h\otimes l\mapsto \sum_{b\in_j\LL_i} b^*(h)(g(l))_b \big),
$ 
where the elements $(g(l))_b\in N_j$ are uniquely determined by
$g(l)=\sum_{b\in_j\LL_i} b^*\otimes(g(l))_b.$
\end{lemma}


%
%
\subsection{Reflection functors}
  \label{subsec:APR}

In this subsection, we shall introduce the reflection functors in the setting of $\imath$quivers.

Let $(Q, \btau)$ be an acyclic $\imath$quiver. Without loss of generality, we assume $Q$ to be connected and $(Q,\btau)$ of rank $\ge 2$. Recall that $\Omega=\Omega(Q)$ is the orientation of $Q$. For any sink $\ell \in Q_0$, define the quiver $s_\ell (Q)$ by reversing all the arrows ending at $\ell $.
By definition, $\ell $ is a sink of $Q$ if and only if $\btau \ell $ is a sink of $Q$.  Define the quiver
\begin{align}  \label{eq:QQ}
Q' = \bs_\ell Q =\left\{ \begin{array}{cc} s_\ell (Q) & \text{ if } \btau \ell =\ell ,\\ s_{\ell }s_{\btau\ell  }(Q) &\text{ if }\btau \ell \neq \ell .  \end{array}\right.
\end{align}
Note that $s_\ell s_{\btau \ell }(Q)=s_{\btau \ell }s_\ell (Q)$.
Then $\btau$ also induces an involution $\btau'$ on the quiver $Q'$. In this way, we can define $\Lambda'=\K Q'\otimes_k R_2$ with an involution $\btau'^\sharp$, and denote the $\btau'^\sharp$-fixed point subalgebra by $\Lambda'^\imath =\bs_{\ell }\Lambda^{\imath}$. Note that
$\bs_{\ell }\Lambda^{\imath}=\bs_{\btau \ell }\Lambda^{\imath}$
for any sink $\ell \in Q_0$.
The quiver $\ov{Q'}$ of $\bs_\ell \Lambda^{\imath}$ can be constructed from $\ov{Q}$ by reversing all the arrows in $Q$ ending at $\ell $ and $\btau \ell $. Denote by $\Omega':=\Omega(Q')$ the orientation of $Q'$.

We shall define a {\rm reflection functor} associated to a sink $\ell \in Q_0$ (compare with \cite{GLS17})
\begin{align} \label{def:RefF}
F_\ell ^+:\rep(\ov{Q},\ov{I})=\mod(\Lambda^{\imath})\longrightarrow \rep(\ov{Q'},\ov{I'})=\mod(\bs_\ell \iLa),
\end{align}
in \eqref{eq:F+} below.
Using Proposition \ref{prop:modulated representation}, we shall identify the category $\rep(\ov{Q},\ov{I})$ with $\rep(\cm(Q, \btau))$, and respectively, $\rep(\ov{Q'},\ov{I'})$ with $\rep(\cm(Q',\btau'))$.

Without loss of generality, we assume that the sink $\ell \in\ci$; cf. \eqref{eq:ci} for notation $\ci$.
Let $L=(L_i,L_{ji})\in \rep(\cm(Q, \btau))$. Then $L_{ji}:\,_j\BH_i\otimes_{\BH_i} L_i\rightarrow L_j$ is an $\BH_i$-morphism for any $(i,j)\in\ov{\Omega}$. Denote by
\begin{align}
\label{eq:Lin}
L_{\ell ,{\rm in}}:= (L_{\ell i})_i:\bigoplus_{i\in\ov{\Omega}(-,\ell )}  \,_\ell \BH_i\otimes_{\BH_i} L_i\longrightarrow L_\ell.
\end{align}

Let $N_\ell :=\ker(L_{\ell ,{\rm in}})$. By definition, there exists an exact sequence
\begin{align}
\label{def:reflection}
0\longrightarrow N_\ell \longrightarrow \bigoplus_{i\in\ov{\Omega}(-,\ell )}  \,_\ell \BH_i\otimes_{\BH_i} L_i\xrightarrow{L_{\ell ,{\rm in}}} L_\ell .
\end{align}
Denote by
$(N_{i\ell }^\vee)_i$ the inclusion map $N_\ell \rightarrow \bigoplus_{i\in\ov{\Omega}(-,\ell )}\,_\ell \BH_i\otimes_{\BH_i} L_i $.

For any $L\in \rep(\cm(Q, \btau))$, define
\begin{align}
  \label{eq:F+}
F_\ell ^+(L)=(N_s,N_{rs})\in \rep(\cm(Q',\btau')),
\end{align}
where
\[
N_s:=\left\{ \begin{array}{cc} L_s & \text{ if }s\neq \ell ,
\\
N_{\ell } &\text{ if }s=\ell ,
\end{array} \right.
\qquad
 N_{rs}:=\left\{ \begin{array}{cc} L_{rs} &\text{ if } (s,r)\in \ov{\Omega} \text{ with }r\neq \ell ,
\\{\rm ad}_{r\ell }^{-1}(N_{r\ell }^\vee) & \text{ if } (s,r)\in\ov{\Omega}^* \text{ and } s=\ell.
  \end{array}  \right.
  \]
For any $f=(f_i)_{i\in \mathbb{I}_\btau}:L=(L_i,L_{ji})\rightarrow M=(M_i,M_{ji})$, define $F_\ell^+(f)=(g_i)_{i\in\I_\btau}$ where
$g_i=f_i$ if $i\neq \ell$, and $g_\ell$ is the unique morphism such that the following diagram commutes:
\[\xymatrix{0\ar[r]& F_\ell^+(L)_\ell \ar[r] \ar[d]^{g_\ell} &\bigoplus\limits_{i\in\ov{\Omega}(-,\ell )}  \,_\ell \BH_i\otimes_{\BH_i} L_i\ar[r]^{\qquad \quad L_{\ell,{\rm in}}}\ar[d]^{(1\otimes f_i)_i} & L_\ell\ar[d]^{f_\ell}
\\
 0\ar[r]& F_\ell^+(M)_\ell \ar[r]&\bigoplus\limits_{i\in\ov{\Omega}(-,\ell )}  \,_\ell \BH_i\otimes_{\BH_i} M_i\ar[r]^{\quad\qquad M_{\ell,{\rm in}}} & M_\ell}\]
It is routine to verify that $(g_i)_{i\in \mathbb{I}_\btau}$ is a morphism by Lemma 3.2 and the definition of morphisms of representations of modulated graphs (see \S3.1.2).
The above construction is clearly functorial. Hence we obtain a reflection functor  $F_\ell ^+:\mod(\Lambda^{\imath}) \rightarrow\mod(\bs_\ell \Lambda^{\imath}).$

Dually, associated to any source $\ell \in Q_0$, we have a reflection functor
\begin{align}
F_\ell ^-:\mod(\Lambda^{\imath}) \longrightarrow\mod(\bs_\ell \Lambda^{\imath}).
\end{align}

%
%

%
\subsection{AR quivers for rank 2 $\imath$quiver algebras}

We describe explicitly the Auslander-Reiten (AR) quivers for the rank 2 $\imath$quivers, including  $\imath$quiver of diagonal type (which gives rise to $\tU$ of type $A_2$), split $\imath$quiver (which gives rise to split $\tUi$ of type $A_2$), and $\imath$quiver with $\btau \neq \Id$ (which gives rise to quasi-split $\tUi$ of type $A_3$). We refer to \cite[Chapter IV]{ASS04} for the basics of AR theory, and \cite{BG82} for the coverings of the AR quiver
of a representation-finite algebra introduced by Riedtmann.
\subsubsection{$\imath$Quiver of diagonal type}

Let $Q=(Q_0,Q_1)$ be an acyclic quiver, $Q^{\sharp}$ be its double framed quiver. Let $Q^{\dbl} =Q\sqcup  Q^{\diamond}$,  where $Q^{\diamond}$ is an identical copy of $Q$ with a vertex set $\{i^{\diamond} \mid i\in Q_0\}$ and an arrow set $\{ \alpha^{\diamond} \mid \alpha \in Q_1\}$.
Let $\Lambda=kQ\otimes_k R_2$, $\Lambda^{\diamond} =k Q^{\diamond} \otimes_k R_2$, and $\Lambda^{\dbl}:=kQ^{\dbl}\otimes R_2$. Then the double framed quiver $(Q^{\dbl})^{\sharp}$ of $Q^{\dbl}$ is $Q^{\sharp}\sqcup (Q^{\sharp})^{\diamond}$, and
$\Lambda^{\dbl}=\Lambda\times \Lambda^{\diamond} \cong \Lambda\times \Lambda$.

Let $\rm{swap}$ be the involution of $Q^{\rm dbl}$ uniquely determined by
$\swa(i)=i^\diamond$ for any $i\in Q_0$ (by viewing $Q$ and $Q^\diamond$ as subquivers of $Q^{\dbl}$). Then $(\Lambda^{\dbl})^\imath$ is isomorphic to $\Lambda$. Explicitly, let $(\ov{Q}^{\dbl},\ov{I}^{\dbl})$ be the bound quiver of $(\Lambda^{\dbl})^\imath$. Then $(\ov{Q}^{\dbl},\ov{I}^{\dbl})$ coincides with the double $\imath$quiver $(Q^{\sharp},I^{\sharp})$. So we just use $(Q^{\sharp},I^{\sharp})$ as the bound quiver of $(\Lambda^{\dbl})^\imath$ and identify $(\Lambda^{\dbl})^\imath$ with $\Lambda$.

In this subsection, we give the Auslander-Reiten quiver of $\Lambda$ for $Q$ of type $A_2$.

Let $Q=(\xymatrix{ 1\ar[r]^{\alpha} &2})$. Then its double framed quiver $Q^{\sharp}$ is
\begin{center}\setlength{\unitlength}{0.7mm}
\vspace{-2cm}
\begin{equation*}
\begin{picture}(100,40)(0,20)
\put(49,8){\small $1'$}
\put(50,31){\small $1$}
\put(72,8){\small $2'$}
\put(72,31){\small $2$}

\put(53,10){\vector(1,0){18.5}}
\put(53,32.5){\vector(1,0){18.5}}

\put(60,12.5){$_{\alpha'}$}
\put(60,35){$_\alpha$}
\color{purple}
\put(50,13){\vector(0,1){17}}
\put(52,29.5){\vector(0,-1){17}}
\put(72,13){\vector(0,1){17}}
\put(74,29.5){\vector(0,-1){17}}

\put(45,20){\small $\varepsilon_1'$}
\put(53,20){\small $\varepsilon_1$}
\put(67,20){\small $\varepsilon_2'$}
\put(75,20){\small $\varepsilon_2$}
\end{picture}
\end{equation*}
\vspace{-0.2cm}
\end{center}
and $I^{\sharp}$ is generated by
\begin{eqnarray*}
&&\varepsilon_1\varepsilon_1', \,\,\,\,\varepsilon_1'\varepsilon_1, \,\,\,\,\varepsilon_2'\varepsilon_2, \,\,\,\,\varepsilon_2\varepsilon_2',\quad\alpha' \varepsilon_1 -\varepsilon_2 \alpha,\,\,\,\, \alpha \varepsilon_1'- \varepsilon_2' \alpha'.
\end{eqnarray*}
The Auslander-Reiten quiver of $\Lambda$ is displayed in Figure 1. 
The modules on the leftmost column of the figure have to be identified with the corresponding modules on the rightmost one. The projective $\Lambda$-modules are marked with a solid frame.
\begin{center}\setlength{\unitlength}{0.6mm}
\begin{figure}\label{figure 1}
 \begin{picture}(200,100)(0,0)
 \put(0,45){\begin{picture}(10,10)\setlength{\unitlength}{0.5mm}
 \put(5,10){\tiny$1$}
 \put(0,5){\tiny$1'$}
 \put(10,5){\tiny$2$}
 \put(5,0){\tiny$2'$}
 \end{picture}
 }
 \put(12,47){\vector(2,-1){8}}

 \put(-1,44){\begin{picture}(10,20)\setlength{\unitlength}{0.6mm}
\put(0,0){\line(1,0){13}}
\put(0,0){\line(0,1){13}}
\put(13,0){\line(0,1){13}}
\put(0,13){\line(1,0){13}}
\end{picture}}

 \put(5,25){\tiny$1'$}
 \put(9,27){\vector(1,1){10}}

\put(5,75){\tiny$2$}
\put(7,73){\vector(1,-2){14}}
\put(7,79){\vector(1,1){14}}

\put(20,35){\begin{picture}(10,10)\setlength{\unitlength}{0.5mm}
\put(5,10){\tiny$1$}
 \put(0,5){\tiny$1'$}
 \put(10,5){\tiny$2$}
 \end{picture}
 }

 \put(19,90){\begin{picture}(10,10)\setlength{\unitlength}{0.5mm}
 \put(10,5){\tiny$2'$}
 \put(5,0){\tiny$2$}
 \end{picture}
 }
  \put(22,89){\begin{picture}(10,20)\setlength{\unitlength}{0.45mm}
\put(0,0){\line(1,0){13}}
\put(0,0){\line(0,1){13}}
\put(13,0){\line(0,1){13}}
\put(0,13){\line(1,0){13}}
\end{picture}}
\put(33,91){\vector(1,-1){11}}

\put(31,37){\vector(1,-1){11}}

\put(30,44){\vector(1,2){14}}

 \put(50,26){\vector(1,1){12}}

 \put(40,22.5){\begin{picture}(10,10)\setlength{\unitlength}{0.5mm}
 \put(10,5){\tiny$1$}
 \put(5,0){\tiny$2$}
 \end{picture}
 }

 \put(40,70){\begin{picture}(10,10)\setlength{\unitlength}{0.5mm}
 \put(5,10){\tiny$1$}
 \put(0,5){\tiny$1'$}
 \put(10,5){\tiny$2$}
 \put(15,10){\tiny$2'$}
 \end{picture}
 }

 \put(54,82){\vector(1,1){9}}

 \put(60,90){\begin{picture}(10,10)\setlength{\unitlength}{0.5mm}
 \put(10,5){\tiny$1$}
 \put(5,0){\tiny$1'$}
 \end{picture}
 }

  \put(73,91){\vector(1,-1){14}}

 \put(62,35){\begin{picture}(10,10)\setlength{\unitlength}{0.5mm}
\put(5,5){\tiny$2$}
 \put(0,10){\tiny$1$}
 \put(10,10){\tiny$2'$}
 \end{picture}
 }

\put(89,25){\tiny$2'$}
\put(95,23){\vector(1,-1){12}}
\put(95,27){\vector(1,1){11}}

\put(89,75){\tiny$1$}
\put(94,73){\vector(1,-2){14}}

 \put(74,38){\vector(1,-1){11}}

 \put(74,45){\vector(1,2){14}}
 \put(74,43){\vector(1,1){8}}

 \put(84,45){\begin{picture}(10,10)\setlength{\unitlength}{0.5mm}
 \put(5,10){\tiny$1'$}
 \put(0,5){\tiny$1$}
 \put(10,5){\tiny$2'$}
 \put(5,0){\tiny$2$}
 \end{picture}
 }
\put(97,50){\vector(1,-1){9}}

  \put(84,44){\begin{picture}(10,20)\setlength{\unitlength}{0.6mm}
\put(0,0){\line(1,0){13}}
\put(0,0){\line(0,1){13}}
\put(13,0){\line(0,1){13}}
\put(0,13){\line(1,0){13}}
\end{picture}}

  \put(107,35){\begin{picture}(10,10)\setlength{\unitlength}{0.5mm}
\put(5,10){\tiny$1'$}
 \put(0,5){\tiny$1$}
 \put(10,5){\tiny$2'$}
 \end{picture}
 }
 \put(117,38){\vector(1,-1){12}}
 \put(117,45){\vector(1,2){13}}

   \put(110,5){\begin{picture}(10,10)\setlength{\unitlength}{0.5mm}
\put(5,10){\tiny$2$}
 \put(0,5){\tiny$2'$}
 \end{picture}
 }

   \put(108.5,7.5){\begin{picture}(10,20)\setlength{\unitlength}{0.45mm}
\put(0,0){\line(1,0){13}}
\put(0,0){\line(0,1){13}}
\put(13,0){\line(0,1){13}}
\put(0,13){\line(1,0){13}}
\end{picture}}

\put(118,11){\vector(1,1){11}}

  \put(130,20){\begin{picture}(10,10)\setlength{\unitlength}{0.5mm}
 \put(5,10){\tiny$1'$}
 \put(0,5){\tiny$1$}
 \put(10,5){\tiny$2'$}
 \put(15,10){\tiny$2$}
 \end{picture}
 }
  \put(142,24){\vector(1,-1){11}}
 \put(142,31){\vector(1,1){9}}

  \put(128,72.5){\begin{picture}(10,10)\setlength{\unitlength}{0.5mm}
 \put(10,5){\tiny$1'$}
 \put(5,0){\tiny$2'$}
 \end{picture}
 }

 \put(138,71){\vector(1,-2){12}}

    \put(152.5,5){\begin{picture}(10,10)\setlength{\unitlength}{0.5mm}
\put(5,10){\tiny$1'$}
 \put(0,5){\tiny$1$}
 \end{picture}
 }
 \put(161,12){\vector(1,1){11}}

   \put(150,35){\begin{picture}(10,10)\setlength{\unitlength}{0.5mm}
\put(5,5){\tiny$2'$}
 \put(0,10){\tiny$1'$}
 \put(10,10){\tiny$2$}
 \end{picture}
 }
 \put(160,43){\vector(2,1){8}}

 \put(160,39){\vector(1,-1){12}}
 \put(161,47){\vector(1,2){13}}

  \put(170,45){\begin{picture}(10,10)\setlength{\unitlength}{0.5mm}
 \put(5,10){\tiny$1$}
 \put(0,5){\tiny$1'$}
 \put(10,5){\tiny$2$}
 \put(5,0){\tiny$2'$}
 \end{picture}
 }
 \put(169,44){\begin{picture}(10,20)\setlength{\unitlength}{0.6mm}
\put(0,0){\line(1,0){13}}
\put(0,0){\line(0,1){13}}
\put(13,0){\line(0,1){13}}
\put(0,13){\line(1,0){13}}
\end{picture}}
\put(50,72){\vector(1,-2){13}}

 \put(175,25){\tiny$1'$}

\put(175,75){\tiny$2$}

\dashline{3}(34,93)(62,93)
\put(34,93){\vector(-1,0){1}}

\put(10,21){\begin{picture}(0,0)
\dashline{3}(-1,4)(32,4)
\put(-1,4){\vector(-1,0){1}}
\end{picture}
}

\put(8,74){\begin{picture}(0,0)
\dashline{3}(2,3)(32,3)
\put(2,3){\vector(-1,0){1}}
\end{picture}
}

\put(52,73){\begin{picture}(0,0)
\dashline{3}(2,3)(32,3)
\put(2,3){\vector(-1,0){1}}
\end{picture}
}

\put(95,73){\begin{picture}(0,0)
\dashline{3}(2,3)(34,3)
\put(2,3){\vector(-1,0){1}}
\end{picture}
}

\put(138,73){\begin{picture}(0,0)
\dashline{3}(2,3)(34,3)
\put(2,3){\vector(-1,0){1}}
\end{picture}
}

\put(140,23){\begin{picture}(0,0)
\dashline{3}(4,3)(32,3)
\put(4,3){\vector(-1,0){1}}
\end{picture}
}

\put(94,23){\begin{picture}(0,0)
\dashline{3}(2,3)(32,3)
\put(2,3){\vector(-1,0){1}}
\end{picture}
}

\put(52,22){\begin{picture}(0,0)
\dashline{3}(0,3)(32,3)
\put(0,3){\vector(-1,0){1}}
\end{picture}
}

\put(72,38){\begin{picture}(2,0)
\dashline{3}(2,3)(32,3)
\put(2,3){\vector(-1,0){1}}
\end{picture}
}

\put(30,38){\begin{picture}(2,0)
\dashline{3}(2,3)(32,3)
\put(2,3){\vector(-1,0){1}}
\end{picture}
}

\put(118,38){\begin{picture}(2,0)
\dashline{3}(2,3)(32,3)
\put(2,3){\vector(-1,0){1}}
\end{picture}
}

\put(118,8){\begin{picture}(2,0)
\dashline{3}(2,3)(32,3)
\put(2,3){\vector(-1,0){1}}
\end{picture}
}

\put(143,38){\begin{picture}(2,0)
\dashline{3}(17,3)(32,3)
\put(17,3){\vector(-1,0){1}}
\end{picture}
}

\put(-13,38){\begin{picture}(2,0)
\dashline{3}(17,3)(32,3)
\put(17,3){\vector(-1,0){1}}
\end{picture}
}

\end{picture}
\caption{Auslander-Reiten quiver for $\Lambda$ with $Q$ of type $A_2$}
\end{figure}
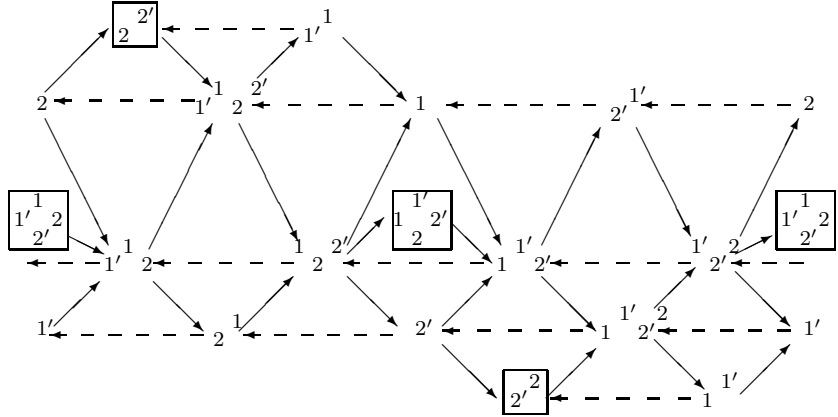
\vspace{-0.5cm}
\end{center}

\subsubsection{Split $\imath$quiver}

Let $Q=(\xymatrix{ 1\ar[r]^{\alpha} &2})$ with $\btau=\Id$. Then $\Lambda^\imath$ is isomorphic to the algebra with its quiver $\ov Q$ and relations as follows:
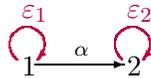
\begin{center}\setlength{\unitlength}{0.7mm}
 \begin{picture}(30,13)(0,0)
\put(0,-2){\small $1$}
\put(2.5,0){\vector(1,0){17}}
\put(10,-0.5){$^{\alpha}$}
\put(20,-2){\small $2$}
\color{purple}
\put(0,9){\small $\varepsilon_1$}
\put(20,9){\small $\varepsilon_2$}

\qbezier(-1,1)(-3,3)(-2,5.5)
\qbezier(-2,5.5)(1,9)(4,5.5)
\qbezier(4,5.5)(5,3)(3,1)
\put(3.1,1.4){\vector(-1,-1){0.3}}

\qbezier(19,1)(17,3)(18,5.5)
\qbezier(18,5.5)(21,9)(24,5.5)
\qbezier(24,5.5)(25,3)(23,1)
\put(23.1,1.4){\vector(-1,-1){0.3}}
\end{picture}
\vspace{0.2cm}
\end{center}
\[
\varepsilon_1^2=0=\varepsilon_2^2, \quad \varepsilon_2 \alpha=\alpha\varepsilon_1.
\]
Indeed, $\Lambda^\imath$ is isomorphic to $\K Q\otimes_\K R_1$.
Its Auslander-Reiten quiver of $\Lambda^\imath$  is displayed in Figure 2; see \cite[\S13.5]{GLS17}. 
The modules on the leftmost column of the figure have to be identified with the corresponding modules on the rightmost one. The projective $\Lambda^\imath$-modules are marked with a solid frame.

\begin{center}\setlength{\unitlength}{0.6mm}
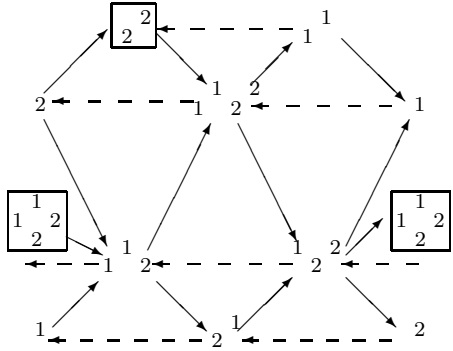
\begin{figure}\label{figure 2}
 \begin{picture}(100,80)(0,20)
 \put(0,45){\begin{picture}(10,10)\setlength{\unitlength}{0.5mm}
 \put(5,10){\tiny$1$}
 \put(0,5){\tiny$1 $}
 \put(10,5){\tiny$2$}
 \put(5,0){\tiny$2 $}
 \end{picture}
 }
 \put(12,47){\vector(2,-1){8}}

 \put(-1,44){\begin{picture}(10,20)\setlength{\unitlength}{0.6mm}
\put(0,0){\line(1,0){13}}
\put(0,0){\line(0,1){13}}
\put(13,0){\line(0,1){13}}
\put(0,13){\line(1,0){13}}
\end{picture}}

 \put(5,25){\tiny$1 $}
 \put(9,27){\vector(1,1){10}}

\put(5,75){\tiny$2$}
\put(7,73){\vector(1,-2){14}}
\put(7,79){\vector(1,1){14}}

\put(20,35){\begin{picture}(10,10)\setlength{\unitlength}{0.5mm}
\put(5,10){\tiny$1$}
 \put(0,5){\tiny$1 $}
 \put(10,5){\tiny$2$}
 \end{picture}
 }

 \put(20,90){\begin{picture}(10,10)\setlength{\unitlength}{0.5mm}
 \put(10,5){\tiny$2$}
 \put(5,0){\tiny$2 $}
 \end{picture}
 }
  \put(22,89){\begin{picture}(10,20)\setlength{\unitlength}{0.45mm}
\put(0,0){\line(1,0){13}}
\put(0,0){\line(0,1){13}}
\put(13,0){\line(0,1){13}}
\put(0,13){\line(1,0){13}}
\end{picture}}
\put(32,92){\vector(1,-1){11}}

\put(32,37){\vector(1,-1){11}}

\put(30,44){\vector(1,2){14}}

 \put(50,26){\vector(1,1){12}}

 \put(40,22.5){\begin{picture}(10,10)\setlength{\unitlength}{0.5mm}
 \put(10,5){\tiny$1$}
 \put(5,0){\tiny$2$}
 \end{picture}
 }

 \put(40,70){\begin{picture}(10,10)\setlength{\unitlength}{0.5mm}
 \put(5,10){\tiny$1$}
 \put(0,5){\tiny$1 $}
 \put(10,5){\tiny$2$}
 \put(15,10){\tiny$2 $}
 \end{picture}
 }

 \put(53,81){\vector(1,1){9}}

 \put(60,90){\begin{picture}(10,10)\setlength{\unitlength}{0.5mm}
 \put(10,5){\tiny$1$}
 \put(5,0){\tiny$1 $}
 \end{picture}
 }

  \put(73,91){\vector(1,-1){14}}

 \put(62,35){\begin{picture}(10,10)\setlength{\unitlength}{0.5mm}
\put(5,5){\tiny$2$}
 \put(0,10){\tiny$1$}
 \put(10,10){\tiny$2 $}
 \end{picture}
 }

\put(89,25){\tiny$2 $}

\put(89,75){\tiny$1$}
\put(50,72){\vector(1,-2){13}}

 \put(74,38){\vector(1,-1){11}}

 \put(74,45){\vector(1,2){14}}
 \put(74,43){\vector(1,1){8}}

 \put(85,45){\begin{picture}(10,10)\setlength{\unitlength}{0.5mm}
 \put(5,10){\tiny$1$}
 \put(0,5){\tiny$1 $}
 \put(10,5){\tiny$2$}
 \put(5,0){\tiny$2 $}
 \end{picture}
 }

  \put(84,44){\begin{picture}(10,20)\setlength{\unitlength}{0.6mm}
\put(0,0){\line(1,0){13}}
\put(0,0){\line(0,1){13}}
\put(13,0){\line(0,1){13}}
\put(0,13){\line(1,0){13}}
\end{picture}}

\dashline{3}(33,93)(62,93)
\put(33,93){\vector(-1,0){1}}

\put(10,21){\begin{picture}(0,0)
\dashline{3}(-1,3)(32,3)
\put(-1,3){\vector(-1,0){1}}
\end{picture}
}

\put(8,74){\begin{picture}(0,0)
\dashline{3}(2,3)(32,3)
\put(2,3){\vector(-1,0){1}}
\end{picture}
}

\put(52,73){\begin{picture}(0,0)
\dashline{3}(2,3)(32,3)
\put(2,3){\vector(-1,0){1}}
\end{picture}
}

\put(52,21){\begin{picture}(0,0)
\dashline{3}(0,3)(32,3)
\put(0,3){\vector(-1,0){1}}
\end{picture}
}

\put(72,38){\begin{picture}(2,0)
\dashline{3}(2,3)(18,3)
\put(2,3){\vector(-1,0){1}}
\end{picture}
}

\put(30,38){\begin{picture}(2,0)
\dashline{3}(2,3)(32,3)
\put(2,3){\vector(-1,0){1}}
\end{picture}
}

\put(-13,38){\begin{picture}(2,0)
\dashline{3}(17,3)(32,3)
\put(17,3){\vector(-1,0){1}}
\end{picture}
}

\end{picture}
\caption{Auslander-Reiten quiver for $\Lambda^\imath$ with $Q$ of type $A_2$ and $\btau=\Id$}
\end{figure}
\vspace{-0.5cm}
\end{center}

\subsubsection{$\imath$Quiver with $\btau \neq \Id$}

Let $Q=(\xymatrix{ 1 &2 \ar[l]_{\alpha} \ar[r]^{\beta}  & 3})$, with $\btau$ being the nontrivial involution of $Q$.
Then $\Lambda^\imath$ is isomorphic to the algebra with its quiver $\ov Q$ and relations as follows:
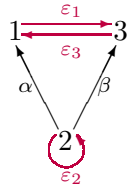
\begin{center}\setlength{\unitlength}{0.7mm}
 \begin{picture}(50,20)(0,-10)
\put(0,-2){\small $1$}
\put(20,-2){\small $3$}
\put(2,-11){$_{\alpha}$}
\put(17,-11){$_{\beta}$}

\put(12.5,-19){\vector(1,2){8}}
\put(9.5,-19){\vector(-1,2){8}}
\put(9.5,-23){\small $2$}
\color{purple}
\put(2.5,1){\vector(1,0){17}}
\put(19.5,-1){\vector(-1,0){17}}
\put(10,1){\small $^{\varepsilon_1}$}
\put(10,-4){\small $_{\varepsilon_3}$}
\put(10,-28){\small $_{\varepsilon_2}$}
\begin{picture}(50,23)(-10,19)
\color{purple}
\qbezier(-1,-1)(-3,-3)(-2,-5.5)
\qbezier(-2,-5.5)(1,-9)(4,-5.5)
\qbezier(4,-5.5)(5,-3)(3,-1)
\put(3.1,-1.4){\vector(-1,1){0.3}}
\end{picture}
\end{picture}
\vspace{1.4cm}
\end{center}
\[
\varepsilon_1\varepsilon_3=0=\varepsilon_3\varepsilon_1,
\quad
 \varepsilon_2^2=0,
 \quad
 \beta  \varepsilon_2=\varepsilon_1 \alpha,
 \quad
 \alpha\varepsilon_2=\varepsilon_3\beta.
\]

\begin{center}\setlength{\unitlength}{0.5mm}
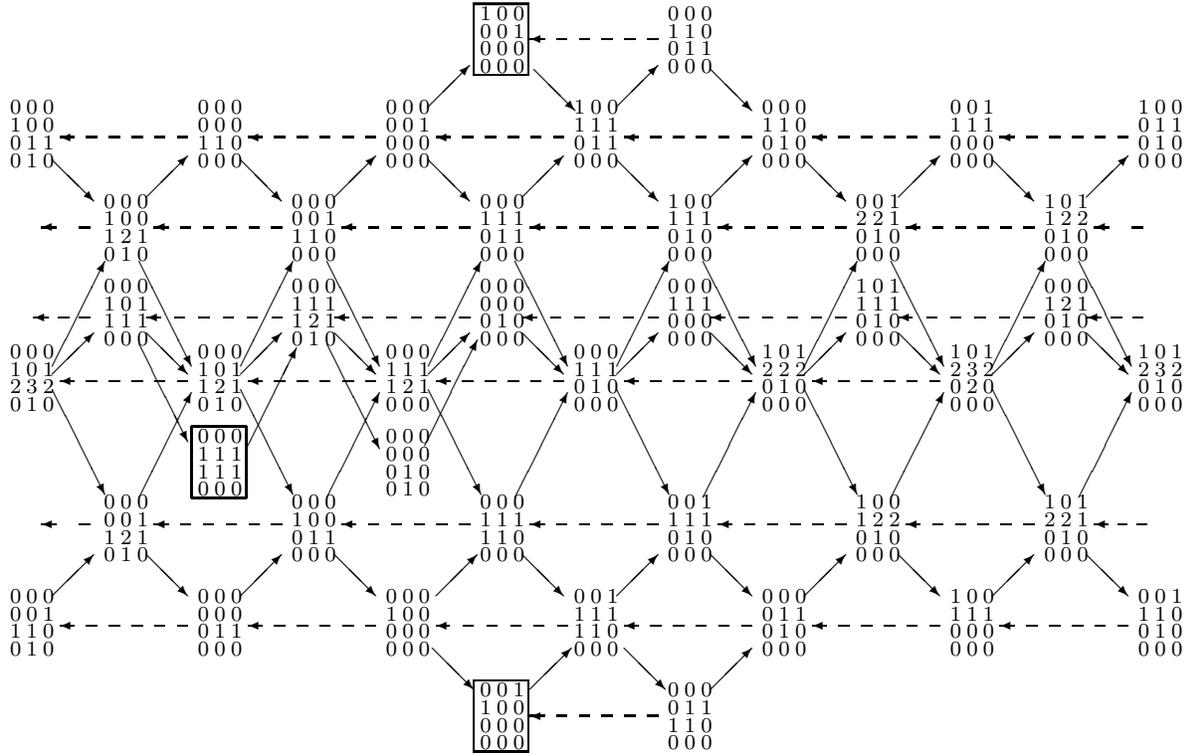
\begin{figure}\label{figure 3}
 \begin{picture}(280,220)(20,0)

\put(0,40){\tiny $\setlength{\arraycolsep}{1pt}\renewcommand{\arraystretch}{0.5}{\begin{array}{cccc} 0&0&0\\ 0&0&1\\ 1&1&0\\ 0&1&0  \end{array} }$ }

\put(50,40){\tiny $\setlength{\arraycolsep}{1pt}\renewcommand{\arraystretch}{0.5}{\begin{array}{cccc} 0&0&0\\ 0&0&0\\ 0&1&1\\ 0&0&0  \end{array} }$ }

\put(100,40){\tiny $\setlength{\arraycolsep}{1pt}\renewcommand{\arraystretch}{0.5}{\begin{array}{cccc} 0&0&0\\ 1&0&0\\ 0&0&0\\ 0&0&0  \end{array} }$ }

\put(150,40){\tiny $\setlength{\arraycolsep}{1pt}\renewcommand{\arraystretch}{0.5}{\begin{array}{cccc} 0&0&1\\ 1&1&1\\ 1&1&0\\ 0&0&0  \end{array} }$ }

\put(200,40){\tiny $\setlength{\arraycolsep}{1pt}\renewcommand{\arraystretch}{0.5}{\begin{array}{cccc} 0&0&0\\ 0&1&1\\ 0&1&0\\ 0&0&0  \end{array} }$ }

\put(250,40){\tiny $\setlength{\arraycolsep}{1pt}\renewcommand{\arraystretch}{0.5}{\begin{array}{cccc} 1&0&0\\ 1&1&1\\ 0&0&0\\ 0&0&0  \end{array} }$ }

\put(300,40){\tiny $\setlength{\arraycolsep}{1pt}\renewcommand{\arraystretch}{0.5}{\begin{array}{cccc} 0&0&1\\ 1&1&0\\ 0&1&0\\ 0&0&0  \end{array} }$ }

\put(75,65){\tiny $\setlength{\arraycolsep}{1pt}\renewcommand{\arraystretch}{0.5}{\begin{array}{cccc} 0&0&0\\ 1&0&0\\ 0&1&1\\ 0&0&0  \end{array} }$ }

\put(125,65){\tiny $\setlength{\arraycolsep}{1pt}\renewcommand{\arraystretch}{0.5}{\begin{array}{cccc} 0&0&0\\ 1&1&1\\ 1&1&0\\ 0&0&0  \end{array} }$ }

\put(175,65){\tiny $\setlength{\arraycolsep}{1pt}\renewcommand{\arraystretch}{0.5}{\begin{array}{cccc} 0&0&1\\ 1&1&1\\ 0&1&0\\ 0&0&0  \end{array} }$ }

\put(225,65){\tiny $\setlength{\arraycolsep}{1pt}\renewcommand{\arraystretch}{0.5}{\begin{array}{cccc} 1&0&0\\ 1&2&2\\ 0&1&0\\ 0&0&0  \end{array} }$ }

\put(275,65){\tiny $\setlength{\arraycolsep}{1pt}\renewcommand{\arraystretch}{0.5}{\begin{array}{cccc} 1&0&1\\ 2&2&1\\ 0&1&0\\ 0&0&0  \end{array} }$ }

\put(25,65){\tiny $\setlength{\arraycolsep}{1pt}\renewcommand{\arraystretch}{0.5}{\begin{array}{cccc} 0&0&0\\ 0&0&1\\ 1&2&1\\ 0&1&0  \end{array} }$ }

\put(0,105){\tiny $\setlength{\arraycolsep}{1pt}\renewcommand{\arraystretch}{0.5}{\begin{array}{cccc} 0&0&0\\ 1&0&1\\ 2&3&2\\ 0&1&0  \end{array} }$ }

\put(50,105){\tiny $\setlength{\arraycolsep}{1pt}\renewcommand{\arraystretch}{0.5}{\begin{array}{cccc} 0&0&0\\ 1&0&1\\ 1&2&1\\ 0&1&0  \end{array} }$ }

\put(100,105){\tiny $\setlength{\arraycolsep}{1pt}\renewcommand{\arraystretch}{0.5}{\begin{array}{cccc} 0&0&0\\ 1&1&1\\ 1&2&1\\ 0&0&0  \end{array} }$ }

\put(150,105){\tiny $\setlength{\arraycolsep}{1pt}\renewcommand{\arraystretch}{0.5}{\begin{array}{cccc} 0&0&0\\ 1&1&1\\ 0&1&0\\ 0&0&0  \end{array} }$ }

\put(200,105){\tiny $\setlength{\arraycolsep}{1pt}\renewcommand{\arraystretch}{0.5}{\begin{array}{cccc} 1&0&1\\ 2&2&2\\ 0&1&0\\ 0&0&0  \end{array} }$ }

\put(250,105){\tiny $\setlength{\arraycolsep}{1pt}\renewcommand{\arraystretch}{0.5}{\begin{array}{cccc} 1&0&1\\ 2&3&2\\ 0&2&0\\ 0&0&0  \end{array} }$ }

\put(300,105){\tiny $\setlength{\arraycolsep}{1pt}\renewcommand{\arraystretch}{0.5}{\begin{array}{cccc} 1&0&1\\ 2&3&2\\ 0&1&0\\ 0&0&0  \end{array} }$ }

\put(25,145){\tiny $\setlength{\arraycolsep}{1pt}\renewcommand{\arraystretch}{0.5}{\begin{array}{cccc} 0&0&0\\ 1&0&0\\ 1&2&1\\ 0&1&0  \end{array} }$ }

\put(75,145){\tiny $\setlength{\arraycolsep}{1pt}\renewcommand{\arraystretch}{0.5}{\begin{array}{cccc} 0&0&0\\ 0&0&1\\ 1&1&0\\ 0&0&0  \end{array} }$ }

\put(125,145){\tiny $\setlength{\arraycolsep}{1pt}\renewcommand{\arraystretch}{0.5}{\begin{array}{cccc} 0&0&0\\ 1&1&1\\ 0&1&1\\ 0&0&0  \end{array} }$ }

\put(175,145){\tiny $\setlength{\arraycolsep}{1pt}\renewcommand{\arraystretch}{0.5}{\begin{array}{cccc} 1&0&0\\ 1&1&1\\ 0&1&0\\ 0&0&0  \end{array} }$ }

\put(225,145){\tiny $\setlength{\arraycolsep}{1pt}\renewcommand{\arraystretch}{0.5}{\begin{array}{cccc} 0&0&1\\ 2&2&1\\ 0&1&0\\ 0&0&0  \end{array} }$ }

\put(275,145){\tiny $\setlength{\arraycolsep}{1pt}\renewcommand{\arraystretch}{0.5}{\begin{array}{cccc} 1&0&1\\ 1&2&2\\ 0&1&0\\ 0&0&0  \end{array} }$ }

\put(0,170){\tiny $\setlength{\arraycolsep}{1pt}\renewcommand{\arraystretch}{0.5}{\begin{array}{cccc} 0&0&0\\ 1&0&0\\ 0&1&1\\ 0&1&0  \end{array} }$ }

\put(50,170){\tiny $\setlength{\arraycolsep}{1pt}\renewcommand{\arraystretch}{0.5}{\begin{array}{cccc} 0&0&0\\ 0&0&0\\ 1&1&0\\ 0&0&0  \end{array} }$ }

\put(100,170){\tiny $\setlength{\arraycolsep}{1pt}\renewcommand{\arraystretch}{0.5}{\begin{array}{cccc} 0&0&0\\ 0&0&1\\ 0&0&0\\ 0&0&0  \end{array} }$ }

\put(150,170){\tiny $\setlength{\arraycolsep}{1pt}\renewcommand{\arraystretch}{0.5}{\begin{array}{cccc} 1&0&0\\ 1&1&1\\ 0&1&1\\ 0&0&0  \end{array} }$ }

\put(200,170){\tiny $\setlength{\arraycolsep}{1pt}\renewcommand{\arraystretch}{0.5}{\begin{array}{cccc} 0&0&0\\ 1&1&0\\ 0&1&0\\ 0&0&0  \end{array} }$ }

\put(250,170){\tiny $\setlength{\arraycolsep}{1pt}\renewcommand{\arraystretch}{0.5}{\begin{array}{cccc} 0&0&1\\ 1&1&1\\ 0&0&0\\ 0&0&0  \end{array} }$ }

\put(300,170){\tiny $\setlength{\arraycolsep}{1pt}\renewcommand{\arraystretch}{0.5}{\begin{array}{cccc} 1&0&0\\ 0&1&1\\ 0&1&0\\ 0&0&0  \end{array} }$ }

\put(125,15){\tiny $\setlength{\arraycolsep}{1pt}\renewcommand{\arraystretch}{0.5}{\begin{array}{cccc} 0&0&1\\ 1&0&0\\ 0&0&0\\ 0&0&0  \end{array} }$ }

\put(175,15){\tiny $\setlength{\arraycolsep}{1pt}\renewcommand{\arraystretch}{0.5}{\begin{array}{cccc} 0&0&0\\ 0&1&1\\ 1&1&0\\ 0&0&0  \end{array} }$ }

\put(125,195){\tiny $\setlength{\arraycolsep}{1pt}\renewcommand{\arraystretch}{0.5}{\begin{array}{cccc} 1&0&0\\ 0&0&1\\ 0&0&0\\ 0&0&0  \end{array} }$ }

\put(175,195){\tiny $\setlength{\arraycolsep}{1pt}\renewcommand{\arraystretch}{0.5}{\begin{array}{cccc} 0&0&0\\ 1&1&0\\ 0&1&1\\ 0&0&0  \end{array} }$ }

\put(50,82.5){\tiny $\setlength{\arraycolsep}{1pt}\renewcommand{\arraystretch}{0.5}{\begin{array}{cccc} 0&0&0\\ 1&1&1\\ 1&1&1\\ 0&0&0  \end{array} }$ }
\put(100,82.5){\tiny $\setlength{\arraycolsep}{1pt}\renewcommand{\arraystretch}{0.5}{\begin{array}{cccc} 0&0&0\\ 0&0&0\\ 0&1&0\\ 0&1&0  \end{array} }$ }

\put(25,122.5){\tiny $\setlength{\arraycolsep}{1pt}\renewcommand{\arraystretch}{0.5}{\begin{array}{cccc} 0&0&0\\ 1&0&1\\ 1&1&1\\ 0&0&0  \end{array} }$ }

\put(75,122.5){\tiny $\setlength{\arraycolsep}{1pt}\renewcommand{\arraystretch}{0.5}{\begin{array}{cccc} 0&0&0\\ 1&1&1\\ 1&2&1\\ 0&1&0  \end{array} }$ }

\put(125,122.5){\tiny $\setlength{\arraycolsep}{1pt}\renewcommand{\arraystretch}{0.5}{\begin{array}{cccc} 0&0&0\\ 0&0&0\\ 0&1&0\\ 0&0&0  \end{array} }$ }

\put(175,122.5){\tiny $\setlength{\arraycolsep}{1pt}\renewcommand{\arraystretch}{0.5}{\begin{array}{cccc} 0&0&0\\ 1&1&1\\ 0&0&0\\ 0&0&0  \end{array} }$ }

\put(225,122.5){\tiny $\setlength{\arraycolsep}{1pt}\renewcommand{\arraystretch}{0.5}{\begin{array}{cccc} 1&0&1\\ 1&1&1\\ 0&1&0\\ 0&0&0  \end{array} }$ }

\put(275,122.5){\tiny $\setlength{\arraycolsep}{1pt}\renewcommand{\arraystretch}{0.5}{\begin{array}{cccc} 0&0&0\\ 1&2&1\\ 0&1&0\\ 0&0&0  \end{array} }$ }

\put(12,49){\vector(1,1){11}}
\put(62,49){\vector(1,1){11}}
\put(114,49){\vector(1,1){11}}
\put(162,49){\vector(1,1){11}}
\put(212,49){\vector(1,1){11}}
\put(262,49){\vector(1,1){11}}

\put(12,108){\vector(1,1){11}}
\put(12,104){\vector(1,-2){14}}
\put(12,110){\vector(1,2){14}}

\put(62,108){\vector(1,1){11}}
\put(62,104){\vector(1,-2){14}}
\put(62,110){\vector(1,2){14}}

\put(112,108){\vector(1,1){11}}
\put(112,104){\vector(1,-2){14}}
\put(112,110){\vector(1,2){14}}

\put(162,108){\vector(1,1){11}}
\put(162,104){\vector(1,-2){14}}
\put(162,110){\vector(1,2){14}}

\put(212,108){\vector(1,1){11}}
\put(212,104){\vector(1,-2){14}}
\put(212,110){\vector(1,2){14}}

\put(262,108){\vector(1,1){11}}
\put(262,104){\vector(1,-2){14}}
\put(262,110){\vector(1,2){14}}

\put(12,165){\vector(1,-1){11}}
\put(62,165){\vector(1,-1){11}}
\put(112,165){\vector(1,-1){11}}
\put(162,165){\vector(1,-1){11}}

\put(212,165){\vector(1,-1){11}}

\put(262,165){\vector(1,-1){11}}

\put(37,60){\vector(1,-1){11}}
\put(87,60){\vector(1,-1){11}}
\put(137,60){\vector(1,-1){11}}
\put(187,60){\vector(1,-1){11}}
\put(237,60){\vector(1,-1){11}}
\put(287,60){\vector(1,-1){11}}

\put(37,119){\vector(1,-1){11}}

\put(87,119){\vector(1,-1){11}}
\put(137,119){\vector(1,-1){11}}
\put(187,119){\vector(1,-1){11}}
\put(237,119){\vector(1,-1){11}}
\put(287,119){\vector(1,-1){11}}

\put(35,75){\vector(1,2){14}}
\put(85,75){\vector(1,2){14}}

\put(135,75){\vector(1,2){14}}
\put(185,75){\vector(1,2){14}}
\put(235,75){\vector(1,2){14}}
\put(285,75){\vector(1,2){14}}

\put(35,138){\vector(1,-2){14}}
\put(85,138){\vector(1,-2){14}}
\put(135,138){\vector(1,-2){14}}
\put(185,138){\vector(1,-2){14}}
\put(235,138){\vector(1,-2){14}}
\put(285,138){\vector(1,-2){14}}

\put(37,155){\vector(1,1){11}}
\put(87,155){\vector(1,1){11}}
\put(137,155){\vector(1,1){11}}
\put(187,155){\vector(1,1){11}}
\put(237,155){\vector(1,1){11}}

\put(287,155){\vector(1,1){11}}
\put(35,116){\vector(1,-2){13}}
\put(85,116){\vector(1,-2){14}}
\put(64,89){\vector(1,2){13}}
\put(111,89){\vector(1,2){14}}

\put(112,179){\vector(1,1){11}}
\put(162,179){\vector(1,1){11}}

\put(139,24){\vector(1,1){11}}

\put(187,24){\vector(1,1){11}}
\put(113,35){\vector(1,-1){11}}
\put(164,35){\vector(1,-1){11}}
\put(187,189){\vector(1,-1){11}}
\put(140,189){\vector(1,-1){11}}

\put(0,1){\begin{picture}(0,0)
\dashline{3}(15,40)(48,40)
\put(15,40){\vector(-1,0){1}}
\end{picture}}

\put(50,1){\begin{picture}(0,0)
\dashline{3}(15,40)(48,40)
\put(15,40){\vector(-1,0){1}}
\end{picture}}

\put(100,1){\begin{picture}(0,0)
\dashline{3}(15,40)(49,40)
\put(15,40){\vector(-1,0){1}}
\end{picture}}

\put(150,1){\begin{picture}(0,0)
\dashline{3}(15,40)(48,40)
\put(15,40){\vector(-1,0){1}}
\end{picture}}

\put(200,1){\begin{picture}(0,0)
\dashline{3}(15,40)(48,40)
\put(15,40){\vector(-1,0){1}}
\end{picture}}

\put(250,1){\begin{picture}(0,0)
\dashline{3}(15,40)(48,40)
\put(15,40){\vector(-1,0){1}}
\end{picture}}

\put(0,66){\begin{picture}(0,0)
\dashline{3}(15,40)(48,40)
\put(15,40){\vector(-1,0){1}}
\end{picture}}

\put(50,66){\begin{picture}(0,0)
\dashline{3}(15,40)(48,40)
\put(15,40){\vector(-1,0){1}}
\end{picture}}

\put(100,66){\begin{picture}(0,0)
\dashline{3}(15,40)(48,40)
\put(15,40){\vector(-1,0){1}}
\end{picture}}

\put(150,66){\begin{picture}(0,0)
\dashline{3}(15,40)(48,40)
\put(15,40){\vector(-1,0){1}}
\end{picture}}

\put(200,66){\begin{picture}(0,0)
\dashline{3}(15,40)(48,40)
\put(15,40){\vector(-1,0){1}}
\end{picture}}

\put(0,131){\begin{picture}(0,0)
\dashline{3}(15,40)(48,40)
\put(15,40){\vector(-1,0){1}}
\end{picture}}

\put(50,131){\begin{picture}(0,0)
\dashline{3}(15,40)(48,40)
\put(15,40){\vector(-1,0){1}}
\end{picture}}

\put(100,131){\begin{picture}(0,0)
\dashline{3}(15,40)(48,40)
\put(15,40){\vector(-1,0){1}}
\end{picture}}

\put(150,131){\begin{picture}(0,0)
\dashline{3}(15,40)(48,40)
\put(15,40){\vector(-1,0){1}}
\end{picture}}

\put(200,131){\begin{picture}(0,0)
\dashline{3}(15,40)(48,40)
\put(15,40){\vector(-1,0){1}}
\end{picture}}

\put(250,131){\begin{picture}(0,0)
\dashline{3}(15,40)(48,40)
\put(15,40){\vector(-1,0){1}}
\end{picture}}

\put(25,28){\begin{picture}(0,0)
\dashline{3}(15,40)(48,40)
\put(15,40){\vector(-1,0){1}}
\end{picture}}

\put(75,28){\begin{picture}(0,0)
\dashline{3}(15,40)(48,40)
\put(15,40){\vector(-1,0){1}}
\end{picture}}

\put(125,28){\begin{picture}(0,0)
\dashline{3}(15,40)(48,40)
\put(15,40){\vector(-1,0){1}}
\end{picture}}

\put(175,28){\begin{picture}(0,0)
\dashline{3}(15,40)(48,40)
\put(15,40){\vector(-1,0){1}}
\end{picture}}

\put(225,28){\begin{picture}(0,0)
\dashline{3}(15,40)(48,40)
\put(15,40){\vector(-1,0){1}}
\end{picture}}

\put(275,28){\begin{picture}(0,0)
\dashline{3}(15,40)(28,40)
\put(15,40){\vector(-1,0){1}}
\end{picture}}

\put(-5,28){\begin{picture}(0,0)
\dashline{3}(15,40)(27,40)
\put(15,40){\vector(-1,0){1}}
\end{picture}}

\put(-5,107){\begin{picture}(0,0)
\dashline{3}(15,40)(27,40)
\put(15,40){\vector(-1,0){1}}
\end{picture}}

\put(25,107){\begin{picture}(0,0)
\dashline{3}(15,40)(48,40)
\put(15,40){\vector(-1,0){1}}
\end{picture}}

\put(75,107){\begin{picture}(0,0)
\dashline{3}(15,40)(48,40)
\put(15,40){\vector(-1,0){1}}
\end{picture}}

\put(125,107){\begin{picture}(0,0)
\dashline{3}(15,40)(48,40)
\put(15,40){\vector(-1,0){1}}
\end{picture}}

\put(175,107){\begin{picture}(0,0)
\dashline{3}(15,40)(48,40)
\put(15,40){\vector(-1,0){1}}
\end{picture}}

\put(225,107){\begin{picture}(0,0)
\dashline{3}(15,40)(48,40)
\put(15,40){\vector(-1,0){1}}
\end{picture}}

\put(275,107){\begin{picture}(0,0)
\dashline{3}(15,40)(27,40)
\put(15,40){\vector(-1,0){1}}
\end{picture}}

\put(125,-23){\begin{picture}(0,0)
\dashline{3}(15,40)(48,40)
\put(15,40){\vector(-1,0){1}}
\end{picture}}

\put(125,157){\begin{picture}(0,0)
\dashline{3}(15,40)(48,40)
\put(15,40){\vector(-1,0){1}}
\end{picture}}

\put(25,83){\begin{picture}(0,0)
\dashline{3}(13,40)(48,40)
\put(13,40){\vector(-1,0){1}}
\end{picture}}

\put(75,83){\begin{picture}(0,0)
\dashline{3}(13,40)(48,40)
\put(13,40){\vector(-1,0){1}}
\end{picture}}

\put(125,83){\begin{picture}(0,0)
\dashline{3}(13,40)(48,40)
\put(13,40){\vector(-1,0){1}}
\end{picture}}

\put(175,83){\begin{picture}(0,0)
\dashline{3}(13,40)(48,40)
\put(13,40){\vector(-1,0){1}}
\end{picture}}

\put(225,83){\begin{picture}(0,0)
\dashline{3}(13,40)(48,40)
\put(13,40){\vector(-1,0){1}}
\end{picture}}

\put(275,83){\begin{picture}(0,0)
\dashline{3}(13,40)(27,40)
\put(13,40){\vector(-1,0){1}}
\end{picture}}

\put(-5,83){\begin{picture}(0,0)
\dashline{3}(13,40)(27,40)
\put(13,40){\vector(-1,0){1}}
\end{picture}}

\put(124,7.5){\begin{picture}(10,20)
\put(0,0){\line(1,0){14.5}}
\put(0,0){\line(0,1){19}}
\put(14.5,0){\line(0,1){19}}
\put(0,19){\line(1,0){14.5}}
\end{picture}}

\put(124,187.5){\begin{picture}(10,20)
\put(0,0){\line(1,0){14.5}}
\put(0,0){\line(0,1){19}}
\put(14.5,0){\line(0,1){19}}
\put(0,19){\line(1,0){14.5}}
\end{picture}}

\put(49,75){\begin{picture}(10,20)
\put(0,0){\line(1,0){14.5}}
\put(0,0){\line(0,1){19}}
\put(14.5,0){\line(0,1){19}}
\put(0,19){\line(1,0){14.5}}
\end{picture}}
%
\end{picture}
\caption{Auslander-Reiten quiver of $\Lambda^\imath$ with $Q$ of type $A_3$ and $\btau \neq \Id$}
\end{figure}
\end{center}
Recall that $\Lambda^\imath$ is a graded algebra by letting $\deg(\varepsilon_i)=1$ and $\deg(\alpha)=0=\deg(\beta)$.
Its Auslander-Reiten quiver of $\Lambda^\imath$  is displayed in Figure 3. 
As vertices we have the graded dimension vectors of the indecomposable $\Lambda^\imath$-modules.
The modules on the leftmost column of the figure have to be identified with the corresponding modules on the rightmost one. The projective $\Lambda^\imath$-modules are marked with a solid frame.

\subsection{APR-tilting}
  \label{subsec:APR2}

Let $\DTr$ be the Auslander-Reiten translation of $\Lambda^{\imath}$, and $\TrD$ be its inverse.
For a sink $\ell$ in $Q$,  define
\begin{align}
\label{eqn:def of T}
T & =T[\ell] :=\left\{
\begin{array}{ll}
(\Lambda^{\imath}\big/ \E_\ell) \oplus  \TrD\E_\ell, & \text{ if }\btau \ell =\ell
 \\
\big(\Lambda^{\imath} \big/(\E_\ell \oplus \E_{\btau \ell } ) \big)\oplus \TrD\E_\ell  \oplus \TrD\E_{\btau \ell}, &\text{ if }\btau \ell \neq \ell  \end{array}\right.
\end{align}


\begin{theorem}
\label{thm:tilting module}
(1) The module $T$ in \eqref{eqn:def of T} is a tilting module of $\Lambda^{\imath}$.

(2) Let $B=\End_{\Lambda^{\imath}}(T)^{op}$. Then the functors
\[
F_\ell ^+: \mod(\Lambda^{\imath})\longrightarrow \mod(\bs_\ell \iLa )\text{ and }\Hom_{\Lambda^{\imath}}(T, -): \mod(\Lambda^{\imath})\longrightarrow \mod(B)
\]
are equivalent, i.e., there exists an equivalence
$S: \mod(\bs_\ell \iLa )\rightarrow \mod(B)$ which gives rise to an isomorphism of functors $S\circ F_\ell^+\cong \Hom_{\Lambda^{\imath}}(T, -)$.
\end{theorem}

We shall adapt and modify the arguments for \cite[Theorem 9.7]{GLS17} with some more details added at various places.

\subsection{Proof of Theorem~\ref{thm:tilting module}(1)}
  \label{subsec:proof1}

We start with some preparations. The left $\iLa$-module $T$ can be described as
\begin{align}
\label{eqn:Tleft}
T=\iLa/\BH_\ell\oplus \TrD(\BH_\ell ),
\end{align}
where $\BH_\ell $ is defined as in \eqref{dfn:Hi}.

Let $I_i:=D(e_i \iLa)$, $P_i:=\iLa e_i$ for any $i\in Q_0$. We have
\begin{align}
\label{eqn: Ti}
T=\bigoplus_{i\in \I} T_i, \qquad \text{ where }
T_i:=\left\{ \begin{array}{cl} P_i, & \text{ if }i\neq \ell ,\btau \ell ,\\
\TrD (\E_{\btau \ell}), & \text{ if }i=\ell ,\\
\TrD (\E_\ell), &\text{ if }i=\btau \ell. \end{array}\right.
\end{align}

From \eqref{eqn:injective resolution of E}, we have the following minimal injective resolution of $\BH_\ell$:
\begin{equation}
0 \longrightarrow \BH_\ell  \longrightarrow \ov{I}_\ell \stackrel{f}{\longrightarrow} \bigoplus_{j\in\ov{\Omega}(-,\ell)} I_j\otimes_{\BH_j}{}_j\BH_\ell   \longrightarrow 0,
\end{equation}
where
\[
\ov{I}_\ell=\left\{\begin{array}{cc} I_\ell, & \text{ if }\btau \ell=\ell,
\\
I_\ell\oplus I_{\btau \ell}, &\text{ if }\btau \ell\neq \ell.  \end{array} \right.
\]
By applying the inverse Nakayama functor
\[
\nu^-=\Hom_{\iLa}(D(\iLa),-):{\rm inj}(\iLa)\longrightarrow \proj(\iLa),
\]
we have $\nu^{-}(f):\BH_\ell \rightarrow\bigoplus_{j\in\ov{\Omega}(-,\ell)} P_{j}\otimes_{\BH_j}{}_j\BH_\ell$ by noting that
$\nu^-(I_\ell)=P_{\ell}=\E_\ell$ and $\nu^-(I_{\btau\ell})=P_{\btau\ell}=\E_{\btau \ell}$ since $\ell$, $\btau\ell$ are sink vertices in $Q$.
Then there is an exact sequence
\begin{equation}\label{eqn: pd of tau inverse of Ha}
0\longrightarrow \BH_\ell \xrightarrow{\nu^-(f)} \bigoplus_{j\in\ov{\Omega}(-,\ell)} P_{j}\otimes_{\BH_j}{}_j\BH_\ell \longrightarrow \TrD(\BH_\ell )\longrightarrow0,
\end{equation}
which is a minimal projective resolution of $\TrD(\BH_\ell )$.
So the projective dimension of $\TrD(\BH_\ell )$ is at most one.
Moreover, by \eqref{eqn:injective resolution of E} and \eqref{eqn: pd of tau inverse of Ha}, we obtain a minimal projective resolution of $\TrD(\E_{\btau\ell })$:
\begin{align}\label{eqn: pd of tau inverse of Ea}
0\longrightarrow \E_{ \ell }\longrightarrow \bigoplus_{(\alpha:j\rightarrow  \ell)\in Q_1 } \iLa \, e_j\longrightarrow \TrD(\E_{\btau\ell} )\longrightarrow0.
\end{align}

It follows from the Auslander-Reiten formulas (cf. \cite[Chapter IV.2, Corollary 2.14]{ASS04}) that
\begin{align}
\label{eqn:AR1}
&\Hom_{\iLa}(\TrD(\BH_\ell ),\iLa)=0;
\end{align}

Now we are ready to prove that $T$ is a tilting module.

\begin{proof}[Proof of Theorem~\ref{thm:tilting module}(1)]

It follows from \eqref{eqn: pd of tau inverse of Ha} that $\pd_{\La^\imath} T= 1$.
Furthermore, we have
\begin{align*}
\Ext^1_{\La^\imath}(T,T)=&\Ext^1_{\La^\imath}( \TrD(\BH_\ell),T)
\\
\cong&D\Hom_{\La^\imath}(T,\BH_\ell)
\\
=&D\Hom_{\La^\imath}(\TrD(\BH_\ell ),\BH_\ell)
=0.
\end{align*}
The first equality above holds since $\iLa/\BH_\ell$ is projective; cf. \eqref{eqn:Tleft}. The isomorphism in the second line follows by the Auslander-Reiten formulas and $\pd_{\Lambda^\imath} \TrD(\BH_\ell)\leq1$. The equality in the third line holds since $\ell,\btau\ell$ are sink vertices, and the last equality follows from \eqref{eqn:AR1}.
Note that $T$ has exactly $|\I|$ pairwise non-isomorphic indecomposable direct summands; cf. \eqref{eqn: Ti}. Therefore $T$ is a tilting module.
\end{proof}

\subsection{Proof of Theorem~\ref{thm:tilting module}(2)}
  \label{subsec:proof2}

We write ${}_B D(T)$ to denote $D(T)$ viewed as a left $B$-module.

\begin{lemma}
\label{lem:BDT}
We have
\begin{align*}
{}_B D(T)\cong \left\{
\begin{array}{cc} \big(D(B)/D(e_\ell B) \big)\oplus \DTr_B D(e_\ell B) &  \text{ if }\btau \ell=\ell,\\
D(B)\big/ \big(D(e_\ell B)\oplus D(e_{\btau \ell}B) \big)\oplus \DTr_B D(e_\ell B) \oplus \DTr_B D(e_{\btau \ell}B) &  \text{ if }\btau \ell\neq \ell.\end{array}\right.
\end{align*}
\end{lemma}

\begin{proof}
We have
\begin{align*}
{}_BD(T)\cong \Hom_{\La^\imath}(T,D(\Lambda^\imath))=\bigoplus_{j\in\I} \Hom_{\La^\imath}(T,D(e_j\La^\imath)).
\end{align*}
For $j\notin\{ \ell,\btau \ell\}$, there are left $B$-module isomorphisms
\begin{align*}
\Hom_{\La^\imath}(T,D(e_j\La^\imath))\cong D\Hom_{\La^\imath}(\La^\imath e_j,T)=D\Hom_{\La^\imath}(T_j,T)\cong D(e_jB).
\end{align*}
For $j\in\{\ell,\btau \ell\}$, since $\Ext^1_{\La^\imath}(T,\iLa e_j)\neq0$, the Connecting Lemma \cite[\S2.3]{HR82} implies that
\begin{align*}
\Hom_{\iLa}(T,D(e_j\iLa))\cong \TrD_B(\Ext^1_{\iLa}(T,\iLa e_j)).
\end{align*}
Using the Auslander-Reiten formulas and $\ind \iLa e_j\leq1$, we have the $B$-module isomorphisms
\begin{align*}
\Ext^1_{\iLa}(T,\iLa e_j)\cong D\Hom_{\iLa}(\TrD(\iLa e_j),T)\cong D(e_{\btau j}B).
\end{align*}
Here the last isomorphism follows from \eqref{eqn: Ti}. This proves the lemma.
\end{proof}

For any of the algebras $A\in \{\iLa,\bs_\ell(\iLa),B\}$ and any $j\in Q_0$ let
\begin{align}
\ct_j^A:=\{X\in\mod(A)\mid \Hom_A(X,S_j\oplus S_{\btau j})=0\},\\
\cs_j^A:=\{X\in\mod(A)\mid \Hom_A(S_j\oplus S_{\btau j},X)=0\}.\label{eq:defSAj}
\end{align}
Let
\begin{align}
\ct:=&\{X\in\mod(\iLa)\mid \Ext^1_{\iLa}(T,X)=0\},
\\
\cy:=&\{Y\in\mod(B)\mid \Tor^1_{B}(T,Y)=0\}
\label{eq:YY}.
\end{align}

\begin{lemma}
\label{lem:tilting theory}
The functors $F:=\Hom_{\iLa}(T,-)$ and $G:=T\otimes_B-$ induce mutually quasi-inverse equivalences $F:\ct\rightarrow \cy$ and $G:\cy\rightarrow \ct$. In particular, we have
$F(\mod(\iLa))\subseteq \cy\text{ and }G(\mod(B))\subseteq \ct$, and
\begin{align}
\label{eq:desY}
\ct= \ct_\ell^{\iLa}, \qquad\cy=\cs_\ell^{B}.
\end{align}
\end{lemma}

\begin{proof}
Since $T$ is a tilting module, the first statement follows from the classical tilting theory (see, e.g., \cite[\S2]{HR82}). In particular,
$F(\mod(\iLa))\subseteq \cy\text{ and }G(\mod(B))\subseteq \ct$.

Let us prove the first formula in \eqref{eq:desY}. By using the Auslander-Reiten formulas, we have
\begin{align}
\label{eq:desT2}
\ct=\{X\in\mod(\iLa)\mid \Ext^1_{\iLa}(T,X)=0\}=\{X\in\mod(\iLa)\mid D\Hom_{\iLa}(X,\BH_\ell )=0\}.
\end{align}
Denote by $S=S_\ell\oplus S_{\btau\ell}$ if $\btau\ell\neq \ell$, by $S=S_\ell$ if $\btau\ell=\ell$. 
Since the set of composition factors of $\BH_\ell$ is $\{S_\ell,S_{\btau \ell}\}$ and there is an injective morphism $S\rightarrow \BH_\ell$,
we have
\begin{align*}
\text{RHS}\eqref{eq:desT2} = \{X\in\mod(\iLa)\mid D\Hom_{\iLa}(X,S_\ell\oplus S_{\btau \ell} )=0\}=\ct_\ell^{\iLa}.
\end{align*}
The desired formula follows.

We prove the second formula in \eqref{eq:desY}. By using the Auslander-Reiten formulas, we have
\begin{align}
\cy
=&\{Y\in\mod(B)\mid D\Ext^1_{B}(Y,D(T))=0\}
 \label{eq:YY2} \\
=&\{Y\in\mod(B)\mid \Hom_{B}(D(e_\ell B)\oplus D(e_{\btau \ell}B),Y)=0 \}.
 \notag
\end{align}
For any $j\neq \ell,\btau \ell$, by Lemma \ref{lem:BDT} and the  Auslander-Reiten formulas, if $\btau \ell=\ell$, we have
\begin{align*}
\Hom_B(D(e_\ell B),D(e_jB))\cong \Ext^1_B(D(e_j B),D(T))=0;
\end{align*}
if $\btau \ell\neq\ell$, we have
\begin{align*}
\Hom_B(D(e_\ell B)\oplus D(e_{\btau \ell}B),D(e_jB))\cong \Ext^1_B(D(e_j B),D(T))=0.
\end{align*}
So every composition factor of $D(e_\ell B)\oplus D(e_{\btau \ell}B)$ is isomorphic to $S_{\ell}$ or $S_{\btau \ell}$.

{\bf Claim.} There is an epimorphism
$D(e_\ell B)\oplus D(e_{\btau \ell}B)\rightarrow S_\ell\oplus S_{\btau \ell}$.

Assuming the Claim for now, we have
\begin{align*}
& \text{RHS}\eqref{eq:YY2}
=\{Y\in\mod(B)\mid \Hom_{B}(S_\ell\oplus S_{\btau \ell},Y)=0 \}=\cs_\ell^{B},
\end{align*}
and then the desired formula follows.

It remains to prove the Claim. From the above, there exists an epimorphism $f:D(e_\ell B)\rightarrow S_j$ for some $j\in\{\ell,\btau \ell\}$. We can now assume $\btau\ell \neq\ell$. Denote also by $\btau$ the isomorphism of $\mod(\Lambda^\imath)$ induced by $\btau$. Since $\btau T=T$, $\btau$ induced an involution of $B$, and then of $\mod(B)$, which is also denoted by $\btau$. Then we obtain another epimorphism $\btau(f):D(e_{\btau \ell} B)\rightarrow S_{\btau j}$. The claim follows by combining the above two epimorphisms.
\end{proof}

\begin{lemma}
\label{lem:equisubcat}
Via restriction of the functors $F^+_\ell, F\circ F^-_\ell$ and $F$ we have a commutative diagram
\[\xymatrix{ \ct_\ell^{\Lambda^\imath} \ar[rr]^{F_\ell^+} \ar@{=}[d]&& \cs_\ell^{\bs_\ell\iLa} \ar[d]^{F\circ F_\ell^-} \\
\ct_\ell^{\iLa} \ar[rr]^{F} && \cs_\ell^B}\]
of equivalences of subcategories.
\end{lemma}

\begin{proof}
First, note that $L\in \ct_\ell^{\iLa}$ if and only if $L_{\ell ,{\rm in}}: \,_\ell \BH_i\otimes_{\BH_i} L_i\rightarrow L_\ell $ is epic; cf. \eqref{eq:Lin}. In this case, \eqref{def:reflection} is a short exact sequence. Denote by $N=F_\ell^+(L)$. From \eqref{def:reflection}, it is obvious that $N\in  \cs_\ell^{\bs_\ell\iLa}$. Similarly, we have $F_\ell^-(\cs_\ell^{\bs_\ell\iLa})\subseteq \ct_\ell^{\iLa}$. By their constructions, we have $F_\ell^-F_\ell^+(L)=L$ for any $L\in \ct_\ell^{\iLa}$, and
$F_\ell^-F_\ell^+(f)=f$ for any morphism  $f$ in $\ct_\ell^{\iLa}$, i.e., $F_\ell^-F_\ell^+\simeq\Id$ by restricting to $\ct_\ell^{\iLa}$. Dually, $F_\ell^+F_\ell^-\simeq\Id$ by restricting to $\cs_\ell^{\bs_\ell\iLa}$.
Together with Lemma \ref{lem:tilting theory}, the result follows.
\end{proof}

\begin{lemma}
\label{lem:comdiag2}
Via restriction of the functor $F\circ F_\ell^-$ we obtain a commutative diagram
\[\xymatrix{ \mod(\iLa)\ar[rr]^{F_\ell^+} \ar@{=}[d]&& \cs_\ell^{\bs_\ell\iLa} \ar[d]^{F\circ F_\ell^-} \\
\mod(\iLa) \ar[rr]^{F} && \cs_\ell^B}\]
with $F\circ F_\ell^-$ an equivalence of subcategories.
\end{lemma}

\begin{proof}
It follows from Lemma \ref{lem:equisubcat}, the definition of $F^+_\ell$ and Lemma \ref{lem:tilting theory}.
\end{proof}

\begin{lemma}
\label{lem:projinclusion}
We have
\begin{align*}
\proj(\bs_\ell\iLa)\subseteq \cs_\ell^{\bs_\ell\iLa}, \text{ and } \proj(B)\subseteq \cs_\ell^B.
\end{align*}
\end{lemma}

\begin{proof}
Since $\ell$ is a source of $\bs_\ell Q$, we have $\proj(\bs_\ell\iLa)\subseteq \cs_\ell^{\bs_\ell\iLa}$ by definition; see \eqref{eq:defSAj}.
The second inclusion is obvious since the $B$-modules $\Hom_{\iLa}(T,T_j)$ are (up to isomorphism) the indecomposable projective $B$-modules.
\end{proof}

Now we can complete the proof of Theorem \ref{thm:tilting module}.

\begin{proof}[Proof of Theorem \ref{thm:tilting module}(2)]
By Lemmas~ \ref{lem:comdiag2}--\ref{lem:projinclusion} and noting that $\cs_\ell^A$ is closed under taking submodules for $A\in\{\bs_\ell\iLa,B\}$, the assumptions of \cite[Lemma 2.2]{APR79} are satisfied. The theorem follows now by applying \cite[Lemma 2.2]{APR79}.
\end{proof}
\subsection{Some consequences}

The following is a direct corollary of Theorem \ref{thm:tilting module}.

\begin{corollary}
The tilting module $T$ in \eqref{eqn:def of T} induces a derived equivalence
\[\RHom_{\Lambda^\imath}(T,-):D^b(\mod(\Lambda^\imath))\longrightarrow D^b(\mod(\bs_\ell\Lambda^\imath)).\]
\end{corollary}

From Lemma \ref{lem:comdiag2}, we have $F\circ F_\ell^-\circ F_\ell^+\simeq F$.
Via restriction of $F\circ F_\ell ^-$, we obtain an equivalence of subcategories
\[
\proj(\bs_\ell \iLa)\longrightarrow \proj(B).
\]
In particular, this equivalence induces an algebra isomorphism
\begin{align}
 \label{cor:Tiso}
\bs_\ell \iLa\xrightarrow{\sim}B=\End_{\Lambda^\imath}(T)^{op}.
\end{align}
We shall always identify $\End_{\Lambda^\imath}(T)^{op}$ with $\bs_\ell \iLa$ in this way, and hence identify $F$ with $F_\ell^+$ (in particular when studying the BGP-type reflection functors later on). Furthermore,  it follows from \eqref{cor:Tiso}, the definition of $F_\ell^+$ and \eqref{eqn: pd of tau inverse of Ea} that
\[
F_\ell ^+(T_i)=P_i',
\qquad \forall i \in \I.
\]

\begin{corollary}
\label{lem:reflecting dimen}
Let $(Q, \btau)$ be an $\imath$quiver with a sink $\ell$. Let $L\in\mod(\K Q)\subseteq \mod(\iLa)$ be an indecomposable module. Then either $F_\ell^+(L)=0$ (equivalently $L\cong S_\ell$ or $S_{\btau \ell}$) or $F_\ell^+(L)$ is indecomposable with
$\dimv F_\ell^+(L)=\bs_\ell(\dimv L)$, where $\bs_\ell$ is as in \eqref{def:simple reflection}.
\end{corollary}

\begin{proof}
Denote the indecomposable module $L$ by $L=(L_i,L_{ji})\in\cm(Q, \btau)$.
Let $T$ be the corresponding tilting module defined in \eqref{eqn:def of T}. Let $B=\End_{\iLa}(T)^{op}\cong \bs_\ell(\iLa)$.

If $L\cong S_\ell$ or $S_{\btau \ell}$, by definition of $F_\ell^+$, we have $F_\ell^+(L)=0$.

Otherwise, we have $L\in \Fac T$. Then $F_\ell^+(L)\cong \Hom_{\iLa}(T,L)$. As $\Hom_{\iLa}(T,-):\Fac T\subseteq \ct\xrightarrow{\simeq} \cy$ (see \eqref{eq:YY}), we have $\End_B(F_\ell^+(L))\cong \End_{\iLa}(L)$ which is a local algebra since $L$ is indecomposable. So $F_\ell^+(L)$ is also indecomposable.

Let $F^+_{\ell}(L)=(N_i,N_{ji})\in\cm(Q, \btau)$.
Since $L$ is indecomposable and $\ell$ is a sink, it follows that the morphism $L_{\ell,{\rm in}}$ in \eqref{def:reflection} is surjective, and then \eqref{def:reflection} becomes a short exact sequence. 
 We have
\begin{align*}
\dim_\K (e_{\ell}N_\ell) &=\sum_{(\alpha:i\rightarrow \ell)\in Q_1} \dim_\K (e_i L_i)-\dim_\K (e_{\ell} L_\ell),
  \\
\dim_\K (e_{\btau\ell}N_\ell)& =\sum_{(\alpha:i\rightarrow \btau\ell)\in Q_1} \dim_\K (e_{\btau i} L_i)-\dim_\K (e_{\btau\ell} L_\ell).
\end{align*}
Since $N_i=L_i$ for $i\neq \ell$, from the above, $\dimv F^+_{\ell}(L)=\bs_\ell(\dimv L)$ by noting that $\bs_\ell=s_\ell$ if $\btau\ell=\ell$ and $\bs_\ell=s_\ell s_{\btau \ell}$ otherwise.
\end{proof}

Dual result also holds for $F_\ell^-$ for any source $\ell$ of $Q$.

\section{Symmetries of $\imath$Hall algebras}
  \label{sec:symmetryHall}

In this section, we show that the reflection functor associated to a sink $\ell \in Q_0$ induces an isomorphism of $\imath$Hall algebras, $\Gamma_{\ell}: \tMH \stackrel{\sim}{\longrightarrow}  \tMHl$,  for an arbitary acyclic $\imath$quiver. In the case of Dynkin $\imath$quivers, we work out explicitly the actions of $\Gamma_\ell$ on the simple modules.

\subsection{Tilting invariance}
  \label{subsec:tilt}

Let $(Q, \btau)$ be an $\imath$quiver, not necessarily of Dynkin type.  We assume $Q$ to be connected and  $(Q,\btau)$ of rank $\ge 2$. 
Let $\ell$ be a sink in $Q$.
Recall $Q'= \bs_\ell Q$ from \eqref{eq:QQ}. As in \S\ref{subsec:APR}, $\btau$ induces an involution $\btau'$ of $Q'$. Let $\bs_\ell\iLa$ denote the $\imath$quiver algebra associated to $(Q', \btau')$.

Specializing \cite[Theorem A.22]{LW19} to our setting gives us the following; here we recall Theorem \ref{thm:tilting module} and \eqref{cor:Tiso} regarding the tilting module $T$ defined in \eqref{eqn:def of T}.
\begin{lemma}
  \label{lem:Upsilon}
  Let $\ell$ be a sink of $Q$. Recall the tilting module $T$ defined in \eqref{eqn:def of T}. Set $F=\Hom_\Lambda(T,-)$.
Then we have an isomorphism of algebras:
\begin{align}
  \label{eqn:reflection functor 1}
\Gamma_{\ell}:\utMH&\stackrel{\sim}{\longrightarrow} \cm\ch(\bs_\ell\iLa)\\\notag
[M]&\mapsto q^{-\langle M,T_M\rangle} [F(T_M)]^{-1}\diamond [F(X_M)],
\end{align}
where $X_M\in \Fac T$ and $T_M\in\add T$ fit in a short exact sequence
$0\rightarrow M\rightarrow X_M\rightarrow T_M\rightarrow0.$ (Note that $\Gamma_\ell$ does not depend on the choice of $X_M$, $T_M$.)
\end{lemma}

It is well known that the Cartan matrix $C$ for $Q$ is the matrix of the symmetric bilinear form $(\cdot,\cdot)_Q$ defined by
\[
(x,y)_Q:=\langle x,y\rangle_Q+\langle y,x\rangle_Q
\]
for $x,y\in K_0(\mod(\K Q))$. Here $\langle\cdot,\cdot\rangle_Q$ is the Euler form of $\K Q$.

For any $M,N\in \mod(\Lambda^{\imath})$, let
\begin{equation}\label{eqn:resolution of M and N}
0\longrightarrow M\longrightarrow X_M\longrightarrow T_M\longrightarrow0, \qquad
0\longrightarrow N\longrightarrow X_N\longrightarrow T_N\longrightarrow0
\end{equation}
be short exact sequences with $X_M,X_N\in\Fac T$, and $T_M,T_N\in \add T$. 

\begin{lemma}   \label{lem:reflection functor preserve bilnear form}
Let $\ell$ be a sink of $Q$. For any $M,N\in \mod(\Lambda^{\imath})$ as in \eqref{eqn:resolution of M and N}, we have
\begin{equation}
  \label{eq:Euler2}
\Big\langle \widehat{\res(F_{\ell}^+(X_M))}- \widehat{\res(F_{\ell}^+(T_M))}, \widehat{\res(F_{\ell}^+(X_N))}- \widehat{\res(F_{\ell}^+(T_N))} \Big \rangle_{Q'}=\langle \res(M),\res(N) \rangle_Q.
\end{equation}
\end{lemma}

\begin{proof}
Denote by $S_i'$ the simple $\K Q'$-module corresponding to $i\in \I$.
Similar to the proof of \cite[Theorem~A.15]{LW19}, it is routine to check that the left-hand side of \eqref{eq:Euler2}
is independent of the choices in \eqref{eqn:resolution of M and N}.

It suffices to prove \eqref{eq:Euler2} when $M,N$ are simple modules. Assume that $M=S_i$, $N=S_j$.
The proof is separated into the following four cases.

Case (1): \underline{$i,j\notin \{\ell, \btau\ell\}$}. Then
$S_i,S_j\in \Fac T$. Choose $X_{S_i}=S_i$, $X_{S_j}=S_j$ and $T_{S_i}=0=T_{S_j}$ in \eqref{eqn:resolution of M and N}. By Corollary \ref{lem:reflecting dimen}, we have
\begin{align*}
\big \langle & \res(F_{\ell}^+([S_i])),\res(F_{\ell}^+([S_j])) \big \rangle_{Q'}\\
&= \left \langle \res(F_{\ell}^+(S_i)), \res(F_{\ell}^+(S_j)) \right \rangle_{Q'}\\
&= \left\{ \begin{array}{ll} \left \langle \widehat{S'_i}+ |c_{\ell i}|\widehat{S'_{\ell}}+|c_{\btau \ell,i}| \widehat{S'_{\btau \ell}},\,  \widehat{S'_j}+ |c_{\ell j}|\widehat{S'_{\ell}} + |c_{\btau \ell, \,  j}|\widehat{S'_{\btau \ell}} \right \rangle_{Q'} & \text{ if }\btau \ell\neq \ell,\\
 \left \langle \widehat{S'_i}+ |c_{\ell i}|\widehat{S'_{\ell}},\,  \widehat{S'_j}+ |c_{\ell j}|\widehat{S'_{\ell}}  \right \rangle_{Q'} & \text{ if }\btau \ell= \ell.\end{array}\right.
\end{align*}
It follows that
\begin{align*}
\left \langle \res(F_{\ell}^+([S_i])),\res(F_{\ell}^+([S_j])) \right \rangle_{Q'}=\langle \widehat{S'_i},  \widehat{S'_j}\rangle_{Q'}=\langle \widehat{S_i},  \widehat{S_j}\rangle_{Q},
\end{align*}
since $\langle S_{\ell}',S'_{\btau \ell}\rangle_{Q'}=0=\langle S'_{\btau \ell},S_{\ell}'\rangle_{Q'}$, $\langle S_{\ell}',S_{\ell}'\rangle_{Q'}=1=\langle S_{\btau \ell}',S_{\btau \ell}'\rangle_{Q'}$.

Case (2): \underline{$i\in \{\ell, \btau\ell\}$, $j\notin \{\ell, \btau\ell\}$}. Without loss of generality, assume that $i=\ell$. As $\ell$ is a sink, there exist at least one arrow $\alpha:j\rightarrow \ell$ in $Q_0$. So there exists a string module $Y$ with its string $\ell\xleftarrow{\alpha} j\xrightarrow{\varepsilon_j}\btau j$ (see \cite{BR87} for the definitions of strings and string modules).
Then $Y,\E_j\in\Fac T$, and there exists a short exact sequence
$0\rightarrow S_\ell\rightarrow Y\rightarrow \E_j\rightarrow0$. Clearly,
$$\widehat{\res(F_{\ell}^+(X_{S_{\ell}}))}- \widehat{\res(F_{\ell}^+(T_{S_{\ell}}))}= \widehat{\res(F_{\ell}^+(Y))}- \widehat{\res(F_{\ell}^+(\E_j))}.$$
By direct calculations, we obtain that $\widehat{\res(F_{\ell}^+(Y))}- \widehat{\res(F_{\ell}^+(\E_j))}=-\widehat{S_{\ell}'}$.
Then
\begin{align*}
\left \langle \widehat{\res(F_{\ell}^+(X_{S_\ell}))} \right.
& \left. - \widehat{\res(F_{\ell}^+(T_{S_\ell}))}, \widehat{\res(F_{\ell}^+([S_j]))} \right\rangle_{Q'}\\
&= \left\{ \begin{array}{ll} \left \langle -\widehat{S_{\ell}'},\,   \widehat{S'_j}+ |c_{\ell j}|\widehat{S'_{\ell}} + |c_{\btau \ell,\,j}|\widehat{S'_{\btau \ell}} \right \rangle_{Q'} & \text{ if }\btau \ell\neq \ell,\\
 \left \langle -\widehat{S_{\ell}'},\,   \widehat{S'_j}+ |c_{\ell j}|\widehat{S'_{\ell}}  \right \rangle_{Q'} & \text{ if }\btau \ell= \ell.\end{array}\right.
 \end{align*}
It follows that
 \begin{align*}
\left \langle \widehat{\res(F_{\ell}^+(X_{S_\ell}))} \right.
 \left. - \widehat{\res(F_{\ell}^+(T_{S_\ell}))}, \widehat{\res(F_{\ell}^+([S_j]))} \right\rangle_{Q'}
= 0=\langle \widehat{S_\ell},  \widehat{S_j}\rangle_{Q},
\end{align*}
since $\langle\widehat{S_{\ell}'},   \widehat{S'_j}\rangle_{Q'}=-|c_{\ell j}|$.

Case (3): \underline{$j\in \{\ell, \btau\ell\}$, $i\notin \{\ell, \btau\ell\}$}. From Case (2), we obtain that
\begin{align*}
\left \langle\widehat{\res(F_{\ell}^+([S_i]))}, \right. &
\left.  \widehat{\res(F_{\ell}^+(X_{S_j}))}- \widehat{\res(F_{\ell}^+(T_{S_j}))}\right \rangle_{Q'}\\
&= \left\{ \begin{array}{ll} \left\langle    \widehat{S'_i}+ |c_{\ell i}|\widehat{S'_{\ell}} + |c_{\btau \ell,i}|\widehat{S'_{\btau \ell}} ,-\widehat{S_{\ell}'}\right\rangle_{Q'} & \text{ if }\btau\ell\neq\ell, \\
 \left\langle    \widehat{S'_i}+ |c_{\ell i}|\widehat{S'_{\ell}} ,-\widehat{S_{\ell}'}\right\rangle_{Q'} & \text{ if }\btau\ell\neq\ell.
 \end{array} \right.
 \end{align*}
It follows that
\begin{align*}
\left \langle\widehat{\res(F_{\ell}^+([S_i]))}, \right.
\left.  \widehat{\res(F_{\ell}^+(X_{S_j}))}- \widehat{\res(F_{\ell}^+(T_{S_j}))}\right \rangle_{Q'}= -|c_{\ell i}|=\langle \widehat{S_i},  \widehat{S_\ell}\rangle_{Q},
\end{align*}
since $\langle \widehat{S'_i},\widehat{S'_\ell}\rangle_{Q'}=0$.

Case (4): \underline{$i,j\in \{\ell, \btau\ell\}$}.
As in Case (2), we have
\[
\widehat{\res(F_{\ell}^+(X_{S_i}))}- \widehat{\res(F_{\ell}^+(T_{S_i}))}= -\widehat{S'_i},
\qquad
\widehat{\res(F_{\ell}^+(X_{S_j}))}- \widehat{\res(F_{\ell}^+(T_{S_j}))}= -\widehat{S'_j},
\]
and then the desired formula follows.
\end{proof}

\subsection{Isomorphisms $\Gamma_\ell$ from reflection functors}
   \label{subsec:RF}

We identify $\End_{\iLa}(T)^{op} \cong \bs_\ell(\iLa)$ by \eqref{cor:Tiso}. Then the functor $F: \Fac T \longrightarrow\mod(\End_{\iLa}(T)^{op})$ is equivalent to the functor
\[
F_\ell^+: \Fac T \longrightarrow\mod(\bs_\ell\iLa),
\]
which allows us to carry out computations effectively via quivers.
Recall $\sqq =\sqrt{q}.$ As the Euler forms are matched by Lemma \ref{lem:reflection functor preserve bilnear form}, the following proposition follows by Lemma~\ref{lem:Upsilon} (to which we refer for notations).

\begin{theorem}
  \label{thm:Gamma}
The isomorphism $\Gamma_\ell$ in (\ref{eqn:reflection functor 1}) induces the following isomorphism of $\imath$Hall algebras: 
\begin{eqnarray}
\Gamma_{\ell}:\tMH & \stackrel{\sim}{\longrightarrow} & \tMHl,
  \label{eqn:reflection functor 2} \\
\, [M]&\mapsto& \sqq^{\langle \res (T_M),\res(M)\rangle_Q}q^{-\langle T_M,M\rangle} [F_{\ell}^+(T_M)]^{-1}* [F_{\ell}^+(X_M)].
 \notag
\end{eqnarray}
\end{theorem}

For any $i\in \I$, denote by $S_i$ (respectively, $S_i'$) the simple $\iLa$-module (respectively, $\bs_\ell\iLa$-module), denote by $\E_i$ (respectively, $\E_i'$) the
generalized simple $\iLa$-module (respectively, $\bs_\ell\iLa$-module). We similarly define $\E_\alpha,\E_\beta'$ for $\alpha\in K_0(\mod(\K Q))$, $\beta\in K_0(\mod(\K(\bs_\ell Q)))$
in the (twisted) semi-derived Ringel-Hall algebras (where $\ell$ is a sink of $Q$).

Recall the root lattice $\Z^{\I}=\Z\alpha_1\oplus\cdots\oplus\Z\alpha_n$. We have an isomorphism of abelian groups $\Z^\I\rightarrow K_0(\mod(\K Q))$, $\alpha_i\mapsto \widehat{S_i}$. This isomorphism induces an action of the reflection $s_i$ on $K_0(\mod(\K Q))$.
Thus, for $\alpha\in K_0(\mod (\K Q))$ and $i\in\I$, $\E_{s_i\alpha} \in \tMH$ is well defined. Similarly,  we have $\E'_{s_i\alpha}\in \tMHl$.

\begin{proposition}
   \label{prop:reflection}
Let $(Q, \btau)$ be an $\imath$quiver, and $\ell\in Q_0$ be a sink.
Then the isomorphism $\Gamma_{\ell}:\tMH\xrightarrow{\sim}  \tMHl$ acts on the generators as follows:
\begin{align}
\Gamma_{\ell}([M])&= [F_{\ell}^+(M)], \quad \forall M\in\Fac (T),
 \label{eqn:reflection 1}
  \\
\Gamma_{\ell}([S_\ell])&= \left\{
\begin{array}{cc}\sqq[\E'_\ell]^{-1}* [S_{\btau \ell}'], &\text{ if }\btau \ell\neq \ell
\\ \,[\E'_\ell]^{-1}* [S_{\btau \ell}'], &\text{ if }\btau \ell= \ell, \end{array}\right.
\label{eqn:reflection 2}
 \\
\Gamma_{\ell}([S_{\btau \ell}])&=
\sqq [\E'_{\btau \ell}]^{-1} * [S_{\ell}'], \quad\quad \text{ if }\btau \ell\neq \ell,
\label{eqn:reflection 3}
 \\
\Gamma_{\ell}([\E_\alpha])&= [\E'_{\bs_{\ell}\alpha}],
\quad \forall\alpha\in K_0(\mod(\K Q)).\label{eqn:reflection 4}
\end{align}
\end{proposition}

\begin{proof}
The identity \eqref{eqn:reflection 1} follows from \eqref{eqn:reflection functor 2}.

To prove (\ref{eqn:reflection 4}), it suffices to check it for $\E_i$, for $i\in Q_0$.

Case (1): \underline{$i\notin \{\ell, \btau\ell\}$.} Then $\E_i\in\Fac T$, which implies that $\Gamma_{\ell}([\E_i])= [F_{\ell}^+([\E_i])]$.
If $i\notin \Omega(-,\ell)\cup \Omega(-,\btau \ell)$, then by the definition of $F_{\ell}^+$, we obtain that $\Gamma_{\ell}([\E_i])= [F_{\ell}^+([\E_i])]=[\E'_i]$.
If $i\in \Omega(-,\ell)\cup \Omega(-,\btau \ell)$, then $\E_i\in \Fac T$. By using $F_{\ell}^+$, we obtain that
\begin{align*}
\widehat{F_{\ell}^+(\E_i)}=\left\{\begin{array}{cc} |c_{\ell i}| \, \widehat{\E}'_a +\widehat{\E}'_i=\widehat{\E}'_{s_{\ell}(\widehat{S_i})}, & \text{ if }\btau \ell=\ell, \\
|c_{\ell i}|\widehat{\E}'_\ell +|c_{\btau \ell ,i}|\widehat{\E}'_{\btau \ell}+\widehat{\E}'_i=\widehat{\E}'_{s_{\ell}s_{\btau \ell}(\widehat{S_i})}, &\text{ if }\btau \ell\neq \ell
\end{array}\right.
\end{align*}
in $K_0(\cp^{<\infty}(s_{\ell}(\iLa)))$.
So we obtain that
\begin{align*}
\Gamma_{\ell}([\E_i])=\left\{ \begin{array}{cc}\E'_{s_{\ell}(\widehat{S_i})} &\text{ if }\btau \ell=\ell,\\\, \E'_{s_{\ell}s_{\btau \ell}(\widehat{S_i})} &\text{ if }\btau \ell\neq \ell.\end{array}\right.
\end{align*}

Case (2): \underline{$i\in\{\ell,\btau\ell\}$.} We only give the details for $i=\ell$.
Note that \eqref{eqn: pd of tau inverse of Ea}
has the property analogous to \eqref{eqn:resolution of M and N}.
By \cite[Theorem~A.22]{LW19}, as $F=\Hom_{\iLa}(T,-): \Fac T\rightarrow\mod(\bs_\ell\iLa)$ is a full embedding which also preserves all $\Ext$-groups, we obtain
\begin{align*}
\Gamma_{\ell}([\E_\ell])&= \sqq^{\langle \widehat{\res(\TrD(\E_{\btau\ell}) )}, \widehat{\res(\E_\ell)}\rangle_Q}q^{-\langle\TrD(\E_{\btau \ell}), \E_\ell\rangle}[F(\TrD(\E_{\btau \ell}))]^{-1}* [F(\oplus_{(\alpha:j\rightarrow \ell)\in Q_1}\Lambda^{\imath}e_j)]\\
&= [P_\ell']^{-1}* [\oplus_{(\alpha:\ell\rightarrow j)\in Q_1'} P_j']\\
&= [\E'_\ell]^{-1}=\E'_{-\widehat{S_\ell}}=\E_{\bs_{\ell}(\widehat{S_\ell})},
\end{align*}
where $P_j':=(\bs_\ell\iLa)e_j$ for any $j\in Q'_0$. The last equality follows from $\bs_{\ell}( \widehat{S_\ell})=-\widehat{S_\ell}$ by definition. The formula \eqref{eqn:reflection 4} is proved.

Below we shall prove \eqref{eqn:reflection 2}, while skipping the completely analogous remaining case \eqref{eqn:reflection 3}.

Since $\ell$ is a sink, there exists at least one arrow $\alpha:j\rightarrow \ell$ in $Q_0$. So there exists a string module $X$ with its string $\ell\xleftarrow{\alpha} j\xrightarrow{\varepsilon_j}\btau j$.
Then $X,\E_j\in\Fac T$, which admits a short exact sequence
$0\rightarrow S_\ell\rightarrow X\rightarrow \E_j\rightarrow0$. Then
\begin{align*}
\Gamma_{\ell}([S_\ell])&= \sqq^{\langle\res(\E_j),\res(S_\ell)\rangle_Q}q^{-\langle \E_j,S_\ell\rangle} [F_{\ell}^+(\E_j)]^{-1}* [F_{\ell}^+(X)]\\
&= \sqq^{\langle \widehat{S_{\btau j}}-\widehat{ S_j},\widehat{S_\ell}\rangle_Q} [\E'_{s_{\ell}s_{\btau \ell}(\widehat{S_j})}]^{-1}* [F_{\ell}^+(X)].
\end{align*}
By definition of $F_{\ell}^+$, similar to the above, we have
\begin{align*}
[F_{\ell}^+(X)]=&[\E'_{s_{\ell}s_{\btau \ell}(\widehat{S_j})-\widehat{S'_{\ell}}}\oplus S_{\btau \ell}']\\
=&\sqq^{-\langle  \res(\E'_{s_{\ell}s_{\btau \ell}(\widehat{S_j})}) , S'_{\btau \ell} \rangle_{Q'}+\langle S_{\ell}'\oplus S_{\btau \ell}',S'_{\btau \ell}\rangle_{Q'}}
q^{\langle s_{\ell}s_{\btau \ell}(\widehat{S_j}), S'_{\btau \ell} \rangle_{Q'}-\langle S_{\ell}',S_{\btau \ell}'\rangle_{Q'}}[\E'_{s_{\ell}s_{\btau \ell}(\widehat{S_j}) -\widehat{S'_{\ell}} }]*[S'_{\btau \ell}]\\
=&\sqq^{1-\delta_{\ell, \btau\ell}+c_{\btau \ell,\btau j}-c_{\btau \ell, j}}  [\E'_{s_{\ell}s_{\btau \ell}(\widehat{S_j})-\widehat{S'_{\ell}}}]*[S'_{\btau \ell}].
\end{align*}
where we have used the fact $S_{\btau \ell}'$ is injective in $\mod(\bs_\ell\iLa)$. Note that $c_{\btau \ell,\btau j}=c_{\ell j}=\langle \widehat{S_{j}}, \widehat{S_\ell}\rangle_Q$,
and $c_{\btau \ell, j}=\langle\widehat{ S_j},\widehat{S_{\btau \ell}}\rangle_Q=\langle\widehat{ S_{\btau j}},\widehat{S_{\ell}}\rangle_Q$. So by combining the above computations, we obtain
\begin{align*}
\Gamma_{\ell}([S_\ell])&= \sqq^{\langle \widehat{S_{\btau j}}-\widehat{ S_j},\widehat{S_\ell}\rangle_Q}  \sqq^{1-\delta_{\ell, \btau\ell}+c_{\btau \ell,\btau j}-c_{\btau \ell, j}}  [\E'_{s_{\ell}s_{\btau \ell}(\widehat{S_j})}]^{-1}* [\E'_{s_{\ell}s_{\btau \ell}(\widehat{S_j})-\widehat{S'_{\ell}}}]*[S'_{\btau \ell}]\\
&= \sqq^{1-\delta_{\ell, \btau\ell}}[\E'_\ell]^{-1}* [S_{\btau \ell}'],
\end{align*}
whence \eqref{eqn:reflection 2}. The proposition is proved.
\end{proof}

\begin{remark}
 \label{rem:FT}
Similar to \cite{SV99, XY01}, for a sink $\ell \in \I$, there exists a Fourier transform (which is an algebra isomorphism)
\begin{align}
\texttt{FT}_\ell : \tMHl \longrightarrow \tMH,
\end{align}
which maps $[S'_j]\mapsto [S_j]$, $[\E_j']\mapsto [\E_j]$, for each $j\in \I$; compare with Theorem \ref{thm:Ui=iHall}.
The composition of $\Gamma_\ell$ with $\texttt{FT}_\ell$ gives us an automorphism
$\texttt{FT}_\ell\circ\Gamma_\ell: \tMH \rightarrow \tMH.$
\end{remark}

\subsection{Formulas of $\Gamma_i$ for split Dynkin $\imath$quivers}
  \label{subsec:formulaRF}

In this subsection we consider the split Dynkin $\imath$quivers $(Q, \btau=\Id)$. Note that its Cartan matrix $C=(c_{ij})_{i,j \in \I}$ satisfies that $c_{ij}=0$ or $-1$ for $i\neq j$.  We shall work out the action of $\Gamma_i$ for a sink $i\in Q_0$ on the generators $[S_j]$.

\begin{lemma}\label{lemma:reflection for split involution}
Let $(Q,\Id)$ be the Dynkin $\imath$quiver with trivial involution. Then for any sink $i\in Q_0$, the isomorphism $\Gamma_{i}:\tMH\xrightarrow{\sim}  \tMHi$ sends $\Gamma_{i} ([\E_\alpha]) =[\E'_{\bs_i \alpha}]$, for $\alpha \in K_0(\mod(\K Q))$ and
\begin{align*}
\Gamma_i([S_j])= \left\{
\begin{array}{cl}
{} \frac{\sqq}{q-1}[S'_i]*[S'_j]-\frac{1}{q-1}[S_j']*[S_i'], & \qquad \text{ if }c_{ij}=-1, \\
{} [\E'_i]^{-1}*[S'_i], & \qquad \text{ if }j=i,\\
{}[S'_j], & \qquad \text{ if } c_{ij}=0.
\end{array}
\right.
\end{align*}
\end{lemma}

\begin{proof}
The action of $\Gamma_i$ on $[\E_\alpha]$ has been uniformly given by \eqref{eqn:reflection 4}.
By Proposition~\ref{prop:reflection}, it remains to prove the formula for $\Gamma_i([S_j])$ when $c_{ij}=-1$. In this case, there exists an arrow from $j$ to $i$. Let $X'_{ij}$ denote the indecomposable $\K Q'$-module with $\widehat{S'_i}+\widehat{S'_j}$ as its class in $K_0(\mod(\K Q'))$, viewed as a $\bs_i\Lambda^{\imath}$-module. Then $\Gamma_i([S_j])= [F_i^+(S_j)]=[X'_{ij}]$. Moreover, we have
\begin{align*}
\sqq[S'_i]*[S'_j]-[S_j']*[S_i']&= \sqq^{1+\langle S_i',S_j' \rangle_{Q'} } ([S_i'\oplus S_j']+(q-1)[X'_{ij}] )-\sqq^{\langle S_j',S_i' \rangle_{Q'} } ([S_i'\oplus S_j']\\
&= (q-1)[X'_{ij}].
\end{align*}
Hence the formula follows.
\end{proof}

\subsection{Formulas of $\Gamma_i$ for non-split Dynkin $\imath$quivers}
  \label{subsec:formulaRF2}

In this subsection we consider the non-split Dynkin quivers $(Q, \btau \neq \Id)$. We separate the cases of type ADE in \S\ref{subsec:A}-\S\ref{subsec:E6}, respectively. We shall work out the action of $\Gamma_i$, for a sink $i\in Q_0$, on the generators $[S_j]$. The action of $\Gamma_i$ on $[\E_\alpha]$ is uniformly given by \eqref{eqn:reflection 4} as before.

\subsubsection{Quasi-split type $A_{2r+1}$}
 \label{subsec:A}

In this subsection and \S\ref{subsec:Dn}--\ref{subsec:E6} below, we consider the quasi-split (but not split) $\imath$quantum groups of type ADE, respectively. We introduce the following notation
\begin{equation} \label{eq:v-comm}
[x,y]_v:=xy-v yx.
\end{equation}
We will often use this when setting $v=\sqq$.

Let us illustrate with a quiver of type $A_{2r+1}$ below, though any other orientation preserved by the diagram involution will also work:
\begin{center}\setlength{\unitlength}{0.7mm}
\vspace{-2.2cm}
\begin{equation}
\label{diagram of A}
 \begin{picture}(100,40)(0,20)
\put(0,10){$\circ$}
\put(0,30){$\circ$}
\put(50,10){$\circ$}
\put(50,30){$\circ$}
\put(72,10){$\circ$}
\put(72,30){$\circ$}
\put(92,20){$\circ$}
\put(0,6){$r$}
\put(-2,34){${-r}$}
\put(50,6){\small $2$}
\put(48,34){\small ${-2}$}
\put(72,6){\small $1$}
\put(70,34){\small ${-1}$}
\put(92,16){\small $0$}

\put(3,11.5){\vector(1,0){16}}
\put(3,31.5){\vector(1,0){16}}
\put(23,10){$\cdots$}
\put(23,30){$\cdots$}
\put(33.5,11.5){\vector(1,0){16}}
\put(33.5,31.5){\vector(1,0){16}}
\put(53,11.5){\vector(1,0){18.5}}
\put(53,31.5){\vector(1,0){18.5}}

\put(75,12){\vector(2,1){17}}
\put(75,31){\vector(2,-1){17}}
\end{picture}
\end{equation}
\vspace{0.3cm}
\end{center}

\begin{lemma}\label{lem:RefFunctorA}
Let $(Q, \btau \neq \Id)$ be a Dynkin $\imath$quiver of type $A_{2r+1}$ as in \eqref{diagram of A}. Assume $i\in\ci =\{0,\dots,r\}$  is a sink. Then the isomorphism $\Gamma_{i}:\tMH\xrightarrow{\sim}  \tMHi$ sends $\Gamma_{i} ([\E_\alpha]) =[\E'_{\bs_i \alpha}]$, for $\alpha \in K_0(\mod(\K Q))$. Moreover, for $i\neq0$,
\begin{eqnarray*}
{\Gamma}_i([S_j])&=&\left\{\begin{array}{cl}
\frac{\sqq}{q-1}[S'_i]*[S'_j]-\frac{1}{q-1}[S_j']*[S_i'],  & \text{ if }c_{ij}=-1 \text{ and } c_{\btau i,j}=0,\\
 \frac{\sqq}{q-1}[S'_{\btau i}]*[S'_j]-\frac{1}{q-1}[S_j']*[S_{\btau i}'] ,   &  \text{ if } c_{ij}=0 \text{ and }c_{\btau i,j}=-1,\\
 \frac{1}{(q-1)^2} \big[ \big[[S_j'],[S_i'] \big]_{\sqq}, [S'_{\btau i}] \big]_{\sqq} + [S'_j]*[\E_i'],  &  \text{ if } c_{ij}=-1 \text{ and }c_{\btau i,j}=-1,\\
 \sqq[\E'_i]^{-1}*[S_{\btau i}'], & \text{ if }j=i,\\
 \sqq[\E'_{\btau i}]^{-1}*[S_i'], & \text{ if }j=\btau i,\\
 \,[S'_j], &\text{ otherwise;}
 \end{array}   \right.
\\
{\Gamma}_0([S_j])&=&\left\{ \begin{array}{cl} \frac{\sqq}{q-1}[S'_0]*[S'_j]-\frac{1}{q-1}[S_j']*[S_0'],  & \text{ if }j=\pm1, \\\,
[\E'_0]^{-1}* [S_0'],  & \text{ if }j=0, \\
  {[S'_j]},  & \text{ otherwise.}
  \end{array} \right.
\end{eqnarray*}
\end{lemma}

\begin{proof}
For $i=0$, the second and third formulas follow by Proposition~\ref{prop:reflection}; the first formula can be proved similarly to Lemma ~\ref{lemma:reflection for split involution}.

Now assume $i\neq0$. Most formulas in this case follow again from Proposition~\ref{prop:reflection} and similar computations as in Lemma~ \ref{lemma:reflection for split involution}. We shall only verify the third (and most involved) formula, which is new; in this case, we have $j=0$.

Let $X'_{ij}$ (respectively, $X'_{\btau i,j}$) be the indecomposable $\K Q'$-module with $\widehat{S'_i}+\widehat{S'_j}$ (respectively, $\widehat{S'_{\btau i}}+\widehat{S'_j}$) as its class in $K_0(\mod(\K Q'))$, and let $Y_0'$ be the indecomposable $\K Q'$-module with
$\widehat{S'_0}+\widehat{S_{1}'}+\widehat{S_{-1}'}$ as its class in $K_0(\mod(\K Q'))$; all these $kQ'$-modules are then viewed as $\bs_i\Lambda^{\imath}$-modules. We compute

\begin{align*}
[S_j']*[S_i']*[S'_{\btau i}]&= [S_j'\oplus S_i'\oplus S'_{\btau i}] +(q-1)[S_j'\oplus \E_i'],
\\
[S_i']*[S_j']*[S'_{\btau i}]&= (\frac{1}{\sqq} ([S_j'\oplus S_i'] +(q-1)[X'_{ij}]) )*[S'_{\btau i}]\\
&= \frac{1}{\sqq} [S_j'\oplus S_i'\oplus S'_{\btau i}]+\frac{q^2-q}{\sqq}[S_j'\oplus \E_i']+\frac{q-1}{\sqq}[X'_{ij}\oplus S'_{\btau i}].
\\
[S'_{\btau i}]*[S_j']*[S_i']&= (\frac{1}{\sqq} ([S_j'\oplus S_i'] +(q-1)[X'_{ij}]) )*[S'_{\btau i}]\\
&= \frac{1}{\sqq} [S_j'\oplus S_i'\oplus S'_{\btau i}]+\frac{q^2-q}{\sqq}[S_j'\oplus \E_{\btau i}']+\frac{q-1}{\sqq}[X'_{\btau i,j}\oplus S'_{i}].
\\
[S'_{\btau i}]*[S_i']*[S_j']&= (\frac{1}{\sqq} ([S_j'\oplus S_i'] +(q-1)[X'_{ij}]) )*[S'_{\btau i}]\\
&= \frac{1}{q}[S_j'\oplus S_i'\oplus S'_{\btau i}] +\frac{q-1}{q} [X'_{\btau i,j}\oplus S_i']+\frac{q-1}{q}[X'_{ij}\oplus S'_{\btau i}]+(q-1)[S_j'\oplus \E'_i]\\
&+(q-1)[S_j'\oplus \E_{\btau i}']+\frac{(q-1)^2}{q}[Y_0'].
\end{align*}

Combining the above formulas, we have
\begin{align*}
\big[ \big[[S_j'],[S_i'] \big]_{\sqq}, [S'_{\btau i}] \big]_{\sqq}&= (q-1)^2[Y_0']-(q-1)^2[S_j]*[\E_i'].
\end{align*}

From Proposition~\ref{prop:reflection} and the definition of $F_{i}^+$, we obtain that
${\Gamma}_i([S_j])=[Y'_0],$
and then the desired formula follows.
\end{proof}

\subsubsection{Quasi-split type $D_{n}$}
 \label{subsec:Dn}

Let us illustrate with a quiver of type $D_{n}$  below, though any other orientation preserved by the diagram involution will also work:
\begin{center}\setlength{\unitlength}{0.7mm}
\vspace{-1cm}
\begin{equation}
\label{diagram of D}
 \begin{picture}(50,20)(30,-0.5)
\put(0,-1){$\circ$}
\put(0,-5){\tiny$1$}
\put(20,-1){$\circ$}
\put(20,-5){\tiny$2$}
\put(64,-1){$\circ$}
\put(58,-5){\tiny$n-2$}
\put(84,-10){$\circ$}
\put(80,-13){\tiny${n-1}$}
\put(84,9.5){$\circ$}
\put(84,12.5){\tiny${n}$}

\put(19.5,0){\vector(-1,0){16.8}}
\put(38,0){\vector(-1,0){15.5}}
\put(64,0){\vector(-1,0){15}}

\put(40,-1){$\cdots$}

\put(83.5,9.5){\vector(-2,-1){16.5}}
\put(83.5,-8.5){\vector(-2,1){16.5}}
\end{picture}
\end{equation}
\vspace{0.5cm}
\end{center}

\begin{lemma}
  \label{lem:RefFunctorD}
Let $(Q, \btau \neq \Id)$ be a Dynkin $\imath$quiver of type $D_{n}$ as in \eqref{diagram of D}. Assume that $i\in \ci $ is a sink. Then, for $1\leq i\leq n-2$, the isomorphism $\Gamma_{i}:\tMH\xrightarrow{\sim}  \tMHi$ sends $\Gamma_{i} ([\E_\alpha]) =[\E'_{\bs_i \alpha}]$, for $\alpha \in K_0(\mod(\K Q))$, and moreover,
\begin{eqnarray*}
{\Gamma}_i([S_j])&=&\left\{\begin{array}{ll}
\frac{\sqq}{q-1}[S'_i]*[S'_j]-\frac{1}{q-1}[S_j']*[S_i'],   & \text{ if }c_{ij}=-1,\\
\,[\E'_i]^{-1}*[S'_i],  &\text{ if }j=i,\\
  \,[S'_j],  &\text{ otherwise.}
 \end{array}   \right.
\end{eqnarray*}
For $i=n-1$, we have
\begin{eqnarray*}
{\Gamma}_{n-1}([S_j])&=&\left\{\begin{array}{ll}
\frac{1}{(q-1)^2} \big[ \big[[S_j'],[S_{n-1}'] \big]_{\sqq}, [S'_{n}] \big]_{\sqq} + [S'_j]*[\E_{n-1}'],  & \text{ if }j=n-2,\\
 \sqq[\E'_{n-1}]^{-1}*[S_{n}'],  & \text{ if }j=n-1,\\
 \sqq[\E'_{n}]^{-1}*[S_{n-1}'],  & \text{ if }j=n,\\
  \,[S'_j],  &\text{ otherwise.}
 \end{array}   \right.
\end{eqnarray*}
\end{lemma}

\begin{proof}
The proof is similar to the proofs of Lemmas \ref{lemma:reflection for split involution}--\ref{lem:RefFunctorA} (as all the rank 2 $\imath$subquivers of \eqref{diagram of D} appear in \eqref{diagram of A}), and will be skipped.
\end{proof}

\subsubsection{Quasi-split type $E_6$}
  \label{subsec:E6}

Let us illustrate  with a quiver of type $E_6$ below, though any other orientation preserved by the diagram involution will also work:
\begin{center}\setlength{\unitlength}{0.7mm}
\vspace{-2cm}
\begin{equation}
\label{diagram of E}
\begin{picture}(100,40)(0,20)
\put(10,6){\small${6}$}
\put(10,10){$\circ$}
\put(32,6){\small${5}$}
\put(32,10){$\circ$}

\put(10,30){$\circ$}
\put(10,33){\small${1}$}
\put(32,30){$\circ$}
\put(32,33){\small${2}$}

\put(31.5,11){\vector(-1,0){19}}
\put(31.5,31){\vector(-1,0){19}}

\put(52,22){\vector(-2,1){17.5}}
\put(52,20){\vector(-2,-1){17.5}}

\put(54.7,21.2){\vector(1,0){19}}

\put(52,20){$\circ$}
\put(52,23){\small$3$}
\put(74,20){$\circ$}
\put(74,23){\small$4$}
\end{picture}
\end{equation}
\vspace{0.3cm}
\end{center}

\begin{lemma}
  \label{lem:RefFunctorE}
Let $(Q, \btau \neq \Id)$ be a Dynkin $\imath$quiver of type $E_6$ as in \eqref{diagram of E}.  Then for any sink $i\in \ci $, the isomorphism $\Gamma_{i}:\tMH\xrightarrow{\sim}  \tMHi$ sends $\Gamma_{i} ([\E_\alpha]) =[\E'_{\bs_i \alpha}]$, for $\alpha \in K_0(\mod(\K Q))$, and moreover,
\begin{eqnarray*}
&&{\Gamma}_1([S_j])=\left\{  \begin{array}{cl}
\,  \frac{\sqq}{q-1}[S'_1]*[S'_2]-\frac{1}{q-1}[S_2']*[S_1'],           & \text{ if }j=2,\\\,
          \frac{\sqq}{q-1}[S'_{6}]*[S'_{5}]-\frac{1}{q-1}[S_{5}']*[S_{6}'],        & \text{ if }j=5,\\
 \sqq[\E'_1]^{-1}*[S_{6}'],             &\text{ if }j=1,\\
    \sqq[\E'_{6}]^{-1}*[S_{1}'],           &\text{ if }j=6, \\
 {} [S'_j], &  \text{ if }j=3,4;
\end{array}\right.
\end{eqnarray*}
\begin{eqnarray*}
&& {\Gamma}_2([S_j])=\left\{  \begin{array}{cl}
 \frac{1}{(q-1)^2} \big[ \big[[S_3'],[S_2'] \big]_{\sqq}, [S'_{5}] \big]_{\sqq} + [S'_3]*[\E_2'], &  \text{ if }j=3,\\\,
  \frac{\sqq}{q-1}[S'_2]*[S'_1]-\frac{1}{q-1}[S_1']*[S_2'],   & \text{ if }j=1,\\
 \frac{\sqq}{q-1}[S'_{5}]*[S'_{6}]-\frac{1}{q-1}[S_{6}']*[S_{5}'],    &  \text{ if }j=6,\\
  \sqq[\E'_2]^{-1}*[S_{5}'], & \text{ if }j=2,\\
 \sqq[\E'_{5}]^{-1}*[S_2'], & \text{ if }j=5,\\
 \,[S_4'], &\text{ if }j=4;
\end{array}\right.
\end{eqnarray*}
\begin{eqnarray*}
&& {\Gamma}_3([S_j])=\left\{  \begin{array}{cl}
 {}\frac{\sqq}{q-1}[S'_3]*[S'_j]-\frac{1}{q-1}[S_j']*[S_3'],              &\text{ if }j=2,4,5,  \\
{} [\E'_3]^{-1}*[S'_3], &  \text{ if }j=3,  \\
  {}         [S_j'],         & \text{ if }j=1,6;
\end{array}\right.
\end{eqnarray*}
and
\begin{eqnarray*}
&& {\Gamma}_4([S_j])=\left\{  \begin{array}{cl}
 \frac{\sqq}{q-1}[S'_4]*[S'_3]-\frac{1}{q-1}[S_3']*[S_4'],          & \text{ if }j=3,\\\,
[S'_j], &  \text{ if }j=1,2,5,6\\\,
  [\E'_4]^{-1}*[S'_4],       & \text{ if }j=4.
\end{array}\right.
\end{eqnarray*}
\end{lemma}

\begin{proof}
The proof is analogous to the proofs of Lemmas \ref{lemma:reflection for split involution}--\ref{lem:RefFunctorA}, and hence is skipped.
\end{proof}

\section{Symmetries of the $\imath$quantum group $\tUi$}
  \label{sec:symmetryUi}

In this section, we convert the isomorphisms for the $\imath$Hall algebras in the previous section to automorphisms $\tTT_i$ of $\tUi$, for $i\in \ci$. For Dynkin $\imath$quivers, we formulate explicitly the actions of $\tTT_i$ on generators of $\tUi$.

\subsection{Symmetries $\tTT_i$ of $\tUi$}

For a sink $i$ of a Dynkin $\imath$quiver $(Q,\btau)$, we denote $Q'= \bs_i Q$. Similar to the isomorphisms $\widetilde{\psi}_{Q}$ in \eqref{eq:psi2} and $\psi_{Q}$ in \eqref{eq:psi}, there exist algebra isomorphisms $\widetilde{\psi}_{Q'}: \tUi_{|v={\sqq}}\rightarrow \tMHi$ and $\widetilde{\psi}_{Q'}: \tUi \rightarrow \tMHgi$ with the right-hand sides of the formulas \eqref{eq:split}--\eqref{eq:extra} replaced by the counterparts in prime notation.

When combined with $\widetilde{\psi}_{Q}$ and $\widetilde{\psi}_{Q'}$, the isomorphism $\Gamma_i:\tMH\rightarrow \tMHi$ in \eqref{eqn:reflection functor 2} induces an isomorphism
 \[
 \Gamma_i: \tMHg \longrightarrow \tMHgi.
 \]

Define an algebra automorphism $\tTT_i \in \Aut (\tUi)$ so that the following diagram commutes:
\begin{align}
  \label{eq:defT}
\xymatrix{
\tUi  \ar[r]^{\tTT_i}  \ar[d]^{\widetilde{\psi}_Q} & \tUi \ar[d]^{\widetilde{\psi}_{Q'}} \\
\tMHg \ar[r]^{\Gamma_i} &  \tMHgi
}
\end{align}
In other words, specializing at $v=\sqq$ we have the following commutative diagram:
\begin{equation}
  \label{eq:Up=T}
\xymatrix{
\tUi _{|v={\sqq}}\ar[r]^{\tTT_i}  \ar[d]^{\widetilde{\psi}_Q} & \tUi _{|v={\sqq}}\ar[d]^{\widetilde{\psi}_{Q'}}
 \\
\tMH \ar[r]^{\Gamma_i} &  \tMHi
}
\end{equation}

It is also natural to formulate the following commutative diagram:
\begin{equation}
  \label{eq:Up=T2}
\xymatrix{
\tUi _{|v={\sqq}}\ar[r]^{\tTT_i}  \ar[d]^{\widetilde{\psi}_Q} & \tUi _{|v={\sqq}}\ar[d]^{\texttt{FT}_i \circ \widetilde{\psi}_{Q'}}
 \\
\tMH \ar[r]^{\texttt{FT}_i \circ \Gamma_i} &  \tMH
}
\end{equation}
where $\texttt{FT}_i : \tMHi \rightarrow \tMH$ denotes a Fourier transform as in Remark \ref{rem:FT}. Similar to \cite[Lemma 11.1]{SV99}, it follows by the diagram \eqref{eq:Up=T2} that the formula of $\tTT_i$ does not depend on the orientation of the quiver $Q$, thanks to \eqref{eqn:reflection 1}--\eqref{eqn:reflection 4}; this also follows from the description of the actions of $\tTT_i$ in Lemmas \ref{lemma:braid group of split involution}--\ref{lem:braid group of E} below.

\subsection{Formulas for $\tTT_i$}

 We shall use the diagram \eqref{eq:Up=T} to convert the formulas for the actions of $\Gamma_i$ on Hall algebras in Section~\ref{subsec:formulaRF} to formulas for $\tTT_i$ on $\tUi$.

For any $i\in\I$, define
\begin{align}
\label{def:b}
b_i:=\left\{ \begin{array}{ll} -v^2, & \text{ if }\btau i =i,
\\
1,&\text{ if }\btau i\neq i . \end{array}  \right.
\end{align}
For $\alpha=\sum_{i=1}^n d_i \alpha_i\in K_0(\mod(\K Q))$, define
\begin{equation}
  \label{def:b2}
b_{\alpha}:= \prod_{i=1}^n b_i^{d_i}.
\end{equation}

Then by the isomorphism $\widetilde \psi$ \eqref{eq:extra} and the diagram \eqref{eq:Up=T}, the formula \eqref{eqn:reflection 4} is converted to be
\begin{align}
\tTT_i(\tk_\alpha) =\frac{b_{\bs_i\alpha}}{b_\alpha}\tk_{\bs_i \alpha},\quad \forall \alpha\in \Z^\I.
\end{align}
\subsubsection{Split case}
\label{subsec:Braid split}

The formulas in Lemma \ref{lemma:reflection for split involution} are transformed via the diagram \eqref{eq:Up=T} into the following formulas.

\begin{lemma}  
   \label{lemma:braid group of split involution}
Assume $\btau=\Id$. Then there exists a unique automorphism $\tTT_i$, for $i \in \I$, of the $\Q(v)$-algebra $\tUi$ such that $\tTT_i(\tk_\alpha) =\frac{b_{\bs_i\alpha}}{b_\alpha}\tk_{\bs_i \alpha}$ for $\alpha \in \Z^\I$,  and
\[
\tTT_i(B_j)=\left\{ \begin{array}{cl}
B_jB_i-vB_iB_j,  & \text{ if }c_{ij}=-1, \\
 (-v^2\tk_i)^{-1} B_i,  &\text{ if }j=i,\\
B_j,  &\text{ if } c_{ij}=0.
\end{array}\right.
\]
\end{lemma}

\subsubsection{Quasi-split type $A_{2r+1}$}
 \label{subsec:AtUi}

The formulas in Lemma \ref{lem:RefFunctorA} are then transformed via the commutative diagram \eqref{eq:Up=T} into the following formulas.

\begin{lemma}
    \label{lem:braid group of A}
Let $(Q, \btau \neq \Id)$ be of type $A_{2r+1}$ (\ref{diagram of A}).
Then there exists a unique automorphism $\tTT_i$ ($0\leq i\leq r$) of $\tUi$ such that $\tTT_i(\tk_\alpha) =\frac{b_{\bs_i\alpha}}{b_\alpha}\tk_{\bs_i \alpha}$ for $\alpha \in \Z^\I$,  and for $1\leq i\leq r$,
\begin{align*}
\tTT_i(B_j) &=\left\{ \begin{array}{ll}
B_jB_i -vB_iB_j,  & \text{ if }c_{ij}=-1 \text{ and } c_{\btau i,j}=0,\\
B_{\btau i}B_j-v^{-1} B_jB_{\btau i},  & \text{ if } c_{ij}=0 \text{ and }c_{\btau i,j}=-1 ,\\
-{v^{-1} [[B_j,B_i]_v,B_{\btau i}]_v}+ B_j\tk_{i},  & \text{ if }i=\pm 1, j=0, \\
-\tk_{i}^{-1}B_{\btau i},  & \text{ if }j=i,\\
-v^2\tk_{\btau i}^{-1}  B_i,  &\text{ if }j=\btau i,\\
B_j, & \text{ otherwise;}
 \end{array}\right.
\\
\tTT_0(B_j) & =\left\{ \begin{array}{ll} B_jB_0 -vB_0B_j, & \text{ if } j=\pm1,\\
(-v^2\tk_0)^{-1} B_0,  & \text{ if } j=0, \\
B_j,  & \text{ otherwise.}  \end{array}\right.
\end{align*}
\end{lemma}

\subsubsection{Quasi-split type $D_{n}$}
 \label{subsec:DntUi}

The formulas in Lemma \ref{lem:RefFunctorD} are then transformed via the commutative diagram \eqref{eq:Up=T} into the following formulas.

\begin{lemma}
   \label{lem:braid group of D}
Let $(Q, \btau \neq \Id)$  be of type $D_n$ (\ref{diagram of D}). Then there exists a unique algebra automorphism $\tTT_i$ of $\tUi$, for $1\le i \le n-1$, such that $\tTT_i(\tk_\alpha) =\frac{b_{\bs_i\alpha}}{b_\alpha}\tk_{\bs_i \alpha}$ for $\alpha \in \Z^\I$, and for $1\leq i\leq n-2$,
\begin{align*}
\tTT_i(B_j)  &=\left\{ \begin{array}{ll}
B_jB_i-vB_iB_j,  &\text{ if }c_{ij}=-1,\\
 (-v^2\tk_i)^{-1} B_i,  & \text{ if } j=i,\\
B_j, & \text{otherwise;}
 \end{array}\right.
   \\ 
\tTT_{n-1}(B_j)  &=\left\{ \begin{array}{ll} - {v^{-1} [[B_j,B_{n-1}]_v,B_{n}]_v}+ B_j\tk_{n-1}, & \text{ if } j=n-2,\\\,
-\tk_{n-1}^{-1}B_{n}, & \text{ if }j=n-1,\\
 -v^2 \tk_{n}^{-1}B_{n-1},  & \text{ if }j=n,\\
  \,B_j, &\text{otherwise.}  \end{array}\right.
\end{align*}
%
\end{lemma}

\subsubsection{Quasi-split type $E_6$}
  \label{subsec:E6tUi}

The formulas in Lemma \ref{lem:RefFunctorE} are then transformed via the commutative diagram \eqref{eq:Up=T} into the following formulas.

\begin{lemma}
   \label{lem:braid group of E}
Let $(Q, \btau \neq \Id)$  be of type $E_6$  (\ref{diagram of E}). Then there exists a unique algebra automorphism $\tTT_i$ of $\tUi$ such that $\tTT_i(\tk_\alpha) =\frac{b_{\bs_i\alpha}}{b_\alpha}\tk_{\bs_i \alpha}$ for $\alpha \in \Z^\I$, and
\begin{eqnarray*}
&&\tTT_1(B_j)=\left\{  \begin{array}{ll}
     B_2B_1-vB_1B_2,       & \text{ if }j=2,\\
                    B_6B_5-v^{-1} B_5B_6,        & \text{ if }j=5,\\
 -\tk_1^{-1} B_6 ,            &\text{ if }j=1,\\
 -v^2\tk_6^{-1} B_1,         &\text{ if }j=6, \\
B_j, &  \text{ if }j=3,4;
\end{array}\right.
\end{eqnarray*}
\begin{eqnarray*}
&&\tTT_2(B_j)=\left\{  \begin{array}{ll}
 -v^{-1}  \big[[B_3,B_2]_v,B_5 \big]_v+B_3\tk_2, &  \text{ if }j=3,\\
  B_1B_2-vB_2B_1,   & \text{ if }j=1,\\
 B_5B_6-v^{-1} B_6B_5,    &  \text{ if }j=6,\\
   - \tk_2^{-1}B_5, & \text{ if }j=2,\\
  -v^2\tk_5^{-1} B_2,  & \text{ if }j=5,\\
B_4, &\text{ if }j=4;
\end{array}\right.
\end{eqnarray*}
\begin{align*}
\tTT_3(B_j)& =\left\{  \begin{array}{ll}
 B_jB_3-vB_3B_j,              &\text{ if }j=2,4,5,  \\
 (-v^2\tk_3)^{-1} B_3,&  \text{ if }j=3, \\
B_j, &  \text{ if }j=1, 6;
\end{array}\right.
\\
\tTT_4(B_j)& =\left\{  \begin{array}{ll}
B_3B_4-vB_4B_3,        & \text{ if }j=3,  \\
 (-v^2\tk_4)^{-1} B_4,       & \text{ if }j =4,  \\
B_j,      & \text{ if }j\neq3, 4.
\end{array}\right.
\end{align*}
\end{lemma}

\section{Braid group actions on $\imath$quantum groups of finite type}
  \label{sec:braid}

In this section, for Dynkin $\imath$quivers, we show that the symmetries $\tTT_i$ on $\tUi$ 
satisfy the braid group relations. The proofs here do not extend to more general $\imath$quivers, which would require to consider additional $\imath$subquivers.

\subsection{Reduction via $\imath$subquivers}

Let $({}'Q, {}'\btau)$ be an $\imath$subquiver of a Dynkin $\imath$quiver $(Q,\btau)$. Denote by ${}'\tUi$ and $\tUi$ their $\imath$quantum groups and denote by ${}'\tTT_i$ and $\tTT_i$ the isomorphisms defined in \eqref{eq:defT}, respectively.

\begin{lemma}
\label{lem: reduction1}
Let $({}'Q, {}'\btau)$ be an $\imath$subquiver of $(Q,\btau)$. Then ${}'\tUi$ is a subalgebra of $\tUi$, which is compatible with the action of $\tTT_i$,
i.e., ${}'\tTT_i(x) = \tTT_i(x)$ for any $x\in {}'\tUi$, and $i \in Q'$.
\end{lemma}


Let $(Q,\btau)$ be an $\imath$Dynkin quiver. Let $(Q^1,\btau)$ and $(Q^2,\btau)$ be two connected (full) $\imath$subquivers of $Q$ such that $Q^1\cap Q^2=\emptyset$. 
Let $\iLa$, ${}^1\iLa$, and ${}^2\iLa$ be the $\imath$quiver algebras associated to $(Q,\btau)$, $(Q^1,\btau)$, and $(Q^2,\btau)$ respectively. We identify $\mod({}^1\iLa)$ and $\mod({}^2\iLa)$ as subcategories of $\mod(\iLa)$, which are closed under taking extensions. So $\cs\cd\widetilde{\ch}({}^1\iLa),\cs\cd\widetilde{\ch}({}^2\iLa)$ can be viewed as subalgebras of $\tMH$.

 Assume that there exists an arrow $\alpha:i\rightarrow j$ such that $i\in Q^1$, $j\in Q^2$. Let $Q'$ be the quiver constructed from $Q$ by reversing $\alpha$ and $\btau(\alpha)$. Then $(Q,\btau)$ induces an $\imath$quiver $(Q',\btau')$. Denote by $'\iLa$ its $\imath$quiver algebra.
Then $\mod({}^1\iLa)$, $\mod({}^2\iLa)$ can be viewed as subcategories of $\mod('\iLa)$. So $\cs\cd\widetilde{\ch}({}^1\iLa),\cs\cd\widetilde{\ch}({}^2\iLa)$ can be viewed as subalgebras of $\cs\cd\widetilde{\ch}('\iLa)$.
We have isomorphisms
$\widetilde{\psi}:  \tUi_{|v={\sqq}}\rightarrow \tMH$ and $\widetilde{\psi}':  \tUi _{|v={\sqq}}\rightarrow \cs\cd\widetilde{\ch}('\iLa)$ by Theorem~ \ref{thm:Ui=iHall}.

\begin{lemma}
\label{lem:reduction3}
Retain the notation and assumption as above. For any $M\in \mod({}^1\iLa)$, $N\in \mod({}^2\iLa)$, we have
\begin{align*}
\widetilde{\psi}^{-1}([M])\cdot \widetilde{\psi}^{-1}( [N])= \widetilde{\psi}'^{-1}([M])\cdot \widetilde{\psi}'^{-1}([N])
\end{align*}
in $\tUi_{|v={\sqq}}$.
In particular, we have
$\widetilde{\psi}^{-1}([M]*[N])= \widetilde{\psi}'^{-1}([M]*[N]).$
\end{lemma}

\begin{proof}
Denote by ${}^1\tUi$, ${}^2\tUi$ the $\imath$quantum groups of $(Q^1,\btau^1)$, $(Q^2,\btau^2)$. Then ${}^1\tUi$, ${}^2\tUi$ are subalgebras of $\tUi$.
There exist two isomorphisms $\widetilde{\psi}^1: {}^1\tUi _{|v={\sqq}}\rightarrow \cs\cd\widetilde{\ch}({}^1\iLa)$ and $\widetilde{\psi}^2: {}^2\tUi_{|v={\sqq}} \rightarrow \cs\cd\widetilde{\ch}({}^2\iLa)$. In particular, $\widetilde{\psi}|_{{}^l\tUi} =\widetilde{\psi}^l=\widetilde{\psi}'|_{{}^l\tUi}$ for $l=1,2$.
So we have
\[
\widetilde{\psi}^{-1}([M])\cdot \widetilde{\psi}^{-1}( [N])= (\widetilde{\psi}^1)^{-1}([M])\cdot (\widetilde{\psi}^2)^{-1}( [N])
= \widetilde{\psi}'^{-1}([M])\cdot \widetilde{\psi}'^{-1}([N]).
\]
The lemma is proved.
\end{proof}

It is worth noting that the formula in Lemma \ref{lem:reduction3} also holds for $\tUi$ and generic Hall algebras by using the isomorphism $\widetilde \psi$ in \eqref{eq:psi2}, which will be used below.

\subsection{Braid relations I}

Recall that $\{\alpha_i\mid i\in\I\}$ is the set of simple roots. For any $\gamma =\sum_{i=1}^n a_i\alpha_i\in \Phi^\imath$, we shall write $\fu_{\gamma} = \fu_{\delta_{\gamma}}$ in generic Hall algebra,  where $\delta_{\gamma}$ is the characteristic function in $\tilde{\mathfrak{P}}^\imath$; cf. \eqref{eq:tPi}.

\begin{lemma}
\label{lem:BGij=0}
For any $i,j\in\ci$ such that $c_{ij}=0$, we have
$\tTT_i\tTT_j=\tTT_j\tTT_i$.
\end{lemma}

\begin{proof}
Let  $(Q,\btau)$ be the $\imath$quiver such that $\widetilde{\psi}:\tUi\xrightarrow{\sim}\tMHg$. Since $c_{ij}=0$, we  can assume that $i$ and $j$ are sink vertices of $Q$. Then
$i$ (and respectively, $j$) is a sink of $\bs_j Q$ (and respectively, $\bs_i Q$). We shall show that $\Gamma_i\Gamma_j(\fu_{\alpha_l})= \Gamma_j\Gamma_i(\fu_{\alpha_l})$ for any $i\in\I$.

For any $l\notin\{i,\btau i,j,\btau j\}$, from the identity \eqref{eqn:reflection 1} and Corollary~ \ref{lem:reflecting dimen}, we have
$\Gamma_i\Gamma_j(\fu_{\alpha_l})=\fu_{\bs_j\bs_i\alpha_l}$, and
$\Gamma_i\Gamma_j(\fu_{\alpha_l})=\fu_{\bs_j\bs_i\alpha_l}$. Since $c_{ij}=0$, we have
$\bs_i\bs_j=\bs_j\bs_i$, and then $\fu_{\bs_i\bs_j\alpha_l}=\fu_{\bs_j\bs_i\alpha_l}$. So $\Gamma_i\Gamma_j(\fu_{\alpha_l})=\Gamma_j\Gamma_i(\fu_{\alpha_l})$.

For any $l\in\{i,\btau i,j,\btau j\}$, we only prove the identity for $l=i$ since the other cases are similar. It follows from Proposition~ \ref{prop:reflection} that
\[
\Gamma_i\Gamma_j(\fu_{\alpha_i})=\Gamma_i(\fu_{\alpha_i})=v^{1-\delta_{i,\btau i}} \fu_{\upgamma_i}^{-1}*\fu_{\alpha_{\btau i}}.
\]
Therefore, we have
\[
\Gamma_j\Gamma_i(\fu_{\alpha_i})= \Gamma_j(v^{1-\delta_{i,\btau i}} \fu_{\upgamma_i}^{-1}*\fu_{\alpha_{\btau i}})=v^{1-\delta_{i,\btau i}} \fu_{\upgamma_i}^{-1}*\fu_{\alpha_{\btau i}}=\Gamma_i\Gamma_j(\fu_{\alpha_i})
\]
by noting that $c_{i,\btau j}=0$, $c_{\btau i,j}=0$ and $c_{\btau i,\btau j}=0$.
%
%
\end{proof}

The following preparatory lemma will be used in Lemmas~\ref{lem:bgrank2}--\ref{lem:bgij5} below.
\begin{lemma}
\label{lem: BG rank 2}
Let $i,j\in\ci$. Then, for any $l\in\{i,\btau i,j,\btau j\}$, we have
\begin{align}
\label{eqn:BG bij1}
\tTT_i\tTT_j\tTT_i(B_l)=\tTT_j\tTT_i\tTT_j(B_l),& \qquad \text{ if } \quad
{\tiny \xymatrix{ (\btau i =) i & j (=\btau j) \ar@{-}[l]} }
\text{ or }\;
\begin{picture}(30,13)(0,0)
\put(-2,-2){\tiny $\btau i$}
\put(5.5,0){\line(1,0){20}}
\put(27,-2){\tiny $\btau j$}
\put(1,7){\tiny $i$}
\put(5,9){\line(1,0){23}}
\put(29,7){\tiny $j$}
\end{picture}
%
\\
\label{eqn:BG bij2}
\tTT_i\tTT_j\tTT_i\tTT_j(B_l)=\tTT_j\tTT_i\tTT_j\tTT_i(B_l), & \qquad \text{ if }
\begin{picture}(100,40)(0,20)
\put(21,8){\tiny $i$}
\put(17,33){\tiny $\btau i$}
\put(49,20){\tiny $j (=\btau j)$}
\put(25,10){\line(2,1){23}}
\put(25,35){\line(2,-1){23}}
\end{picture}
\end{align}
\end{lemma}


%

\begin{proof}
From Lemma \ref{lem: reduction1}, without loss of generality, we assume that $\tUi$ is of rank 2, generated by $B_l$, $\tk_l$ for $l\in \{i,j,\btau i,\btau j\}$. Let  $(Q,\btau)$ be the $\imath$quiver and recall $\widetilde{\psi}:\tUi\stackrel{\sim}{\longrightarrow}\tMHg$ from \eqref{eq:psi2}.

We shall only prove \eqref{eqn:BG bij2}, while skipping a similar proof for \eqref{eqn:BG bij1}.





First, we assume that $j$ is a sink of $Q$. Then $\Gamma_i\Gamma_j\Gamma_i\Gamma_j$ is well defined since $i$ is a sink of $\bs_j Q$, and $j$ is a sink of $\bs_i\bs_j Q$, and so on.  Note
that $\bs_i \bs_j\bs_i\bs_j Q=Q$. So
\begin{align*}
\tTT_i\tTT_j\tTT_i\tTT_j(B_i)
=& \widetilde{\psi}^{-1}\Gamma_i\Gamma_j\Gamma_i\Gamma_j\widetilde{\psi}(B_i)
\\
=&\widetilde{\psi}^{-1}\Gamma_i\Gamma_j\Gamma_i \Gamma_j  (-\frac{1}{v^2-1}\fu_{\alpha_i})
=-\frac{1}{v^2-1}\widetilde{\psi}^{-1}\Gamma_i\Gamma_j\Gamma_i(\fu_{\alpha_{i}+\alpha_j})
\\
=&-\frac{1}{v^2-1}\widetilde{\psi}^{-1}\Gamma_i\Gamma_j(\fu_{\alpha_j+\alpha_{\btau i}})
=-\frac{1}{v^2-1}\widetilde{\psi}^{-1}\Gamma_i(\fu_{\alpha_{\btau i}})
\\
=&-\frac{1}{v^2-1}\widetilde{\psi}^{-1}(v \fu_{-\gamma_{\btau i}}* \fu_{\alpha_i} )
=v \tk_{\btau i}^{-1}B_i.
\end{align*}

On the other hand, assuming $i$ is a sink, we have
\begin{align*}
\tTT_j\tTT_i\tTT_j\tTT_i(B_i)
=& \widetilde{\psi}^{-1}\Gamma_j\Gamma_i\Gamma_j\Gamma_i\widetilde{\psi}(B_i)
\\
=&\widetilde{\psi}^{-1}\Gamma_j\Gamma_i\Gamma_j \Gamma_i  (-\frac{1}{v^2-1}\fu_{\alpha_i})
=-\frac{1}{v^2-1}\widetilde{\psi}^{-1}\Gamma_j\Gamma_i\Gamma_j(v\fu_{-\gamma_{ i}}*\fu_{\alpha_{\btau i}})
\\
=&-\frac{v}{v^2-1}\widetilde{\psi}^{-1}\Gamma_j\Gamma_i(\fu_{-\gamma_i-\gamma_j}* \fu_{\alpha_{\btau i+\alpha_j}})
= -\frac{v}{v^2-1}\widetilde{\psi}^{-1}\Gamma_j(\fu_{-\gamma_{\btau i}-\gamma_j} *\fu_{\alpha_{i}+\alpha_j})
\\=&-\frac{v}{v^2-1}\widetilde{\psi}^{-1}( \fu_{-\gamma_{\btau i}}* \fu_{\alpha_i} )
=v \tk_{\btau i}^{-1}B_i.
\end{align*}
This proves \eqref{eqn:BG bij2} for $l=i$.

For $l=j$ or $l=\btau i$, the proof is similar and hence omitted.
\end{proof}

\subsection{Braid relations II}

\begin{lemma}
\label{lem:bgrank2}
Suppose that $i,j\in\ci$ are contained in a Dynkin subdiagram of the form $\small \xymatrix{ (\btau i =) i & j (=\btau j) \ar@{-}[l]}$. Then we have
\begin{align*}
\tTT_i\tTT_j\tTT_i=\tTT_j\tTT_i\tTT_j.
\end{align*}
\end{lemma}

\begin{proof}

By definition,  from Lemma \ref{lemma:reflection for split involution} (or Lemma~\ref{lemma:braid group of split involution}), we have $\tTT_i(\tk_\alpha)=\frac{b_{s_i\alpha}}{b_\alpha}\tk_{s_i\alpha}, \text{ for }\alpha\in \Z^\I$, and
\begin{align}
\label{eqn:braid group3}
\tTT_i(B_l)=\left\{ \begin{array}{ll} (-v^2\tk_i)^{-1}B_i,  &\text{ if }l=i \\
B_l & \text{ if }c_{il}=0,\\
B_lB_i-vB_iB_l, & \text{ if }c_{il}=-1 . \end{array}\right.
\end{align}


From Lemma \ref{lem: BG rank 2}, it is enough to check for the case $c_{il}=-1$ or $c_{jl}=-1$. We only prove the identity in the case $c_{jl}=-1$ since the other case is similar. Note that $c_{il}=0$.
Then
\begin{align*}
\tTT_i\tTT_j\tTT_i(B_l)=& \tTT_i\tTT_j(B_l)=\tTT_i(B_lB_j-vB_jB_l   ) \\
=&B_lB_jB_i-vB_lB_iB_j-vB_jB_iB_l+v^2B_iB_jB_l.
\end{align*}

Similar to the proof of Lemma \ref{lem: BG rank 2}, we have
$\tTT_j\tTT_i(B_j)=B_i$.
Then
\begin{align*}
\tTT_j\tTT_i\tTT_j(B_l)=& \tTT_j (B_lB_jB_i-vB_lB_iB_j-vB_jB_iB_l+v^2B_iB_jB_l)
\\
=&\tTT_j (B_l) \tTT_j \tTT_i(B_j)- v  \tTT_j \tTT_i(B_j) \tTT_j(B_l)
 \\
 =&(B_lB_j-vB_jB_l) B_i-vB_i (B_lB_j-vB_jB_l)
  \\
=&B_lB_jB_i-vB_lB_iB_j-vB_jB_iB_l+v^2B_iB_jB_l
\\
=&\tTT_i\tTT_j\tTT_i(B_l).
\end{align*}

For $k_\alpha$, it is obvious that
$
\tTT_i\tTT_j\tTT_i(\tk_\alpha)=\tTT_j\tTT_i\tTT_j(\tk_\alpha).
$

Combining the above, we have proved
\begin{align*}
\tTT_i\tTT_j\tTT_i=\tTT_j\tTT_i\tTT_j.
\end{align*}
The lemma is proved.
\end{proof}

\subsection{Braid relations III}

\begin{lemma}
Suppose that $i,j\in\ci$ are contained in a Dynkin subdiagram
$
\begin{picture}(30,13)(0,0)
\put(-2,-2){\tiny $\btau i$}
\put(5.5,0){\line(1,0){20}}
\put(27,-2){\tiny $\btau j$.}
\put(1,7){\tiny $i$}
\put(5,9){\line(1,0){23}}
\put(29,7){\tiny $j$}
\end{picture}$
\;\; Then
\begin{align*}
\tTT_i\tTT_j\tTT_i=\tTT_j\tTT_i\tTT_j.
\end{align*}
\end{lemma}

\begin{proof}
By definition,  from Lemma \ref{lem:RefFunctorA} (or Lemma~\ref{lem:braid group of A}), we have $\tTT_i(\tk_\alpha) =\frac{b_{s_i\alpha}}{b_\alpha}\tk_{s_i\alpha}, \text{ for }\alpha\in \Z^\I$, and
\begin{align}
\label{eqn:braid2}
\tTT_i(B_l) &=\left\{ \begin{array}{ll}
B_lB_i -vB_iB_l,  & \text{ if }c_{il}=-1 \text{ and } c_{\btau i,l}=0,\\
B_{\btau i}B_l-v^{-1} B_lB_{\btau i},  & \text{ if } c_{il}=0 \text{ and }c_{\btau i,l}=-1 ,\\
-{v^{-1} [[B_l,B_i]_v,B_{\btau i}]_v}+B_lk_{i},  & \text{ if } c_{i, l}=-1, \btau l= l \\
-\tk_{i}^{-1}B_{\btau i},  & \text{ if }l=i,\\
-v^2\tk_{\btau i}^{-1} B_i,  &\text{ if }l=\btau i,\\
B_j, & \text{otherwise;}
 \end{array}\right.
\end{align}





From Lemma \ref{lem: BG rank 2}, it suffices to check for the case $l\notin \{i,j,\btau i,\btau j\}$, and $c_{il}=-1$ or $c_{jl}=-1$. We shall only prove the case when $c_{jl}=-1$ as the other case is similar. This case is divided into 2 subcases.

\underline{Subcase (i):}  $\btau l\neq l$. This subcase is similar to Lemma \ref{lem:bgrank2}, and we skip the details.
\vspace{2mm}


\underline{Subcase (ii):}  $\btau l=l$. Then $c_{jl}=-1=c_{\btau j,l}$. By Lemma \ref{lem: reduction1}, we may assume that $Q$ is $i\leftarrow j\leftarrow l\rightarrow \btau j \rightarrow\btau i$.
As in the proof of Lemma \ref{lem: BG rank 2}, $\Gamma_i\Gamma_j\Gamma_i$ is well defined, and
\begin{align}
 \label{eq:Tiji}
\tTT_i\tTT_j\tTT_i(B_l)=& \frac{-1}{v^2-1}\widetilde{\psi}_{Q'}^{-1}\Gamma_i\Gamma_j\Gamma_i(\fu_{\alpha_l} )
=\frac{-1}{v^2-1}\widetilde{\psi}_{Q'}^{-1}(\fu_{\alpha_i+ \alpha_{\btau i}+ \alpha_j +\alpha_{\btau j}+\alpha_l}).
\end{align}
Here we have denoted $Q'=i\rightarrow j\rightarrow l\leftarrow \btau j\leftarrow \btau i$. Note that $\tTT_i(B_l)=B_l$.

Let $'\iLa$ be the $\imath$quiver algebra of $(Q',\btau)$. We shall drop the notation $*$ for the multiplication in $\tMHg$ in this proof. Then
in $\tM ('\iLa)$, we have
\begin{align*}
[M_q & ({\alpha_i+\alpha_j})]  [M_q({\alpha_{\btau i}+\alpha_{ \btau j}})]  [M_q({\alpha_l})]
\\
=&\sqq^{-2} \big( [M_q({\alpha_i+\alpha_j}) \oplus M_q({\alpha_{\btau i} +\alpha_{\btau j}})   +(q-1) [M_q({\alpha_i +\alpha_j +\alpha_l}) \oplus M_q({\alpha_{\btau i} +\alpha_{\btau j}}) ]
\\
&+ (q-1) [M_q({\alpha_i +\alpha_j }) \oplus M_q({\alpha_{\btau i} +\alpha_{\btau j }+\alpha_l}) ] + (q-1)^2 [ M_q({\alpha_i+\alpha_j+\alpha_{\btau i} +\alpha_{\btau j} +\alpha_l})] \big)
\\
&+(q-1) [M_q({\upgamma_i})] [M_q({\upgamma_j})]   [M_q({\alpha_l})],
\\  \\
%
[M_q & ({\alpha_i+\alpha_j}) ]  [M_q({\alpha_l})] [M_q({\alpha_{\btau i}+\alpha_{ \btau j}})]
\\
=&\sqq^{-1} \big( [M_q({\alpha_i+\alpha_j}) \oplus M_q({\alpha_{\btau i} +\alpha_{\btau j}} ) ]  +(q-1) [M_q({\alpha_i +\alpha_j +\alpha_l} )\oplus M_q({\alpha_{\btau i} +\alpha_{\btau j}} )] \big)
\\
&+\sqq(q-1) [M_q({\upgamma_i})] [M_q({\upgamma_j} )]  [M_q({\alpha_l})],
\end{align*}
\begin{align*}
 [M_q& ({\alpha_{\btau i}+\alpha_{ \btau j}})]  [M_q({\alpha_i+\alpha_j} )] [ M_q({\alpha_l})]
\\
=&\sqq^{-1} \big( [M_q({\alpha_i+\alpha_j} )\oplus M_q({\alpha_{\btau i} +\alpha_{\btau j}} ) ]  +(q-1) [M_q({\alpha_i +\alpha_j })\oplus M_q({\alpha_{\btau i} +\alpha_{\btau j} +\alpha_l} )]  \big)
\\
&+\sqq(q-1) [M_q({\upgamma_{\btau i}})] [M_q({\upgamma_{\btau j}} )]  [M_q({\alpha_l})],
\\ \\
%
 [M_q& ({\alpha_l})] [ M_q({\alpha_{\btau i}+\alpha_{ \btau j}})]  [M_q({\alpha_i+\alpha_j} )]
\\
=&[M_q({\alpha_i+\alpha_j}) \oplus M_q({\alpha_{\btau i} +\alpha_{\btau j}}  )]  +(q-1)  [M_q({\alpha_l} )] [M_q({\upgamma_{\btau i}})] [ M_q({\upgamma_{\btau j}})].
\end{align*}
Then we obtain
\begin{align*}
[ M_q & ({\alpha_i+\alpha_j+\alpha_{\btau i} +\alpha_{\btau j} +\alpha_l})]
\\ \notag
=& \frac{1}{(q-1)^2} \big(q [M_q({\alpha_i+\alpha_j} )] [M_q({\alpha_{\btau i}+\alpha_{ \btau j}})]  [M_q({\alpha_l})]
\\
& - \sqq[ M_q({\alpha_i+\alpha_j})]   [M_q({\alpha_l})] [M_q({\alpha_{\btau i}+\alpha_{ \btau j}})]
\\
& -\sqq [M_q({\alpha_{\btau i}+\alpha_{ \btau j}})]  [M_q({\alpha_l})] [M_q({\alpha_i+\alpha_j} +   \fu_{\alpha_l})]  [M_q({\alpha_{\btau i}+\alpha_{ \btau j}})]  [M_q({\alpha_i+\alpha_j} )] \big) \notag
 \\
 &+[M_q({\alpha_l})]  [M_q({\upgamma_{\btau i}})] [M_q({\upgamma_{\btau j}})].\notag
\end{align*}

So in $\widetilde{\ch}( Q',\btau)$, we have
\begin{align}
\label{eqn: bgij1}
 \fu_{\alpha_i+\alpha_j+\alpha_{\btau i} +\alpha_{\btau j} +\alpha_l}
=& \frac{1}{(v^2-1)^2} \big (v^2 \fu_{\alpha_i+\alpha_j}  \fu_{\alpha_{\btau i}+\alpha_{ \btau j}}  \fu_{\alpha_l} - v \fu_{\alpha_i+\alpha_j}   \fu_{\alpha_l} \fu_{\alpha_{\btau i}+\alpha_{ \btau j}}
\\
&
 -v \fu_{\alpha_{\btau i}+\alpha_{ \btau j}}  \fu_{\alpha_l} \fu_{\alpha_i+\alpha_j} +   \fu_{\alpha_l}  \fu_{\alpha_{\btau i}+\alpha_{ \btau j}}  \fu_{\alpha_i+\alpha_j} \big) +\fu_{\alpha_l}  \fu_{\upgamma_{\btau i}} \fu_{\upgamma_{\btau j}}.\notag
\end{align}

Denote $Q''=i\rightarrow j\leftarrow l\rightarrow \btau j\leftarrow \btau i$.
It follows from \eqref{eqn: bgij1} and Lemma \ref{lem: reduction1} that
\begin{align*}
\tTT_j\tTT_i\tTT_j(B_l)
=&\frac{-1}{v^2-1}\tTT_j \widetilde{\psi}_{Q'}^{-1}(\fu_{\alpha_i+ \alpha_{\btau i}+ \alpha_j +\alpha_{\btau j}+\alpha_l})
\\
=&\frac{-1}{(v^2-1)^4}\tTT_j \widetilde{\psi}_{Q''}^{-1} \big(v^2 \fu_{\alpha_i+\alpha_j}  \fu_{\alpha_{\btau i}+\alpha_{ \btau j}}  \fu_{\alpha_l} - v \fu_{\alpha_i+\alpha_j}   \fu_{\alpha_l} \fu_{\alpha_{\btau i}+\alpha_{ \btau j}}
\\
& \qquad\qquad -v \fu_{\alpha_{\btau i}+\alpha_{ \btau j}}   \fu_{\alpha_l}  \fu_{\alpha_i+\alpha_j} +   \fu_{\alpha_l}  \fu_{\alpha_{\btau i}+\alpha_{ \btau j}}  \fu_{\alpha_i+\alpha_j} \big)
\\
&+\frac{-1}{(v^2-1)^2}\tTT_j \widetilde{\psi}_{Q''}^{-1}( \fu_{\alpha_l}  \fu_{\upgamma_{\btau i}} \fu_{\upgamma_{\btau j}}).
\end{align*}

Therefore, we have
\begin{align}
  \label{eq:TTT}
\tTT_j & \tTT_i\tTT_j(B_l)
\\
=&\frac{-1}{(v^2-1)^4}  \widetilde{\psi}_{\bs_jQ''}^{-1} \Gamma_j \big(v^2 \fu_{\alpha_i+\alpha_j}  \fu_{\alpha_{\btau i}+\alpha_{ \btau j}}  \fu_{\alpha_l} - v \fu_{\alpha_i+\alpha_j}   \fu_{\alpha_l} \fu_{\alpha_{\btau i}+\alpha_{ \btau j}}
 \notag
\\
&\qquad  -v \fu_{\alpha_{\btau i}+\alpha_{ \btau j}} \fu_{\alpha_l} \fu_{\alpha_i+\alpha_j}+   \fu_{\alpha_l}  \fu_{\alpha_{\btau i}+\alpha_{ \btau j}}  \fu_{\alpha_i+\alpha_j} \big)
 +\frac{-1}{(v^2-1)^2} \widetilde{\psi}_{\bs_jQ''}^{-1} \Gamma_j ( \fu_{\alpha_l}  \fu_{\upgamma_{\btau i}} \fu_{\upgamma_{\btau j}})
\notag
\\
=&\frac{-1}{(v^2-1)^4}  \widetilde{\psi}_{Q'}^{-1} \big(v^2 \fu_{\alpha_i}  \fu_{\alpha_{\btau i}}  \fu_{\alpha_l+\alpha_{j} +\alpha_{\btau j}} - v \fu_{\alpha_i}    \fu_{\alpha_l+\alpha_{j} +\alpha_{\btau j}} \fu_{\alpha_{\btau i}}
 \notag
\\
&-v \fu_{\alpha_{\btau i}} \fu_{\alpha_l+\alpha_{j} +\alpha_{\btau j}} \fu_{\alpha_i}+   \fu_{\alpha_l+\alpha_{j} +\alpha_{\btau j}}  \fu_{\alpha_{\btau i}}  \fu_{\alpha_i} \big)
+\frac{-1}{(v^2-1)^2} \widetilde{\psi}_{Q'}^{-1}(  \fu_{\alpha_l+\alpha_{j} +\alpha_{\btau j}} \fu_{\upgamma_{\btau i}}).
\notag
\end{align}
The last equality follows by using Lemma \ref{lem:reduction3} again.

Similar to \eqref{eqn: bgij1}, in $\widetilde{\ch}(Q',\btau)$, one can show that
\begin{align}
  \label{eq:uu}
\fu_{\alpha_i+ \alpha_{\btau i}+ \alpha_j +\alpha_{\btau j}+\alpha_l}
=& \frac{1}{(v^2-1)^2} \big(v^2 \fu_{\alpha_i}  \fu_{\alpha_{\btau i}}  \fu_{\alpha_l+\alpha_{j} +\alpha_{\btau j}} - v \fu_{\alpha_i}    \fu_{\alpha_l+\alpha_{j} +\alpha_{\btau j}} \fu_{\alpha_{\btau i}}
\\
&-v \fu_{\alpha_{\btau i}}  \fu_{\alpha_l+\alpha_{j} +\alpha_{\btau j}}\fu_{\alpha_i} +   \fu_{\alpha_l+\alpha_{j} +\alpha_{\btau j}}  \fu_{\alpha_{\btau i}}  \fu_{\alpha_i} \big)
+\fu_{\alpha_l+\alpha_{j} +\alpha_{\btau j}} \fu_{\upgamma_{\btau i}}.
\notag
\end{align}

Therefore, combining \eqref{eq:TTT}--\eqref{eq:uu}  and then comparing with \eqref{eq:Tiji} we obtain
\begin{align*}
\tTT_j\tTT_i\tTT_j(B_l)
=&  \frac{1}{(v^2-1)^2} \widetilde{\psi}_{Q'}^{-1} (\fu_{\alpha_i+ \alpha_{\btau i}+ \alpha_j +\alpha_{\btau j}+\alpha_l})
=\tTT_i\tTT_j\tTT_i(B_l).
\end{align*}

Also, it is clear that $\tTT_i\tTT_j\tTT_i(k_\alpha)=\tTT_j\tTT_i\tTT_j(k_\alpha).$

Summarizing the above, we have proved the lemma.
\end{proof}

\subsection{Braid relations  IV}

\begin{lemma}
\label{lem:bgij5}
Suppose that $i,j\in\ci$ are contained in a Dynkin subdiagram
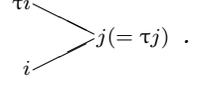
\begin{picture}(65,16)(16,20)
\put(21,8){\tiny $i$}
\put(17,33){\tiny $\btau i$}
\put(49,20){\tiny $j (=\btau j)$}
\put(25,10){\line(2,1){23}}
\put(25,35){\line(2,-1){23}}
\end{picture}.
Then we have
\begin{align*}
\tTT_i\tTT_j\tTT_i\tTT_j=\tTT_j\tTT_i\tTT_j\tTT_i.
\end{align*}
\end{lemma}

\begin{proof}
From Proposition~\ref{prop:reflection}, we have the action of $\tTT_i$ as given in \eqref{eqn:braid2}. On the other hand, the action of $\tTT_j$ is given in \eqref{eqn:braid group3}.

In this proof, we shall drop the notation $*$ when writing the multiplication in $\tMHg$. By Lemma \ref{lem: BG rank 2}, it remains to check $\tTT_i\tTT_j\tTT_i\tTT_j(B_l)=\tTT_j\tTT_i\tTT_j\tTT_i(B_l)$
in the 2 cases: (1) $c_{il}=-1$, or (2) $c_{jl}=-1$.

(1) Assume $c_{il}=-1$. We may assume that $Q = l\rightarrow i\rightarrow j\leftarrow \btau i\leftarrow \btau l$ by Lemma \ref{lem:reduction3}. Let $Q'=\bs_j\bs_i\bs_j Q$.
Then $\Gamma_j\Gamma_i\Gamma_j$ is well defined,
and
\begin{align*}
\tTT_i\tTT_j\tTT_i\tTT_j(B_l)
=& \tTT_i\widetilde{\psi}_{Q'}^{-1} \Gamma_j\Gamma_i\Gamma_j\widetilde{\psi}_Q(B_l)
\\
=&\frac{-1}{v^2-1}\tTT_i\widetilde{\psi}_{Q'}^{-1}\Gamma_j\Gamma_i\Gamma_j(\fu_{\alpha_l})
=\frac{-1}{v^2-1}\tTT_i\widetilde{\psi}_{Q'}^{-1}(\fu_{\alpha_i+\alpha_j+\alpha_l})
\\
=& \frac{-1}{(v^2-1)^2}\tTT_i\widetilde{\psi}_{Q'}^{-1}(v\fu_{\alpha_i+\alpha_j} \fu_{\alpha_l}-\fu_{\alpha_l} \fu_{\alpha_i+\alpha_j}).
\end{align*}
Let $Q''= l\rightarrow i\leftarrow j\rightarrow \btau i\leftarrow \btau l$.
Then Lemma \ref{lem:reduction3} shows that
\begin{align*}
\tTT_i\tTT_j\tTT_i\tTT_j(B_l)
=&\frac{-1}{(v^2-1)^2}\tTT_i\widetilde{\psi}_{Q''}^{-1}(v\fu_{\alpha_i+\alpha_j} \fu_{\alpha_l}-\fu_{\alpha_l} \fu_{\alpha_i+\alpha_j})
\\
=&\frac{-1}{(v^2-1)^2}\widetilde{\psi}_{\bs_iQ''}^{-1}\Gamma_i(v\fu_{\alpha_i+\alpha_j} \fu_{\alpha_l}-\fu_{\alpha_l} \fu_{\alpha_i+\alpha_j})
\\
=&\frac{-1}{(v^2-1)^2}\widetilde{\psi}_{\bs_iQ''}^{-1}(v\fu_{\alpha_{\btau i}+\alpha_j} \fu_{\alpha_l+\alpha_i}-\fu_{\alpha_l+\alpha_i} \fu_{\alpha_{\btau i}+\alpha_j}).
\end{align*}

On the other hand, since $j$ is a sink of $\bs_i Q''$, we can define $Q'''=\bs_j\bs_i Q''$. Then we have
\begin{align*}
&\tTT_j\tTT_i\tTT_j\tTT_i(B_l)
\\
=&\frac{-1}{(v^2-1)^2}\tTT_j\widetilde{\psi}_{\bs_iQ''}^{-1}(v\fu_{\alpha_{\btau i}+\alpha_j} \fu_{\alpha_l+\alpha_i}-\fu_{\alpha_l+\alpha_i} \fu_{\alpha_{\btau i}+\alpha_j})
\\
=&\frac{-1}{(v^2-1)^2} \widetilde{\psi}_{Q'''}^{-1}\Gamma_j(v\fu_{\alpha_{\btau i}+\alpha_j} \fu_{\alpha_l+\alpha_i}-\fu_{\alpha_l+\alpha_i} \fu_{\alpha_{\btau i}+\alpha_j})
\\
=&\frac{-1}{(v^2-1)^2} \widetilde{\psi}_{Q'''}^{-1}(v\fu_{\alpha_{\btau i}} \fu_{\alpha_i+\alpha_j+\alpha_l}-
\fu_{\alpha_i+\alpha_j+\alpha_l} \fu_{\alpha_{\btau i}})
\\
=&\frac{-1}{(v^2-1)^3} \widetilde{\psi}_{Q'''}^{-1} \big( v^2\fu_{\alpha_{\btau i}} \fu_{\alpha_j} \fu_{\alpha_l+\alpha_i} -v\fu_{\alpha_{\btau i}}  \fu_{\alpha_l+\alpha_i} \fu_{\alpha_j}-v\fu_{\alpha_j} \fu_{\alpha_l+\alpha_i} \fu_{\alpha_{\btau i}}
+\fu_{\alpha_l+\alpha_i} \fu_{\alpha_j} \fu_{\alpha_{\btau i}} \big)
\\
=&\frac{-1}{(v^2-1)^3} \widetilde{\psi}_{Q'''}^{-1} \big(v^2\fu_{\alpha_{\btau i}} \fu_{\alpha_j} \fu_{\alpha_l+\alpha_i} - v\fu_{\alpha_l+\alpha_i} \fu_{\alpha_{\btau i}}  \fu_{\alpha_j}-v\fu_{\alpha_j} \fu_{\alpha_{\btau i}}  \fu_{\alpha_l+\alpha_i}
+\fu_{\alpha_l+\alpha_i} \fu_{\alpha_j} \fu_{\alpha_{\btau i}} \big)  \\
&+\frac{-v}{(v^2-1)^2} \widetilde{\psi}_{Q'''}^{-1}(\fu_{\alpha_l} \fu_{\upgamma_i} \fu_{\alpha_j}-\fu_{\alpha_j} \fu_{\alpha_l} \fu_{\upgamma_i}).
\end{align*}
For the last equality above we have used that $\fu_{\alpha_{\btau i}} \fu_{\alpha_l+\alpha_i}=\fu_{\alpha_l+\alpha_i} \fu_{\alpha_{\btau i}}-(v^2-1)\fu_{\alpha_l} \fu_{\upgamma_i}$.
Since $\fu_{\alpha_l} \fu_{\upgamma_i} \fu_{\alpha_j}-\fu_{\alpha_j} \fu_{\alpha_l} \fu_{\upgamma_i}=0,$
we have
\begin{align*}
\tTT_j\tTT_i\tTT_j\tTT_i(B_l)
=&\frac{-1}{(v^2-1)^3}( \widetilde{\psi}_{Q'''}^{-1}(v^2\fu_{\alpha_{\btau i}} \fu_{\alpha_j})\cdot  \widetilde{\psi}_{Q'''}^{-1}(\fu_{\alpha_l+\alpha_i} ) -  \widetilde{\psi}_{Q'''}^{-1}(\fu_{\alpha_l+\alpha_i})\cdot \widetilde{\psi}_{Q'''}^{-1}(v\fu_{\alpha_{\btau i}}  \fu_{\alpha_j})
\\
&- \widetilde{\psi}_{Q'''}^{-1}(v\fu_{\alpha_j} \fu_{\alpha_{\btau i}} ) \cdot  \widetilde{\psi}_{Q'''}^{-1}(\fu_{\alpha_l+\alpha_i}) + \widetilde{\psi}_{Q'''}^{-1}(\fu_{\alpha_l+\alpha_i})\cdot  \widetilde{\psi}_{Q'''}^{-1}(\fu_{\alpha_j} \fu_{\alpha_{\btau i}})  ).
\end{align*}

Comparing $\bs_i Q''$ and $Q'''$, Lemma \ref{lem:reduction3}  implies that
$\widetilde{\psi}_{Q'''}^{-1} (\fu_{\alpha_j} \fu_{\alpha_{\btau i}})=\widetilde{\psi}_{\bs_iQ''}^{-1} (\fu_{\alpha_j} \fu_{\alpha_{\btau i}}) $, and
$\widetilde{\psi}_{Q'''}^{-1}(\fu_{\alpha_l+\alpha_i})=\widetilde{\psi}_{\bs_iQ''}^{-1}(\fu_{\alpha_l+\alpha_i} )$.
Then we obtain
\begin{align*}
\tTT_j & \tTT_i\tTT_j\tTT_i(B_l)
\\
=&\frac{-1}{(v^2-1)^3} \widetilde{\psi}_{\bs_iQ''}^{-1} (v^2\fu_{\alpha_{\btau i}} \fu_{\alpha_j}  \fu_{\alpha_l+\alpha_i} - v \fu_{\alpha_l+\alpha_i}  \fu_{\alpha_{\btau i}}  \fu_{\alpha_j}
-v\fu_{\alpha_j} \fu_{\alpha_{\btau i}}  \fu_{\alpha_l+\alpha_i}
+ \fu_{\alpha_l+\alpha_i}\fu_{\alpha_j} \fu_{\alpha_{\btau i}}  )
\\
=&\frac{-1}{(v^2-1)^3} \widetilde{\psi}_{\bs_iQ''}^{-1} \big(v (v\fu_{\alpha_{\btau i}} \fu_{\alpha_j}-  \fu_{\alpha_j} \fu_{\alpha_{\btau i}} )  \fu_{\alpha_l+\alpha_i} -  \fu_{\alpha_l+\alpha_i}  (v\fu_{\alpha_{\btau i}}  \fu_{\alpha_j}-\fu_{\alpha_j} \fu_{\alpha_{\btau i}}  ) \big)
\\
=&\frac{-1}{(v^2-1)^2}\widetilde{\psi}_{\bs_iQ''}^{-1}(v\fu_{\alpha_{\btau i+\alpha_j}} \fu_{\alpha_l+\alpha_i}-\fu_{\alpha_l+\alpha_i} \fu_{\alpha_{\btau i}+\alpha_j})
\\
=&\tTT_i\tTT_j\tTT_i\tTT_j(B_l).
\end{align*}

(2) Assume $c_{jl}=-1$. We can assume that $Q$ is
\[\xymatrix{
&\btau i \\
i & j \ar[l]  \ar[u]& l\ar[l]} \]
 by Lemma \ref{lem: reduction1}. Let $Q'=\bs_i\bs_j\bs_i(Q)$.
 Then $\Gamma_i \Gamma_j\Gamma_i$ is well defined, and
 \begin{align*}
\tTT_j\tTT_i\tTT_j\tTT_i(B_l)= & \tTT_j\widetilde{\psi}_{Q'}^{-1} \Gamma_j\Gamma_i\Gamma_j\widetilde{\psi}_Q(B_l)
\\
=&\frac{-1}{v^2-1}\tTT_j\widetilde{\psi}_{Q'}^{-1}\Gamma_j\Gamma_i\Gamma_j(\fu_{\alpha_l})
=\frac{-1}{v^2-1}\tTT_j\widetilde{\psi}_{Q'}^{-1}(\fu_{\alpha_i+\alpha_{\btau i} +\alpha_j+\alpha_l})
\\
=& \frac{-1}{(v^2-1)^2}\tTT_j\widetilde{\psi}_{Q'}^{-1}(v\fu_{\alpha_i+\alpha_{\btau i}+ \alpha_j} \fu_{\alpha_l}-\fu_{\alpha_l} \fu_{\alpha_i+\alpha_{\btau i}+\alpha_j}).
\end{align*}
Let $Q''=\bs_j \bs_l Q'$. Then Lemma \ref{lem:reduction3} shows that
 \begin{align*}
\tTT_j\tTT_i\tTT_j\tTT_i(B_l)
=& \frac{-1}{(v^2-1)^2}\tTT_j\widetilde{\psi}_{\bs_l Q'}^{-1}(v\fu_{\alpha_i+\alpha_{\btau i}+ \alpha_j} \fu_{\alpha_l}-\fu_{\alpha_l} \fu_{\alpha_i+\alpha_{\btau i}+\alpha_j})
\\
=&  \frac{-1}{(v^2-1)^2}  \widetilde{\psi}_{Q''}^{-1} (v\fu_{\alpha_i+\alpha_{\btau i}+\alpha_j }  \fu_{\alpha_l+\alpha_j} -\fu_{\alpha_l+\alpha_j} \fu_{\alpha_i+\alpha_{\btau i}+\alpha_j}  ) .
\end{align*}

On the other side, let $Q'''=\bs_i Q''$. Then
 \begin{align*}
\tTT_i\tTT_j\tTT_i\tTT_j(B_l)
=& \frac{-1}{(v^2-1)^2} \tTT_i \widetilde{\psi}_{Q''}^{-1} (v\fu_{\alpha_i+\alpha_{\btau i}+\alpha_j }  \fu_{\alpha_l+\alpha_j} -\fu_{\alpha_l+\alpha_j} \fu_{\alpha_i+\alpha_{\btau i}+\alpha_j}  )
\\
=&\frac{-1}{(v^2-1)^2}  \widetilde{\psi}_{Q'''}^{-1} (v\fu_{\alpha_j}  \fu_{\alpha_l+\alpha_j+\alpha_i+\alpha_{\btau i} } -
 \fu_{\alpha_l+\alpha_j+\alpha_i+\alpha_{\btau i} }   \fu_{\alpha_j} ).
\end{align*}

By definition of $Q'''$, $\Gamma_i\Gamma_j\Gamma_l$ is well defined. In particular,
$\bs_i\bs_j\bs_l Q'''= Q'$. So we have
\begin{align}
\label{eqn:bgl}
\tTT_i\tTT_j\tTT_l\tTT_i\tTT_j&\tTT_i\tTT_j(B_l)
= \frac{-1}{(v^2-1)^2}  \widetilde{\psi}_{Q'}^{-1} (v\fu_{\alpha_l}  \fu_{\alpha_j} -\fu_{\alpha_j}  \fu_{\alpha_l} )
\\
=&
\frac{-1}{(v^2-1)^2}  \widetilde{\psi}_{\bs_l Q'}^{-1} (v\fu_{\alpha_l}  \fu_{\alpha_j} -\fu_{\alpha_j}  \fu_{\alpha_l} )\notag
=\frac{-1}{v^2-1} \widetilde{\psi}_{\bs_l  Q'}^{-1}(\fu_{\alpha_l+\alpha_j}).\notag
\end{align}
The second equality above follows from Lemma \ref{lem:reduction3}.

Let $''\iLa$ be the $\imath$quiver algebra of $(Q'',\btau)$. Then in $\tM(''\iLa)$, we have the following:
\begin{align*}
[M_q({\alpha_i})] & [M_q({\alpha_{\btau i}})] [M_q({\alpha_j} )]
\\
=&[M_q({\alpha_j})\oplus M_q({\alpha_i})\oplus M_q({\alpha_{\btau i} })] + (q-1) [M_q({\upgamma_i} )]  [M_q({\alpha_j})],
\\
\\
[M_q({\alpha_j}) ] & [M_q({\alpha_{\btau i}})]  [M_q({\alpha_i})]
\\
=& q^{-1} \Big([M_q({\alpha_j})\oplus M_q({\alpha_i})\oplus M_q({\alpha_{\btau i} })]+ (q-1) [M_q({\alpha_j+\alpha_i}) \oplus M_q({\alpha_{\btau i}})]
\\
& +(q-1) [M_q({\alpha_j+\alpha_{\btau i}} )\oplus M_q({\alpha_i})] +(q-1)^2[M_q({\alpha_j+\alpha_i +\alpha_{\btau i}})] \Big)
\\
& +(q-1)[M_q({\alpha_j})] [M_q({\upgamma_{\btau i}})],
\end{align*}

\begin{align*}
 [M_q({\alpha_{\btau i}})] &  [M_q({\alpha_j})]  [M_q({\alpha_i})]
\\
=& \sqq^{-1} ([M_q({\alpha_j})\oplus M_q({\alpha_i})\oplus M_q({\alpha_{\btau i} })]+ (q-1) [M_q({\alpha_j+\alpha_i}) \oplus M_q({\alpha_{\btau i}})])
\\
&+\sqq(q-1)[M_q({\alpha_j})] [M_q({\upgamma_{\btau i}})],
\\
\\
%
[M_q({\alpha_{i}})] & [ M_q({\alpha_j}) ] [M_q({\alpha_{\btau i}})]
\\
=& \sqq^{-1} ([M_q({\alpha_j})\oplus M_q({\alpha_i})\oplus M_q({\alpha_{\btau i} })]+ (q-1) [M_q({\alpha_j+\alpha_{\btau i}}) \oplus M_q({\alpha_{ i}})])
\\
&+\sqq(q-1)[ M_q({\upgamma_{i}})] [M_q({\alpha_j})].
\end{align*}
Hence we obtain
\begin{align*}
[M_q & ({\alpha_i+\alpha_{\btau i}+\alpha_j })]
\\
=&\frac{1}{(q-1)^2} ( [M_q({\alpha_i})] [M_q({\alpha_{\btau i}})] [M_q({\alpha_j})] -\sqq[M_q({\alpha_{\btau i}} )]  [M_q({\alpha_j})]  [M_q({\alpha_i})]
\\
&-\sqq [M_q({\alpha_{i}})]   [M_q({\alpha_j})]  [M_q({\alpha_{\btau i}})]+ q[M_q( \fu_{\alpha_j}) ]  [M_q({\alpha_{\btau i}}) ] [M_q({\alpha_i})]) + [M_q({\upgamma_i})] [M_q({\alpha_j})].\notag
\end{align*}

Then in $\widetilde{\ch}(Q'',\btau)$, we have
\begin{align}
\fu_{\alpha_i+\alpha_{\btau i}+\alpha_j }
\label{eqn:bgij2}
=&\frac{1}{(v^2-1)^2} ( \fu_{\alpha_i} \fu_{\alpha_{\btau i}} \fu_{\alpha_j} -v\fu_{\alpha_{\btau i}}   \fu_{\alpha_j}  \fu_{\alpha_i}-v \fu_{\alpha_{i}}   \fu_{\alpha_j}  \fu_{\alpha_{\btau i}}+ v ^2 \fu_{\alpha_j}   \fu_{\alpha_{\btau i}}  \fu_{\alpha_i})
+ \fu_{\upgamma_i} \fu_{\alpha_j}.
\end{align}

 \begin{align*}
\tTT_j & \tTT_i\tTT_j\tTT_i(B_l)
\\
&=  \frac{-1}{(v^2-1)^4}  \widetilde{\psi}_{Q''}^{-1}
 \big (v
(   \fu_{\alpha_i} \fu_{\alpha_{\btau i}} \fu_{\alpha_j} -v\fu_{\alpha_{\btau i}}   \fu_{\alpha_j}  \fu_{\alpha_i}-v \fu_{\alpha_{i}}   \fu_{\alpha_j}  \fu_{\alpha_{\btau i}}+ v ^2 \fu_{\alpha_j}   \fu_{\alpha_{\btau i}}  \fu_{\alpha_i} ) \fu_{\alpha_l+\alpha_j}
\\
&-\fu_{\alpha_l+\alpha_j} ( \fu_{\alpha_i} \fu_{\alpha_{\btau i}} \fu_{\alpha_j} -v\fu_{\alpha_{\btau i}}   \fu_{\alpha_j}  \fu_{\alpha_i}-v \fu_{\alpha_{i}}   \fu_{\alpha_j}  \fu_{\alpha_{\btau i}}+ v ^2 \fu_{\alpha_j}   \fu_{\alpha_{\btau i}}  \fu_{\alpha_i}) \big)
\\
& +\frac{-1}{(v^2-1)^2}  \widetilde{\psi}_{Q''}^{-1} (v \fu_{\upgamma_i} \fu_{\alpha_j} \fu_{\alpha_l+\alpha_j}-  \fu_{\alpha_l+\alpha_j}   \fu_{\upgamma_i} \fu_{\alpha_j}).
\end{align*}
Comparing $Q''$ and $Q'''$, and using Lemma \ref{lem:reduction3}, we have
\begin{align*}
\tTT_j & \tTT_i\tTT_j\tTT_i(B_l)
\\
&=  \frac{-1}{(v^2-1)^4}  \widetilde{\psi}_{Q'''}^{-1}
\big(v
(   \fu_{\alpha_i} \fu_{\alpha_{\btau i}} \fu_{\alpha_j} -v\fu_{\alpha_{\btau i}}   \fu_{\alpha_j}  \fu_{\alpha_i}-v \fu_{\alpha_{i}}   \fu_{\alpha_j}  \fu_{\alpha_{\btau i}}+ v ^2 \fu_{\alpha_j}   \fu_{\alpha_{\btau i}}  \fu_{\alpha_i} ) \fu_{\alpha_l+\alpha_j}
\\
&-\fu_{\alpha_l+\alpha_j} ( \fu_{\alpha_i} \fu_{\alpha_{\btau i}} \fu_{\alpha_j} -v\fu_{\alpha_{\btau i}}   \fu_{\alpha_j}  \fu_{\alpha_i}-v \fu_{\alpha_{i}}   \fu_{\alpha_j}  \fu_{\alpha_{\btau i}}+ v ^2 \fu_{\alpha_j}   \fu_{\alpha_{\btau i}}  \fu_{\alpha_i}) \big)
\\
& +\frac{-1}{(v^2-1)^2}  \widetilde{\psi}_{Q'''}^{-1} (v \fu_{\upgamma_i} \fu_{\alpha_j} \fu_{\alpha_l+\alpha_j}-  \fu_{\alpha_l+\alpha_j}   \fu_{\upgamma_i} \fu_{\alpha_j}) .
\end{align*}

Similarly, we have
\begin{align*}
&\tTT_i\tTT_j\tTT_l\tTT_j\tTT_i\tTT_j\tTT_i(B_l)
= \frac{-1}{(v^2-1)^4}  \widetilde{\psi}_{Q'}^{-1}
\\
& \big(v
(   \fu_{\alpha_{\btau i} +\alpha_j} \fu_{\alpha_i +\alpha_j}  \fu_{\alpha_l} -v\fu_{\alpha_i+\alpha_j}   \fu_{\alpha_l}  \fu_{\alpha_{\btau i}+\alpha_j} -v \fu_{\alpha_{\btau i}+\alpha_j}   \fu_{\alpha_l}  \fu_{\alpha_i+\alpha_j}
\\
&+ v ^2 \fu_{\alpha_l}   \fu_{\alpha_i+\alpha_j}  \fu_{\alpha_{\btau i} +\alpha_j} ) \fu_{\alpha_i+\alpha_{\btau i}+\alpha_j}  \fu_{\upgamma_j}^{-1}  \fu_{\upgamma_i}^{-1}  \fu_{\upgamma_{\btau i}}^{-1}
\\
&-\fu_{\alpha_i+\alpha_{\btau i}+\alpha_j}   \fu_{\upgamma_j}^{-1}  \fu_{\upgamma_i}^{-1}  \fu_{\upgamma_{\btau i}}^{-1} (   \fu_{\alpha_{\btau i} +\alpha_j} \fu_{\alpha_i +\alpha_j}  \fu_{\alpha_l} -v\fu_{\alpha_i+\alpha_j}   \fu_{\alpha_l}  \fu_{\alpha_{\btau i}+\alpha_j}
\\
& -v \fu_{\alpha_{\btau i}+\alpha_j}   \fu_{\alpha_l}  \fu_{\alpha_i+\alpha_j}
+ v ^2 \fu_{\alpha_l}   \fu_{\alpha_i+\alpha_j}  \fu_{\alpha_{\btau i} +\alpha_j} )\big)
\\
& +\frac{-1}{(v^2-1)^2}  \widetilde{\psi}_{Q'}^{-1} (v \fu_{\upgamma_i}^{-1} \fu_{\alpha_l} \fu_{\alpha_i+\alpha_{\btau i}+\alpha_j}-  \fu_{\alpha_i+\alpha_{\btau i}+\alpha_j}   \fu_{\upgamma_i}^{-1} \fu_{\alpha_l})
\\
=& \frac{-1}{(v^2-1)^4}  \widetilde{\psi}_{\bs_lQ'}^{-1} \big(v
(   \fu_{\alpha_{\btau i} +\alpha_j} \fu_{\alpha_i +\alpha_j}  \fu_{\alpha_l} -v\fu_{\alpha_i+\alpha_j}   \fu_{\alpha_l}  \fu_{\alpha_{\btau i}+\alpha_j} -v \fu_{\alpha_{\btau i}+\alpha_j}   \fu_{\alpha_l}  \fu_{\alpha_i+\alpha_j}
\\
&+ v ^2 \fu_{\alpha_l}   \fu_{\alpha_i+\alpha_j}  \fu_{\alpha_{\btau i} +\alpha_j} ) \fu_{\alpha_i+\alpha_{\btau i}+\alpha_j}  \fu_{\upgamma_j}^{-1}  \fu_{\upgamma_i}^{-1}  \fu_{\upgamma_{\btau i}}^{-1}
\\
&-\fu_{\alpha_i+\alpha_{\btau i}+\alpha_j}   \fu_{\upgamma_j}^{-1}  \fu_{\upgamma_i}^{-1}  \fu_{\upgamma_{\btau i}}^{-1} (   \fu_{\alpha_{\btau i} +\alpha_j} \fu_{\alpha_i +\alpha_j}  \fu_{\alpha_l} -v\fu_{\alpha_i+\alpha_j}   \fu_{\alpha_l}  \fu_{\alpha_{\btau i}+\alpha_j}
\\
& -v \fu_{\alpha_{\btau i}+\alpha_j}   \fu_{\alpha_l}  \fu_{\alpha_i+\alpha_j}
+ v ^2 \fu_{\alpha_l}   \fu_{\alpha_i+\alpha_j}  \fu_{\alpha_{\btau i} +\alpha_j} )\big)
\\
& +\frac{-1}{(v^2-1)^2}  \widetilde{\psi}_{\bs_lQ'}^{-1} (v \fu_{\upgamma_i}^{-1} \fu_{\alpha_l} \fu_{\alpha_i+\alpha_{\btau i}+\alpha_j}-  \fu_{\alpha_i+\alpha_{\btau i}+\alpha_j}   \fu_{\upgamma_i}^{-1} \fu_{\alpha_l}) .
\end{align*}

Similar to \eqref{eqn:bgij2}, in $\widetilde{\ch}(\bs_l Q',\btau)$, we have
\begin{align*}
 \fu_{\alpha_{\btau i} +\alpha_j}  & \fu_{\alpha_i +\alpha_j}  \fu_{\alpha_l} -v\fu_{\alpha_i+\alpha_j}   \fu_{\alpha_l}  \fu_{\alpha_{\btau i}+\alpha_j} -v \fu_{\alpha_{\btau i}+\alpha_j}   \fu_{\alpha_l}  \fu_{\alpha_i+\alpha_j}
\\
&+ v ^2 \fu_{\alpha_l}   \fu_{\alpha_i+\alpha_j}  \fu_{\alpha_{\btau i} +\alpha_j} +(v^2-1)^2 \fu_{\alpha_l} \fu_{\upgamma_{\btau i}}  \fu_{\upgamma_j}
\\
=&(v^2-1)^2\fu_{\alpha_i+\alpha_{\btau i}+2\alpha_j+\alpha_l}.
\end{align*}

Then we have
\begin{align*}
\tTT_i & \tTT_j\tTT_l\tTT_j\tTT_i\tTT_j\tTT_i(B_l)
\\
=&\frac{-1}{(v^2-1)^2}  \widetilde{\psi}_{\bs_lQ'}^{-1} \big( (v \fu_{\alpha_i+\alpha_{\btau i}+2\alpha_j+\alpha_l}   \fu_{\alpha_i+\alpha_{\btau i}+\alpha_j} -  \fu_{\alpha_i+\alpha_{\btau i}+\alpha_j}   \fu_{\alpha_i+\alpha_{\btau i}+2\alpha_j+\alpha_l} )  \fu_{\upgamma_j}^{-1}  \fu_{\upgamma_i}^{-1}  \fu_{\upgamma_{\btau i}}^{-1} \big)
\\
&-\frac{-1}{(v^2-1)}  \widetilde{\psi}_{\bs_lQ'}^{-1}( v\fu_{\alpha_l} \fu_{\upgamma_i}^{-1}  \fu_{\alpha_i+\alpha_{\btau i}+\alpha_j} -\fu_{\alpha_i +\alpha_{\btau i}+\alpha_j}  \fu_{\alpha_l} \fu_{\upgamma_i}^{-1})
\\
&+\frac{-1}{(v^2-1)^2}  \widetilde{\psi}_{\bs_lQ'}^{-1} (v \fu_{\upgamma_i}^{-1} \fu_{\alpha_l} \fu_{\alpha_i+\alpha_{\btau i}+\alpha_j}-  \fu_{\alpha_i+\alpha_{\btau i}+\alpha_j}   \fu_{\upgamma_i}^{-1} \fu_{\alpha_l})
\\
=&\frac{-1}{(v^2-1)^2}  \widetilde{\psi}_{\bs_lQ'}^{-1} \big( (v \fu_{\alpha_i+\alpha_{\btau i}+2\alpha_j+\alpha_l}   \fu_{\alpha_i+\alpha_{\btau i}+\alpha_j} -  \fu_{\alpha_i+\alpha_{\btau i}+\alpha_j}   \fu_{\alpha_i+\alpha_{\btau i}+2\alpha_j+\alpha_l} )  \fu_{\upgamma_j}^{-1}  \fu_{\upgamma_i}^{-1}  \fu_{\upgamma_{\btau i}}^{-1} \big)
\\
=&\frac{-1}{v^2-1}  \widetilde{\psi}_{\bs_lQ'}^{-1}(\fu_{\alpha_j+\alpha_l}),
\end{align*}
where the last equality follows from a direct computation.

Comparing the formula above with \eqref{eqn:bgl}, we have established
$
\tTT_i\tTT_j\tTT_l\tTT_j\tTT_i\tTT_j\tTT_i(B_l)=\tTT_i\tTT_j\tTT_l\tTT_i\tTT_j\tTT_i\tTT_j(B_l),
$ 
and whence, $\tTT_j\tTT_i\tTT_j\tTT_i(B_l)=\tTT_i\tTT_j\tTT_i\tTT_j(B_l)$, since $\tTT_i\tTT_j\tTT_l$ is an isomorphism.
The lemma is proved.
\end{proof}

\subsection{Braid group actions on $\tUi$}

Recall from Lemma~\ref{lem:iWeyl} that the restricted Weyl group $W_{\btau}$ of $(\fg, \fg^\theta)$  is a finite reflection group of type $X_k$ generated by the simple reflections $\bs_i$, for $i \in \ci$. We shall denote by ${\rm Br}(W)$ the braid group associated to the Weyl group $W$ of $\fg$, and denote by ${\rm Br}(X_k)$ the braid group associated to a Weyl group of type $X_k$. Then the braid group associated to the restricted Weyl group for $(\fg, \fg^\theta)$ is of the form
\begin{equation}
  \label{eq:braidCox}
\brW =\langle t_i \mid i\in \I_\btau \rangle
\end{equation}
where $t_i$ satisfy the same braid relations as for $\bs_i$ in $W_{\btau}$. Then
\begin{align}
  \label{eq:BrBr}
\brW=\left\{\begin{array}{ll}
{\;\;}  {\rm Br} (W),   & \text{ if } \btau =\Id ,
\\
\left.\begin{array}{ll}
 {\rm Br} (B_{r}) , & \text{ if } \Delta \text{ is of type }A_{2r+1},\\
{\rm Br} (B_{n-1}) , & \text{ if } \Delta \text{ is of type }D_n,\\
 {\rm Br} (F_4) ,  & \text{ if } \Delta \text{ is of type }E_6,\end{array} \right\}& \text{ if }\btau \neq\Id.\end{array}   \right.
\end{align}

\begin{theorem}
\label{thm:bgtui}
Let $(Q,\btau)$ be a Dynkin $\imath$quiver. Recall $\tTT_i$ from \eqref{eq:defT}. Then there exists a homomorphism
$\brW \rightarrow \aut( \tUi)$, $t_i\mapsto \tTT_i$, for all $i\in \ci$.
\end{theorem}

\begin{proof}
It follows from Lemmas \ref{lem:BGij=0}--\ref{lem:bgij5}.
\end{proof}

For any acyclic $\imath$quiver, we shall continue to use the same definition of $W_\btau$ in \eqref{eq:Wtau} in Kac-Moody setting. We define $\brW$ to be the subgroup of $\text{Br}(W)$ which consists of elements commuting with $\btau$. We expect Theorem~\ref{thm:bgtui} to be valid in general, assuming the validity of \cite[Conjecture 7.9]{LW19} (now a theorem in \cite{LW20}), which states that $\widetilde{\psi}: \tUi \longrightarrow \tMHg$ is an injective homomorphism; compare with Theorem~\ref{generic U MRH} and \eqref{eq:defT}.

\begin{conjecture}
  \label{conj:braidKM}
Let $(Q,\btau)$ be an acyclic $\imath$quiver. Then the following holds.

(1) The commutative diagram \eqref{eq:defT} induces an automorphism $\TT_i \in \Aut (\tUi)$.

(2) There exists a homomorphism
$\brW \rightarrow \aut( \tUi)$, $t_i\mapsto \tTT_i$, for all $i\in \ci$.
\end{conjecture}

\section{$\imath$Quantum groups at distinguished parameters}
  \label{sec:distinguished}

In this short section, we explain why the distinguished parameter $\bvsd$ allows us to descend earlier results on $\tMH$ and $\tUi$ to $\rMH_\bvsd$ and $\Ui_\bvsd$. We then clarify the relation between the braid group action on  $\Ui_\bvsd$ and an earlier version due to Kolb-Pellegrini \cite{KP11}.

\subsection{Distinguished parameters}
 \label{subsec:disting}

Set $\bvsd =(\vs_{\diamond,i})_{i\in \I}$, where
\begin{equation}
  \label{eq:bvsd}
\vs_{\diamond,i}=-v^{-2} \;  \text{ if } i= \btau i,
\qquad
\vs_{\diamond,i}=1\;  \text{ if } i\neq \btau i.
\end{equation}
 We call the parameters $\bvsd$ 
 {\em distinguished}, for the corresponding $\imath$quantum groups $\tUi$ and $\Ui$.

More explicitly, in the split case (i.e., when $\btau=\Id$), the parameter $\bvsd$ is equal to
\begin{equation}
  \label{eq:special split}
\bvsd=(-v^{-2}, -v^{-2}, \dots, -v^{-2}).
\end{equation}
For type $A_{2r+1}$ with labels as in (\ref{diagram of A}) and nontrivial $\btau$,  the parameter $\bvsd$ is equal to
\begin{equation}
  \label{eq:specialA}
\bvsd =(\vs_0, \vs_1, \dots, \vs_r)=(-v^{-2}, 1,\dots,1)
\end{equation}
For type $D_n$ with labels as in  (\ref{diagram of D}) and nontrivial $\btau$, the parameter $\bvsd$ is equal to
\begin{equation}
  \label{eq:specialD}
  \bvsd =(\vs_1,\vs_2,\dots,\vs_{n-2},\vs_{n-1},\vs_n)=(-v^{-2}, -v^{-2}, \dots, -v^{-2}, 1, 1).
\end{equation}
For type $E_6$ with labels as in  (\ref{diagram of E}) and nontrivial $\btau$, the parameter $\bvsd$ is equal to
\begin{equation}
  \label{eq:specialE}
  \bvsd =(\vs_1,\vs_2,\vs_3,\vs_4,\vs_{5},\vs_{6})=(1,1,-v^{-2}, -v^{-2}, 1,1).
\end{equation}
\subsection{Reduced $\imath$Hall algebras}
 \label{subsec:Ui-disting}

We shall use the index $\bvsd$ to indicate the algebras under consideration are associated to the distinguished parameters $\bvsd$. Let $\La^\imath$ be the $\imath$quiver algebra associated to a Dynkin $\imath$quiver $(Q, \btau)$. Recall the reduced $\imath$Hall algebra $\rMH$ with a general parameter $\bvs$; cf. \eqref{eqn: reduce1}. If follows that the reduced $\imath$Hall algebra $\rMHd$ with the distinguished parameter $\bvsd$  is  the quotient algebra of $\tMH$ by the ideal generated by
\begin{align}
\label{ideal of MRH}
[\E_i]-1 \; (\forall i\in \I \text{ with } \btau i=i), \text{ and }[\E_i]*[\E_{\btau i}]-1\; (\forall i\in \I \text{ with }\btau i\neq i).
\end{align}

\begin{proposition}
   \label{prop:derived invariant of QSP}
For any sink $i\in Q_0$, the isomorphism $\Gamma_i$ induces an isomorphism of algebras
$\bar{\Gamma}_i: \rMHd \xrightarrow{\sim} \rMHi$.
\end{proposition}

\begin{proof}
A direct computation shows that $\Gamma_i: \tMH\stackrel{\sim}{\longrightarrow}  \tMHi$ in Proposition~\ref{prop:reflection} preserves the ideal generated by
$[\E_j]-1$ if $\btau j=j$, and $[\E_j]*[\E_{\btau j}]-1$ if $\btau j\neq j$. The assertion follows.
\end{proof}

Recall from Theorem \ref{thm:Ui=iHall} that the isomorphism $\psi_{Q}: \Ui_{\bvsd |v={\sqq}} \rightarrow \rMHd$ sends
\begin{align*}
B_j \mapsto \frac{-1}{q-1}[S_{j}],\text{ if } j\in\ci;
&\qquad\qquad
B_{j} \mapsto \frac{\sqq}{q-1}[S_{j}],\text{ if }j\notin \ci;
 \notag \\
k_j \mapsto [\E_j],\quad \forall  j \in \I.
\notag
\end{align*}
Similarly, there exists an isomorphism of algebras $\psi_{Q'} : \Ui_{\bvsd |v={\sqq}} \longrightarrow \rMHi$ (where the $[S_{j}], [\E_j]$ above are replaced by $[S'_{j}], [\E'_j]$).

The automorphism $\tTT_i \in \Aut (\tUi)$ in \eqref{eq:defT} factors through an automorphism $\TT_i \in \Aut(\Ui_\bvsd)$.
Then the diagram \eqref{eq:Up=T}  induces a commutative diagram below.

\begin{proposition}
   \label{prop:braidDist}
Let $(Q, \btau)$ be a Dynkin $\imath$quiver; see \eqref{diagram of A}, \eqref{diagram of D}, and \eqref{diagram of E} in cases when $\btau\neq \Id$. Then $\tTT_i: \tUi_{\bvsd} \rightarrow \tUi_{\bvsd}$ in \eqref{eq:defT} induces an isomorphism
$\TT_i: \Ui_{\bvsd} \rightarrow \Ui_{\bvsd}$, for each $i\in\ci$.
Moreover, for any sink $i$ in $Q_0$, we have the following commutative diagram of isomorphisms:
\[
\xymatrix{ \Ui_{\bvsd |v={\sqq}}\ar[r]^{\TT_i}  \ar[d]^{\psi_Q} & \Ui_{\bvsd |v={\sqq}}\ar[d]^{\psi_{Q'}} \\
\rMHd \ar[r]^{\bar{\Gamma}_i} & \rMHi }
\]
\end{proposition}

\begin{proof}
The first assertion follows by the same arguments as for Proposition \ref{prop:derived invariant of QSP}.
The second assertion follows by a direct computation.
\end{proof}

Note that $\TT_i: \Ui_{\bvsd} \longrightarrow \Ui_{\bvsd}$ satisfies that $\TT_i(k_\alpha)=k_{\bs_i\alpha}$ for any $\alpha\in\Z^\I$.

\subsection{Braid group action revisited}

Now we consider the braid group symmetries on $\Ui_{\bvsd}$.

\begin{corollary}
  \label{cor:braidUi1}
Let $(Q,\btau)$ be a Dynkin $\imath$quiver.
There exists a homomorphism
$\brW \rightarrow \aut( \Ui_{\bvsd})$ such that $\bs_i\mapsto \TT_i$ for all $i\in \ci $.
\end{corollary}

\begin{proof}
Follows from Theorem \ref{thm:bgtui} and Proposition~\ref{prop:braidDist}.
\end{proof}

\begin{remark}
  \label{rem:KP11}
A braid group action of $\brW$ on $\Q(\sqrt{v})\otimes_{\Q(v)} \Ui_{\bvsd}$ was defined earlier in \cite{KP11} (by computer computations); their operators are denoted below by $\TT_i'$, for $i\in \ci$. We will not recall the explicit formulas for $\TT_i'$ here.  By a direct computation, we show that there exists an automorphism $ \phi$ on $\Q(\sqrt{v})\otimes_{\Q(v)} \Ui_{\bvsd}$ determined by
\begin{align*}
\phi(B_j) &= \left\{
\begin{array}{ll}
v^{1/2} B_j, & \text{ for } j\in \ci, j\neq \btau j,
\\
-v^{-1/2}B_j, & \text{ for } j \not\in \ci,
\\
B_j, & \text{ for }j=\btau j;
\end{array}
\right.
&&
\phi(k_i) =k_i, \quad \forall i\in \ci.
\end{align*}
Another direct computation shows that our operator $\TT_i$ is related to theirs by $\TT_i=\phi \TT_i'\phi^{-1}$.
\end{remark}

\begin{remark}
While the braid relations at the level of $\tUi$ imply the braid relations at the level of $\Ui_{\bvsd}$, the argument can not be reversed. The action of $\tTT_i$ on $\tUi$ moves central elements around, while fixing the scalars (and hence the chosen parameters $\bvs$).
\end{remark}

\begin{remark}
For any acyclic quiver $Q$, let $Q^{\dbl} =Q\sqcup Q^\diamond$,  where $Q^\diamond$ is an identical copy of a quiver $Q$. Consider the $\imath$quiver $(Q^{\dbl},{\rm swap})$ as in \cite[\S8]{LW19} and its $\imath$quiver algebra $\Lambda$. Then $\cs\cd\widetilde{\ch}(\Lambda)$ is isomorphic to the semi-derived Hall algebra defined in \cite{Gor13} and also to the (non-reduced) Bridgeland's Hall algebra \cite{Br13} (which provides a realization of the Drinfeld double quantum group $\tU$). In this setting, a reduced version of the isomorphism (given by reflection functors) in Proposition~ \ref{prop:reflection} (see Proposition~ \ref{prop:derived invariant of QSP}) has been obtained in \cite[Theorem 9.27]{Gor13}. 
When combined with the results in \cite{SV99, XY01}, this provides a reflection functor realization of the braid group action on a quantum group $\U$.
From the viewpoint of this paper, it is expected that the reflection functors gives rise to a braid group action on the Drinfeld double quantum group $\tU$, which descends to the one on the reduced Drinfeld double $\U$. 
\end{remark}

The braid group symmetries on $\Ui$ for a general parameter $\bvs$ will be treated in Section~\ref{sec:PBW} below.

\section{PBW bases for $\imath$quantum groups}
  \label{sec:PBW}

Various PBW bases of quantum groups defined by Lusztig \cite{Lus90a} can be realized using Hall algebras; see \cite{Rin96}. In this section, for Dynkin $\imath$quivers we shall use reflection functors to construct PBW bases for the $\imath$Hall algebras and then construct PBW bases for $\imath$quantum groups. We also write down formulas for the braid group operators on $\Ui$ associated to a general parameter $\bvs$.

\subsection{$\imath$-Admissible sequences}

Let $(Q, \btau)$ be a Dynkin $\imath$quiver. For a sink $\ell\in Q_0$, define
\[
\bs_{\ell}(Q, \btau)=(Q',\btau'),
\]
where $(Q',\btau')$ is the $\imath$quiver defined in \S\ref{subsec:APR}. 
Recall from \eqref{eq:ci} that $\ci$ is a subset of $\I$ which consists of fixed representatives for $\btau$-orbits in $\I$. Note by definition that the $\btau$-orbits coincide with $\btau'$-orbits in $\I$.
Up to relabeling, we can set $\ci=\{1,\dots,r\}\subseteq\I=\{1,\dots,n\}$, where $r$ denotes the number of $\btau$-orbits in $\I$.

Recall from \eqref{def:simple reflection} that $\bs_{i}$ is defined by $\bs_i=s_{i}$ if $\btau i=i$ and $\bs_i=s_is_{\btau i}$ if $\btau i\neq i$.
\begin{definition}
  \label{def:admissible}
A sequence $i_1,\dots,i_t$ in $\ci$ is called {\em $\imath$-admissible} if $i_r$ is a sink in $\bs_{i_{r-1}}\cdots \bs_{i_1}(Q)$, for each $1\leq r\leq t$.
\end{definition}

In particular, if $\btau=\Id$, $\imath$-admissible sequence coincides with the $(+)$-admissible sequence; see, e.g., \cite[\S11.5]{DDPW}.

Let $\Phi$ be the root system of $\fg$, and let $\Phi^+$ be the set of positive roots with simple roots $\alpha_1,\dots,\alpha_n$. Clearly, $\btau$ induces an involution of $\Phi$ (and also of $\Phi^+$) by mapping $\sum_{i\in\I} a_i\alpha_i$ to $\sum_{i\in \I}a_i\alpha_{\btau i}$ for $1\le i\le n$.

Recall from \S\ref{sub:generic} that $M_q(\beta)$ is the indecomposable $\K Q$-module such that $\dimv M_q(\beta)=\beta$ for each $\beta\in\Phi^+$. A sequence or an ordering $\beta_1,\dots,\beta_N$ of the positive roots in $\Phi^+$ is called to be {\em $Q$-admissible} if $\Ext^1_{\K Q}(M_q(\beta_r),M_q(\beta_t))\neq 0$ implies $r>t$ (or $\Hom_{\K Q}(M_q(\beta_r),M_q(\beta_t))\neq0$ only if $r\leq t$). Note that this definition does not depend on the base field $\K=\F_q$. A $Q$-admissible sequence exists, see, e.g., \cite[Corollary 3.34]{DDPW}.

For any $\imath$-admissible sequence $i_1,\dots,i_t$, we define
\begin{align}
\label{def:betaj}
\beta_{j}:=\bs_{i_1}\cdots \bs_{i_{j-1}}(\alpha_{i_j}),
\end{align}
for any $1\le j\le t$.
Note that $\beta_{1}=\alpha_{i_1}$, $\btau(\beta_{1})=\alpha_{\btau i_1}$. 
The following lemma is immediate from the definition of $W_\btau$ in \eqref{eq:Wtau} and definition of $\beta_j$ in \eqref{def:betaj}.

\begin{lemma}
We have $\btau(\beta_j)=\bs_{i_1}\cdots \bs_{i_{j-1}}(\alpha_{\btau i_j})$ for any $1\le j\le t$. In particular, we have $\beta_{j}=\btau(\beta_{j})$ if $\btau i_j=i_j$.
\end{lemma}


Note the longest element $w_0$ in $W$ lies in $W_\btau$ by definition \eqref{eq:Wtau}. 
The following is an $\imath$quiver generalization of the existence of $(+)$-admissible sequence for Dynkin quivers (cf. \cite[\S1.4]{DDPW}).

\begin{theorem}
  \label{thm:i-seq}
Let $(Q,\btau)$ be a Dynkin $\imath$quiver. Then there is an $\imath$-admissible sequence of length $N_\imath$, $\{i_1,\dots,i_{\Ni}\}$, such that
\[
\beta_{1}, \btau(\beta_1), \beta_{2}, \btau(\beta_2), \dots,\beta_{\Ni },\btau(\beta_{\Ni})
\]
is a $Q$-admissible sequence of the roots in $\Phi^+$. (By convention the redundant $\btau(\beta_j)$ is omitted here and below whenever $\btau(\beta_j)=\beta_j$.) Moreover, $\bs_{i_1}\cdots \bs_{i_{\Ni}}$ is a reduced expression of $w_0$ in $W_\btau$ (and it becomes a reduced expression of $w_0 \in W$ if each $\bs_{i_k}$ is replaced by products of simple reflections as in \eqref{def:simple reflection}). In particular, $\Ni$ is the length of $w_0$ in $W_\btau$.
\end{theorem}
The $\imath$-admissible sequence in Theorem~\ref{thm:i-seq} is called {\em complete}.

\begin{proof}
Let $\gamma_1,\dots,\gamma_N$ be a $Q$-admissible ordering of the roots in $\Phi^+$. We define inductively a sequence $\beta_1, \beta_2, \ldots \in \Phi^+$, which is a subsequence of $\{\gamma_i \mid 1\le i \le N\}$, as follows. Set $\beta_1=\gamma_1$. Given $i\ge 2$, assume $\beta_1, \ldots, \beta_{i-1}$ have been fixed. We define $\beta_{i}$ to be $\gamma_k$ for a mimimal $k$ which does not belong to $\{\beta_{j},\btau(\beta_{ j})\mid 1\le j\le i-1 \}$. This process will terminate, and we obtain a longest sequence of length, say, $N_\imath$, $\beta_1,\dots,\beta_{\Ni}$.  (We do not assume that $N_\imath$ is the length of $w_0$ in $W_\btau$ to start with.)

We claim that $\beta_{1},\btau(\beta_1), ,\dots,\beta_{\Ni },\btau(\beta_{\Ni})$ (with redundancy removed) is a $Q$-admissible ordering of the roots in $\Phi^+$.

Indeed, by construction, $\Hom_{\K Q}(M_q(\beta_i),M_q(\beta_j))\neq0$ only if $i\leq j$. Then
$$
\Hom_{\K Q}(M_q(\btau \beta_i),M_q(\btau \beta_j))\neq0 \; \text{
only if } i\leq j.
$$
Assume that $\Hom_{\K Q}(M_q(\beta_i),M_q(\btau \beta_j))\neq0$. We denote $\beta_i=\gamma_{t_i}$, $\btau(\beta_i)=\gamma_{t_l}$, $\beta_j=\gamma_{t_j}$ and $\btau(\beta_j)=\gamma_{t_{k}}$. Then $t_i\le t_l$, $t_j\le t_k$, $t_i\leq t_k$. Since $\btau$ is an automorphism of $\mod(\K Q)$, $\Hom_{\K Q}(M_q(\btau \beta_i),M_q(\beta_j))\neq0$, which implies that $t_l\le t_j$.
So $t_i\leq t_j$, and then $i\leq j$.
Similarly, one can check that $\Hom_{\K Q}(M_q(\btau\beta_i),M_q(\beta_j))\neq0$ only if $i\leq j$. The claim has been proved.

From now on, we shall take $\{\gamma_1,\dots,\gamma_N\}$ to be the $Q$-admissible sequence constructed from $\{\beta_{1},\btau(\beta_1), \dots,\beta_{\Ni },\btau(\beta_{\Ni})\}$ by deleting $\btau(\beta_j)$ whenever $\btau(\beta_j)=\beta_j$. It follows from \cite[Lemma 11.28 \& \S1.4]{DDPW} that there exists a $(+)$-admissible sequence $i_1,\dots,i_N$ such that
\begin{equation}
  \label{eq:gj}
  \gamma_j=s_{i_1}\dots s_{i_{j-1}}(\alpha_{i_j}),
\end{equation}
for all $1\leq j\leq N$. There exists a unique $1\le t_j \le N$ such that
\begin{equation}
  \label{eq:tj}
\beta_j=\gamma_{t_j}
\end{equation}
for any $1\le j\le \Ni$. Note that $\btau(\beta_j)=\gamma_{t_j+1}$ if $\btau(\beta_j)\neq \beta_j$.

We shall prove that $\{i_{t_1}, i_{t_2}, \dots, i_{t_{\Ni}}\}$ forms the desired $\imath$-admissible sequence, i.e., for $1\le k \leq N_\imath$,
\begin{equation}
  \label{eq:i-adm}
\beta_{k}=\bs_{i_{t_1}} \bs_{i_{t_2}} \ldots \bs_{i_{t_{k-1}}}(\alpha_{i_{t_k}}).
\end{equation}
 More precisely, we shall prove that, for $1\le k \leq N_\imath$,
\begin{enumerate}
\item[($1_k$)]
  $\bs_{i_{t_1}} \bs_{i_{t_2}} \ldots \bs_{i_{t_{k-1}}}=s_{i_1} s_{i_2} \ldots s_{i_{t_k-1}}$;
\item[($2_k$)]
if $\btau (\beta_k)\neq \beta_k$, then $\btau(\beta_{k})=\bs_{i_{t_1}} \bs_{i_{t_2}} \ldots \bs_{i_{t_{k-1}}}(\alpha_{i_{t_k+1}})$ and $i_{t_k+1} = \btau i_{t_k}$.
\end{enumerate}
Note that

$(\star)$ \quad Equation~\eqref{eq:i-adm} follows directly by $(1_k)$ and \eqref{eq:gj}--\eqref{eq:tj}.


Let us prove $(1_k)$--$(2_k)$ by induction on $1\le k \leq N_\imath$.
First, note that $\beta_1=\gamma_1=\alpha_{i_1}$. If $\btau(\beta_1)\neq \beta_1$, then $\alpha_{\btau i_1}=\btau(\beta_1)=\gamma_2=s_{i_1}(\alpha_{i_2})$, that is, $s_{i_1}(\alpha_{i_2})$ is a simple root. This could only happen when $s_{i_1}(\alpha_{i_2}) = \alpha_{i_2}$, and so $i_2 =\btau i_1 \neq i_1$. Note that $t_1=1$. So we have verified $(2_k)$ with $k=1$; and $(1_k)$ with $k=1$ is vacuous.

Assume the validity of $(1_k)$--$(2_k)$ for $1\le k<\Ni$. we shall prove the statements $(1_{k+1})$--$(2_{k+1})$ by separating into 2 cases.

Case (1): \underline{$\btau(\beta_{k})=\beta_{k}$}. Then $t_{k+1}=t_k+1$ by \eqref{eq:tj}. By applying \eqref{eq:i-adm} (which is valid by $(\star)$ thanks to the inductive assumption $(1_k)$) to  $\btau(\beta_{k})=\beta_{k}$, we obtain $\btau i_{t_k} =i_{t_k}$, and so $\bs_{i_{t_k}}=s_{i_{t_k}}$. Hence it follows by the inductive assumption $(1_k)$ that
\[
\bs_{i_{t_1}} \bs_{i_{t_2}} \ldots \bs_{i_{t_{k-1}}} \cdot \bs_{i_{t_k}} =s_{i_1} s_{i_2} \ldots s_{i_{t_k-1}} \cdot s_{i_{t_k}}=s_{i_1} s_{i_2} \ldots s_{i_{t_{k+1}-1}},
\]
whence $(1_{k+1})$.

Case (2): \underline{$\btau(\beta_{k}) \neq \beta_{k}$}. Then by using \eqref{eq:tj} and reading off part of the $Q$-admissible sequence: $\beta_{k} =\gamma_{t_k}, \btau(\beta_{k}) =\gamma_{t_k+1}, \beta_{k+1} =\gamma_{t_{k+1}} =\gamma_{t_k+2}, \ldots$, we obtain
\begin{equation}
  \label{eq:tt}
t_{k+1}=t_k+2.
\end{equation}
Note \eqref{eq:i-adm} is valid by $(\star)$ thanks to the inductive assumption $(1_k)$.
Using \eqref{eq:tj}--\eqref{eq:gj}--\eqref{eq:i-adm} and $(1_k)$, we have
\begin{align*}
s_{i_1} s_{i_2} \ldots s_{i_{t_k}} (\alpha_{i_{t_k+1}})
&= \gamma_{t_k+1} =\btau(\beta_{k})
\\
&= \btau ( \bs_{i_{t_1}} \bs_{i_{t_2}} \ldots \bs_{i_{t_{k-1}}}(\alpha_{i_{t_k}}) )
\\
&= \bs_{i_{t_1}} \bs_{i_{t_2}} \ldots \bs_{i_{t_{k-1}}}(\alpha_{\btau i_{t_k}}) )
\\
&= s_{i_1} s_{i_2} \ldots s_{i_{t_k-1}} (\alpha_{\btau i_{t_k}}).
\end{align*}
Hence $s_{i_{t_k}} (\alpha_{i_{t_k+1}}) =\alpha_{\btau i_{t_k}}$, which implies that $i_{t_k+1} = \btau i_{t_k} \neq i_{t_k}$ by an argument similar to Case $k=1$ above.

Thus we have $\bs_{i_{t_k}}=s_{i_{t_k}}s_{i_{t_k+1}}$. It follows  by the inductive assumption $(1_k)$ and \eqref{eq:tt} that
\[
\bs_{i_{t_1}} \bs_{i_{t_2}} \ldots \bs_{i_{t_{k-1}}} \cdot \bs_{i_{t_k}}
= s_{i_1} s_{i_2} \ldots s_{i_{t_k-1}} \cdot s_{i_{t_k}}s_{i_{t_k+1}}
= s_{i_1} s_{i_2} \ldots s_{i_{t_{k+1}-1}},
\]
whence $(1_{k+1})$.

Assume $\btau(\beta_{k+1})\neq \beta_{k+1}$. We shall establish $(2_{k+1})$.
Thanks to $(1_{k+1})$ and then $(\star)$ (with $k$ replaced by $k+1$), we have $\beta_{k+1}=\bs_{i_{t_1}} \bs_{i_{t_2}} \ldots \bs_{i_{t_{k}}}(\alpha_{i_{t_{k+1}}})$, and thus
\[
 \btau(\beta_{k+1})=\bs_{i_{t_1}} \bs_{i_{t_2}} \ldots \bs_{i_{t_k}}(\alpha_{\btau i_{t_{k+1}}}).
\]
On the other hand, by \eqref{eq:gj} and $(1_{k+1})$ we have
\[
\btau(\beta_{k+1})=\gamma_{t_{k+1}+1}
= s_{i_1} s_{i_2} \ldots s_{i_{t_{k+1}-1}} s_{i_{t_{k+1}}} (\alpha_{i_{t_{k+1}+1}})=\bs_{i_{t_1}} \bs_{i_{t_2}} \ldots  \bs_{i_{t_k}} \cdot s_{i_{t_{k+1}}}(\alpha_{i_{t_{k+1}+1}}).
\]
A comparison of the two formulas for $\btau(\beta_{k+1})$ gives us $\alpha_{\btau i_{t_{k+1}}}=s_{i_{t_{k+1}}}(\alpha_{i_{t_{k+1}+1}})$; this implies that ${i_{t_{k+1}+1}} = \btau i_{t_{k+1}} \neq  i_{t_{k+1}}$ by an argument similar to Case $k=1$ above, whence $(2_{k+1})$.

This finishes the inductive proof of $(1_k)$--$(2_k)$, and hence \eqref{eq:i-adm} for all $k$; that is,
 $i_{t_1},i_{t_2},\dots,i_{t_{\Ni}}$ forms the desired $\imath$-admissible sequence.

We obtain by Equation $(2_{\Ni})$ that $\bs_{i_{t_{\Ni}}} =s_{i_{t_{\Ni}}} s_{i_{t_{\Ni}+1}}$ if $\btau(\beta_{\Ni}) \neq \beta_{\Ni}$ and $\bs_{i_{t_{\Ni}}} =s_{i_{t_{\Ni}}}$ otherwise. It then follows by $(1_{\Ni})$--$(2_{\Ni})$ that
\[
\bs_{i_{t_1}} \bs_{i_{t_2}} \ldots \bs_{i_{t_{\Ni}}}=s_{i_1} s_{i_2} \ldots s_{i_{N}}.
\]
Since $s_{i_1} s_{i_2} \ldots s_{i_N}$ is a reduced expression of $w_0 \in W$ (\cite[Lemma 11.28]{DDPW}), we conclude that $\bs_{i_{t_1}} \bs_{i_{t_2}} \ldots \bs_{i_{t_{\Ni}}}$ is a reduced expression of $w_0 \in W_\btau$. In particular, $N_\imath$ is indeed the length of $w_0$ in $W_\btau$.
\end{proof}

\begin{remark}
Assume $\btau=\id$. For any $Q$-admissible ordering of the roots $\beta_{1},\dots,\beta_{N}$ in $\Phi^+$, there is an $\imath$-admissible sequence $i_1,\dots,i_N$ such that $\beta_j=\bs_{i_1}\cdots \bs_{i_{j-1}}(\alpha_{i_j})$
for all $1\le j\le N$; cf., e.g., \cite[Lemma 11.28]{DDPW}.
\end{remark}

\begin{example}
Let $Q=(\xymatrix{ 1\ar[r] &2 & 3\ar[l] })$, with $\btau\neq \Id$.
Let $\ci=\{1,2\}$. Then $\{i_1=2,i_2=1,i_3=2,i_4=1\}$ is a complete $\imath$-admissbile
sequence of $Q$. Indeed, $\beta_{1}=\btau(\beta_{1})=\alpha_2$; $\beta_{2}=\bs_2(\alpha_1)=\alpha_1+\alpha_2$, $\btau(\beta_{2})=\bs_2(\alpha_3)=\alpha_2+\alpha_3$; $\beta_3=\btau(\beta_3)=\bs_2\bs_1(\alpha_2)=\alpha_1+\alpha_2+\alpha_3$; $\beta_4=\bs_2\bs_1\bs_2(\alpha_1)=\alpha_3$, $\btau(\beta_4)=\bs_2\bs_1\bs_2(\alpha_3)=\alpha_1$.
In particular, $\bs_2\bs_1\bs_2\bs_1=s_2s_1s_3s_2s_1s_3$ is the longest element $w_0$ in $S_4$ (or the longest element in $W_\btau =W(B_2)$).
\end{example}

\subsection{Changes of parameters}

Below we write $\Ui=\Ui_{\bvs}$ to indicate its dependence on a parameter $\bvs$.
It is well known that the $\Q(v)$-algebras $\Ui_{\bvs}$ (up to some field extension) are isomorphic for different choices of  parameters $\bvs$  \cite{Let02}; see Lemma~\ref{lem:base change} below. We shall use the index ${\bvs_{\diamond}}$ to indicate the relevant algebras, e.g., $\Ui_{{\bvsd}}$  in distinguished parameters $\bvsd=(\vs_{\diamond,i})_{i\in \I}$ in \eqref{eq:special split}--\eqref{eq:specialE}.

Consider a field extension of $\Q(v)$
\begin{align}
\label{def:ai}
{\F}= \Q(v)(a_i\mid i\in I),
\qquad
\text{ where }a_i=\sqrt{\frac{\vs_{\diamond,i}}{\vs_i}} \quad (i\in \I).
\end{align}
Denote by
\begin{align}
\label{def:basechange}
{}_{\F}\Ui_{\bvs} =\F \otimes_{\Q(v)} \Ui_{\bvs}
\end{align}
 the $\F$-algebra obtained by a base change. By a direct computation a rescaling automorphism on ${}_{\F}\U$ induces an isomorphism in the lemma below.

\begin{lemma}
\label{lem:base change}
There exists an isomorphism of ${\mathbb F}$-algebras
\begin{align*}
\phi_{\bf u}: {}_{\F}\Ui_{\bvs_{\diamond}} & \longrightarrow {}_{\F}\Ui
\\
B_i \mapsto a_iB_i,\qquad k_i & \mapsto k_i, \quad (\forall i\in \I).
\end{align*}
\end{lemma}


Recall from \eqref{eqn:specializing3} that $\rMHg$ denotes the reduced generic $\imath$Hall algebra (associated to a general parameter $\bvs$).
Denote by $\rMHg_{\bvs_{\diamond}}$ the reduced generic $\imath$Hall algebra in the distinguished parameter $\bvs_{\diamond}$. Recall from \eqref{eq:Phi2} that $\Phi^\imath= \Phi^0\cup \Phi^+$, where $\Phi^0=\{\upgamma_i\mid i\in\I\}$.

\begin{lemma}
  \label{lem: base change i-hall}
There exists an isomorphism of ${\mathbb F}$-algebras
\begin{align}
\phi_{\bf h}: {}_{\F}\rMHg_{\bvs_{\diamond}} & \longrightarrow {}_{\F}\rMHg
\notag
\\
\fu_{\lambda}&\mapsto  \prod_{i\in\I}a_i^{\lambda(\upgamma_i) +\lambda(\upgamma_{\btau i})+ \sum_{\beta\in\Phi^+}\lambda(\beta) d_{i}(\beta) }   \cdot \fu_{\lambda}, \quad \forall \lambda\in\widetilde{\fp}^\imath,
  \label{eq:rescale-u}
\end{align}
where we denote $\beta=\sum_{i\in\I} d_{i}(\beta) \alpha_i$.
In particular, $\phi_{\bf h}(\fu_{\beta}) = \prod_{i\in\I}a_i^{ d_{i}(\beta)} \cdot \fu_{\beta}$. 
\end{lemma}

\begin{proof}
The proof is entirely similar to the proof of \cite[Proposition 8.7]{LW19}, and will be skipped.
\end{proof}

The isomorphism $\phi_{\bf h}$ in Lemma \ref{lem: base change i-hall} is designed to be compatible with $\phi_{\bf u}$ in Lemma \ref{lem:base change}, and so we have the following commutative diagram
\begin{equation}
  \label{eq:dag1}
\xymatrix{ {}_{\F}\Ui_{\bvs_{\diamond}}\ar[rr]^{\phi_{\bf u}} \ar[d]^{\psi_{\bvs_\diamond}}
&& {}_{\F}\Ui  \ar[d]^{\psi_{\bvs}} &\ar@{}[d]^{} \\
{}_{\F}\rMHg_{\bvs_{\diamond}} \ar[rr]^{\phi_{\bf h}}& &{}_{\F}\rMHg & }
\end{equation}
where $\psi_\bvsd$ and $\psi_{\bvs}$ are the isomorphisms in \eqref{eq:psi2}.

\subsection{Braid group action on $\Ui$}

As before the notation $\Ui$ stands for the $\imath$quantum group with general parameters $\bvs$. We now define a braid group action $\TT_i$ on ${}_\F\Ui$  from the $\TT_{\diamond,i}$ on $\Ui_{\bvsd}$ (note $\TT_{\diamond,i}$ was simply denoted by $\TT_i$ in Corollary~\ref{cor:braidUi1}) via a conjugation by the isomorphism $\phi_{\bf u}$:
\begin{equation}
  \label{eq:TTTT}
\TT_i = \phi_{\bf u} \TT_{\diamond,i} \phi_{\bf u}^{-1}.
\end{equation}

\begin{theorem}
  \label{thm:braidUigeneral}
  Let $(Q, \btau)$ be a Dynkin $\imath$quiver.
Then there is an automorphism $\TT_i$ on ${}_\F\Ui$, for each $i\in \ci$. Moreover there exists a homomorphism
$\brW \rightarrow \aut({}_\F\Ui)$ such that $t_i\mapsto \TT_i$ in \eqref{eq:TTTT}, for all $i\in \ci$.
\end{theorem}
The actions of $\TT_i$ on generators of $\Ui$ can be made explicit; see Lemmas~\ref{lem:braidsplit}--\ref{lem:braidA} below.

\subsubsection{Split case}

\begin{lemma}
  \label{lem:braidsplit}
Let $\Ui$ be a $\imath$quantum group with $\btau =\Id$. Then there exists a unique algebra automorphism $\TT_i$ of ${}_\F\Ui$ such that
\[
\TT_i(B_j)=\left\{
\begin{array}{ll}
{} \frac{1}{\sqrt{-v^2\vs_i}}(B_jB_i-vB_iB_j), & \text{ if }c_{ij}=-1, \\
B_j,  &\text{ if }j=i \text{ or }c_{ij}=0.
\end{array}
\right.
\]
Moreover, there exists a homomorphism ${\rm Br}(W)\rightarrow \aut(\Ui)$ such that $t_i\mapsto \TT_i$, for all $i\in \I$.
\end{lemma}

\begin{proof}
First consider the case with the distinguished parameter $\bvsd =(\vs_{\diamond,i})_{i\in \I}$; in this case the base change is not needed as $\F=\Q(v)$.
By Corollary~\ref{cor:braidUi1}, there is an automorphism, denoted here by $\TT_{\diamond,i}$ to emphasize $\bvsd$, in $\Aut(\Ui_{\bvsd})$ which descends from $\tTT_i \in \Aut(\tUi)$. Via the relation between $\tUi_{\bvsd}$ and $\Ui_{\bvsd}$ in Proposition~\ref{prop:QSP12}, the formulas in this lemma (where we set $\sqrt{-v^2\vs_{\diamond,i}}=1$ for all $i$ thanks to \eqref{eq:bvsd}) follow directly from the formulas in Lemma~\ref{lemma:braid group of split involution}.

The formulas for $\TT_i$ on $\Ui$ for a general parameter $\bvs$ follows from $\TT_{\diamond,i}$ and \eqref{eq:TTTT} by a direct computation. The assertion on braid group relations follows by the conjugation \eqref{eq:TTTT} and Corollary~\ref{cor:braidUi1}(2).
\end{proof}

\subsubsection{Quasi-split cases}

For $\alpha =\sum_{i\in \I} m_i \alpha_i \in \Z^{\I}$, it is convenient to denote
\begin{equation}
  \label{eq:ka}
k_\alpha = \prod_{i\in \I} k_{\alpha_i}^{m_i} \in \Ui,
\end{equation}
where we set
\[
k_{\alpha_i}=\begin{cases}
1 & \text{ if } i=\btau i, \\
k_i & \text{ if } i \neq \btau i, i\in \I_\btau, \\
k_i^{-1} & \text{ if } i \neq \btau i, i \not \in \I_\btau.
\end{cases}
\]

\begin{lemma}
  \label{lem:braidA}
Let $\Delta$ be a Dynkin graph of type $A_{2r+1}$ as in (\ref{diagram of A}), and $\btau$ be its nontrivial involution. 
Then there exists a unique algebra automorphism $\TT_i$ ($0\leq i\leq r$) of ${}_\F\Ui$  such that $\TT_i(k_\alpha) =k_{\bs_i \alpha}$ for $\alpha \in \Z^{\I_\btau}$,  and for $1\leq i\leq r$,
\begin{align*}
\TT_i(B_j) &=\left\{ \begin{array}{ll}
\frac{1}{\sqrt{\vs_i}}B_jB_i -\frac{v}{\sqrt{\vs_i}}B_iB_j,  & \text{ if }c_{ij}=-1 \text{ and } c_{\btau i,j}=0,\\
\frac{1}{\sqrt{\vs_{\btau i}}}B_{\btau i}B_j-\frac{1}{v\sqrt{\vs_{\btau i}}}B_jB_{\btau i},  & \text{ if } c_{ij}=0 \text{ and }c_{\btau i,j}=-1 ,\\
-{\frac{1}{v\vs_i}[[B_j,B_i]_v,B_{\btau i}]_v}+ B_jk_{i}, & \text{ if }i=\pm 1, j=0, \\
-k_{i}^{-1}B_{\btau i},  & \text{ if }j=i,\\
-v^2k_i B_i,  &\text{ if } j=\btau i,\\
B_j,  & \text{otherwise};
 \end{array}\right.
  \\ %
\TT_0(B_j) &=\left\{ \begin{array}{ll} 
\frac{1}{\sqrt{-v^2\vs_0}}(B_jB_0 -vB_0B_j) & \text{ if } j=\pm1,\\
B_j & \text{ otherwise.}  \end{array}\right.
\end{align*}
Moreover there exists a homomorphism
${\rm Br}(B_{r+1})\rightarrow \aut ( \bU^\imath)$ such that $t_i\mapsto \TT_i$, for all $i\in \ci $.
\end{lemma}

\begin{proof}
The same argument as for Lemma~\ref{lem:braidsplit} works here.
%
\end{proof}

Formulas for the actions of $\TT_i$ can be similarly written down in quasi-split type $D_n$ and $E_6$; we skip the details (which can be found in the arXiv v1 of the paper).

\subsection{PBW bases for $\Ui$}

Recall from \S\ref{sub:generic} the generic $\imath$Hall algebra $\rMHg$. Let $\T(Q, \btau)$ be the quantum torus subalgebra of $\rMHg$ generated by $\fu_\nu$, $\nu\in\fp^0$ (or equivalently, by $\fu_{\upgamma_i}$, $\upgamma_i\in\Phi^0$). By definition, $\T(Q, \btau)$ is isomorphic to the $\Q(v)$-group algebra of the root lattice $\Z^\I$.

\begin{lemma}
$\rMHg$ is free as a right (respectively, left) $\T(Q, \btau)$-module, with a basis given by
$\fu_\lambda$, $\lambda\in\fp$.
\end{lemma}

\begin{proof}
Follows by \cite[Proposition~4.9, Theorem~9.8]{LW19} and the definition of the reduced $\imath$Hall algebra.
\end{proof}

\begin{lemma}
[\text{\cite[Theorem 5.8]{LW19}}]
\label{lem:generic PBW-Hall basis}
Given any ordering $\gamma_1,\dots,\gamma_N$ of the roots in $\Phi^+$, the set
\[
\{\fu_{\gamma_1}^{{*} \lambda_1}*\cdots*\fu_{\gamma_N}^{*\lambda_N} \mid   \lambda_1, \ldots, \lambda_N \in \N \}
\]
is a $\Q(v)$-basis of $\rMHg$ as a right $\T(Q, \btau)$-module.
\end{lemma}
This basis is called a {\em PBW basis }for the generic Hall algebra of the $\imath$quiver $(Q, \btau)$.

Let $i_1,\dots,i_{\Ni}$ be a complete $\imath$-admissible sequence. Let $\Ui:=\Ui_{\vs}$ be an $\imath$quantum group. Recalling $a_i$ from \eqref{def:ai}, for any $\beta_j=\sum_{i\in\I} d_{ji}\alpha_i\in\Phi^+$, we define the {\em $q$-root vectors}
\begin{align}
\label{def: B beta}
B_{\beta_{j}}:=a_{i_j} \prod_{i\in\I}a_i^{-d_{ji}} \cdot \TT_{i_1}^{-1}\cdots \TT_{i_{j-1}}^{-1}(B_{i_j}),\quad
B_{\btau \beta_{j}}:=a_{i_j} \prod_{i\in\I}a_i^{-d_{ji}} \cdot \TT_{i_1}^{-1}\cdots \TT_{i_{j-1}}^{-1} (B_{\btau i_j})
\end{align}
in ${}_{\F}\Ui$ for $1\le j\le \Ni$. Note that $B_{\beta_j}=B_{\btau \beta_j}$ if and only $\btau (i_j) =i_j$.

Let $\psi:\Ui\rightarrow\rMHg$ be the isomorphism obtained in Theorem~\ref{generic U MRH}. Recall $\ci$ is given in \eqref{eqn:representative}.

\begin{lemma}
\label{lem: pbw lemmma}
(1) We have  $B_{\beta_j}\in\Ui$, for all $j$.

(2) The isomorphism $\psi: \Ui\rightarrow  \rMHg$ in \eqref{eq:psi} sends
\begin{align*}
&\psi(B_{\beta_{j}})=\left\{ \begin{array}{cc}\frac{-1}{v^2-1}\fu_{\beta_j},\text{ if }i_j\in\ci, \\
\frac{\bv}{v^2-1}\fu_{\beta_j},
\text{ otherwise},
\end{array}\right.
\end{align*}
for $1\leq j \leq N$.
\end{lemma}

\begin{proof}
In this proof we add the index $\bvsd$ (or simply $\diamond$) to indicate that the case with the distinguished parameter $\bvsd$ is under consideration, e.g., $\psi_{\bvsd}: \Ui_{{\bvsd}}\longrightarrow \rMHg_{{\bvsd}}$, $\TT_{\diamond,i}$, $\bar{\Gamma}_{\diamond,i}$, $B_{\diamond,\beta}$. By convention, no index $\bvs$ is used in the case with a general parameter $\bvs$.

Recall that $i_1,\dots,i_{\Ni}$ is a complete $\imath$-admissible sequence.  Denote by $(Q^j,\btau^j)$ the $\imath$quiver of the $\imath$quiver algebra $\iLa_j:= \bs_{i_{j-1}}\cdots \bs_{i_1}(\iLa)$, for $1\le j \le N_\imath+1$, and by $\psi_{\bvsd}^j: \Ui_{\bvsd} \stackrel{\sim}\longrightarrow \ch_{red}(Q^j,\btau^j)$ the isomorphism defined similarly to $\psi_{\bvsd}$. It follows from Corollary~ \ref{lem:reflecting dimen} and the comparisons of the dimension vectors  that
\[
F_{i_j-1}^+\cdots F_{i_1}^+(M_q(\beta_j))\cong S_{i_j}
\]
where $S_{i_j}$ is the simple $\iLa_j$-module with dimension vector $\alpha_{i_j}$.  So by \eqref{eqn:reflection 1}, for any $j$,
\[
\bar{\Gamma}_{\diamond, i_j-1}\cdots\bar{\Gamma}_{\diamond, i_1} (\fu_{\beta_j})=\fu_{\alpha_{i_j}}.
\]

Recall that $\psi^j_{\bvsd}(B_{i_j})=\frac{-1}{v^2-1}\fu_{\alpha_{i_j}}$, and that
$\psi_{\bvsd}^j\circ\TT_{\diamond, i_j}=\bar{\Gamma}_{\diamond, i_j}\circ\psi_{\bvsd}^{j}$ by the diagram \eqref{eq:Up=T}. So we have
\begin{align*}
B_{i_j}&=\frac{-1}{v^2-1}(\psi_{\bvsd}^j)^{-1}(\fu_{\alpha_{i_j}})
=\frac{-1}{v^2-1}(\psi_{\bvsd}^j)^{-1}\bar{\Gamma}_{\diamond, i_j-1}\cdots\bar{\Gamma}_{\diamond, i_1} (\fu_{\beta_j})\\
&=\frac{-1}{v^2-1}\TT_{\diamond,i_j-1}\cdots \TT_{\diamond,i_1}(\psi_{\bvsd}^{-1}(\fu_{\beta_j})).
\end{align*}
Therefore, we have $\psi_{\bvsd}(B_{\diamond,\beta_{j}})=\frac{-1}{v^2-1}\fu_{\beta_j}$.

Recall $\phi_{\bf u}: {}_{\F}\Ui_{{\bvsd}}  \longrightarrow {}_{\F}\Ui$ in Lemma \ref{lem:base change}. Recall from \eqref{eq:TTTT} that $\TT_{i}=\phi_{\bf u} \TT_{\diamond,i}\phi_{\bf u}^{-1}$.
Then
\begin{align}
\label{eqn:phi on Beta}
\phi_{\bf u}(B_{\diamond,\beta_j})= \TT_{i_1}^{-1}\cdots \TT_{i_{j-1}}^{-1}\phi_{\bf u}(B_{i_j})= a_{i_j}\TT_{i_1}^{-1}\cdots \TT_{i_{j-1}}^{-1}(B_{i_j}).
 \end{align}

From the commutative diagram \eqref{eq:dag1} and Lemma~\ref{lem: base change i-hall} on $\phi_{\bf h}$, it follows that
\begin{align*}
\psi \phi_{\bf u}(B_{\diamond,\beta_j})=&\phi_{\bf h}\psi_{{\bvsd}}(B_{\diamond,\beta_j})=\frac{-1}{v^2-1}\phi_{\bf h}(\fu_{\beta_j})
=\frac{-1}{v^2-1} \prod_{i\in\I}a_i^{d_{ji}}\cdot \fu_{\beta_j}.
\end{align*}
By combining this with \eqref{def: B beta}--\eqref{eqn:phi on Beta}, we obtain that
$\psi(B_{\beta_j})= \frac{-1}{v^2-1}\fu_{\beta_j}$.
Clearly, we have $\fu_{\beta_j}\in \rMHg$ and thus $B_{\beta_j}\in \Ui$, by using the isomorphism $\psi: \Ui \stackrel{\sim}\longrightarrow\rMHg$.
\end{proof}

We give a PBW basis for the $\imath$quantum group $\Ui$ (with arbitrary parameter $\bvs$) below.

\begin{theorem}  [PBW bases]
  \label{thm:PBWUi}
For any ordering $\gamma_1,\ldots,\gamma_N$ of the roots in $\Phi^+$, the set
\[
\{ B_{\gamma_1}^{\lambda_1}\cdots B_{\gamma_N}^{\lambda_N}\mid \lambda\in \fp \}
\]
provides a basis of $\Ui$ as a right $\U^{\imath 0}$-module, where $\lambda_i=\lambda(\gamma_i)$ for $1\le i\le N$.
\end{theorem}

\begin{proof}
From Lemma \ref{lem: pbw lemmma}, for any $\lambda\in\fp$, we have
$$\psi (B_{\gamma_1}^{\lambda_1}\cdots B_{\gamma_N}^{\lambda_N})=d_\lambda \fu_{\gamma_1}^{*\lambda_1}\cdots \fu_{\gamma_N}^{*\lambda_N},$$
for some nonzero scalar $d_\lambda \in \Q(v)$. Then the desired result follows from Lemma \ref{lem:generic PBW-Hall basis}.
\end{proof}
Similar PBW bases for $\tUi$ can be defined, and then alternative PBW bases for $\Ui$ can be obtained from those for $\tUi$ by a central reduction.

\begin{remark}
For the $\imath$quiver of diagonal type, $Q^{\dbl}=Q\bigsqcup Q'$, the above basis reproduces a PBW basis for $\U$ as a module over $\U^0$ by \cite[Proposition~8.6]{LW19} (which is a reformulation of the main theorem of Bridgeland \cite{Br13}); compare with \cite{Rin96}.
\end{remark}


\end{document}